\documentclass[secthm,seceqn,amsthm,ussrhead,reqno,12pt]{amsart}
\usepackage[utf8]{inputenc}
\usepackage[english]{babel}
\usepackage[symbol]{footmisc}
\usepackage{amssymb,amsmath,amsthm,amsfonts,xcolor,enumerate,hyperref, comment,longtable, cleveref}

\usepackage{times}
\usepackage{cite}
\usepackage{pdflscape}
\usepackage{ulem}
\usepackage[mathcal]{euscript}
\usepackage{tikz}
\usepackage{hyperref}
\usepackage{cancel}
\usepackage{stmaryrd}

\usepackage{amsfonts}
\usepackage{amssymb}
\usepackage{times}
\usepackage{xcolor}
\usepackage{tkz-graph}
\usepackage{url}
\usepackage{float}
\usepackage{tasks}

\usepackage{cite}
\usepackage{hyperref}

\usepackage{amsfonts}
\usepackage{amsmath}
\usepackage{eurosym}
\usepackage{geometry}

\newcommand{\blue}{\color{blue}}
\newcommand{\black}{\color{black}}

\usepackage{caption,booktabs}

\theoremstyle{plain}
\newtheorem{theorem}{Theorem}
\newtheorem{lemma}[theorem]{Lemma}

\newtheorem{remark}[theorem]{Remark}

\newtheorem{definition}[theorem]{Definition}
\newtheorem{corollary}[theorem]{Corollary}
\begin{document}
 
 \bigskip

\noindent{\Large
Degenerations of Poisson-type algebras}\footnote{
The 
work is supported by the Spanish Government through the Ministry of Universities grant `Margarita Salas', funded by the European Union - NextGenerationEU; FCT   UIDB/MAT/00212/2020,  UIDP/MAT/00212/2020 and 2022.02474.PTDC.
}


\begin{center}

 {\bf
 Hani Abdelwahab\footnote{Department of Mathematics, 
 Faculty of Science, Mansoura University,  Mansoura, Egypt; \ haniamar1985@gmail.com},
Amir Fernández Ouaridi\footnote{Department of Mathematics, University of Cádiz, Puerto Real, Espa\~na; \ amir.fernandez.ouaridi@gmail.com}
\& 
Ivan Kaygorodov\footnote{CMA-UBI, University of  Beira Interior, Covilh\~{a}, Portugal; 
\    kaygorodov.ivan@gmail.com}

}

\bigskip

\end{center}

\noindent {\bf Abstract:}
{\it  
The degenerations of Poisson-type algebras are studied in the following varieties in dimension two:
Leibniz--Poisson algebras, 
transposed Leibniz--Poisson algebras, 
Novikov--Poisson algebras,
commutative pre-Lie algebras,
anti-pre-Lie Poisson algebras and 
pre-Poisson algebras. 
For these varieties, the algebraic and geometric classifications are given. Also, the complete graph of degenerations is obtained, together with the description of the orbit closures of each of its algebras and parametric families up to isomorphism.
 }



{\small 	 \tableofcontents
}

  

\section*{ Introduction}

\bigskip 

The algebraic classification (up to isomorphism) of algebras of dimension $n$ of a certain variety
defined by a family of polynomial identities is a classic problem in the theory of non-associative algebras, see \cite{ack,afm,afk21,ak21, ccsmv, ckls,erik,Ernst,fkkv,kkl20,kkp20,kv16,japan,petersson,  shaf,abcf2, abcflts, abcfp}.
There are many results related to the algebraic classification of small-dimensional algebras in different varieties of
associative and non-associative algebras.
For example, algebraic classifications of 
$2$-dimensional algebras \cite{petersson},
$3$-dimensional evolution algebras \cite{ccsmv},
$3$-dimensional anticommutative algebras \cite{japan},
$4$-dimensional division algebras \cite{Ernst,erik},
$6$-dimensional anticommutative nilpotent algebras \cite{kkl20}
and 
$8$-dimensional dual Mock Lie algebras  \cite{ckls} have been given.
 Geometric properties of a variety of algebras defined by a family of polynomial identities have been an object of study since the 1970's, see \cite{wolf2, wolf1,  ak21,afk21,   chouhy,     BC99, aleis, aleis2,   gabriel,    ckls, cibils,   GRH, GRH2, ale3, fkkv, ack,  ikp20, ikv17,        kv16, S90}  and references in \cite{k23,l24}. 
 Gabriel described the irreducible components of the variety of $4$-dimensional unital associative algebras~\cite{gabriel}.  
 Cibils considered rigid associative algebras with $2$-step nilpotent radical \cite{cibils}.
 Chouhy  proved that  in the case of finite-dimensional associative algebras,
 the $N$-Koszul property is preserved under the degeneration relation~\cite{chouhy}.
 Burde and Steinhoff constructed the graphs of degenerations for the varieties of    $3$-dimensional and $4$-dimensional Lie algebras~\cite{BC99}. 
 Grunewald and O'Halloran studied the degenerations for the variety of $5$-dimensional nilpotent Lie algebras~\cite{GRH}. 
 Alvarez and Hernández described the degenerations for the variety of $5$-dimensional nilpotent Lie superalgebras \cite{aleis2}.
  Alvarez and Vera  described the degenerations for the variety of $3$-dimensional ${\rm Hom}$-Lie algebras \cite{aleis}.
 Fernández Ouaridi,  Kaygorodov,  Khrypchenko and    Volkov described all the degenerations in the variety of $4$-dimensional commutative nilpotent algebras. 
Ignatyev,  Kaygorodov, and Popov determined the dimensions of the irreducible components of 
$n$-dimensional $2$-step nilpotent algebras (commutative and anticommutative)  \cite{ikp20}.
On the other hand, Kaygorodov, Khrypchenko, and  Lopes studied the dimensions of the irreducible components of $n$-dimensional (all, commutative, anticommutative) nilpotent algebras \cite{kkl21}. 
Degenerations have also been used to study the level of complexity of an algebra~\cite{wolf1,wolf2}.
The notion of degenerations of binary algebras was also extended 
to degenerations of $n$-ary algebras \cite{kv20} and degenerations of algebras with two multiplications \cite{abcf2, abcflts, abcfp, afm}. 
The study of degenerations of algebras is very rich and closely related to deformation theory, in the sense of Gerstenhaber \cite{ger63}.

Poisson algebras arose from the study of Poisson geometry in the 1970s and have appeared in an extremely wide range of areas in mathematics and physics, such as Poisson manifolds, algebraic geometry, operads, quantization theory, quantum groups, and classical and quantum mechanics. The study of Poisson algebras also led to other algebraic structures, such as 
noncommutative Poisson algebras, 
generic Poisson algebras,
transposed Poisson algebras,
Novikov--Poisson algebras,
algebras of Jordan brackets and generalized Poisson algebras,
$F$-manifold algebras,
quasi-Poisson algebras,
Gerstenhaber algebras,
Poisson bialgebras,
double Poisson algebras,
Poisson $n$-Lie algebras, etc. All these classes are related in some way to Poisson algebras, for that reason we call them Poisson-type algebras.


\medskip

The present paper is devoted to the algebraic, geometric, and degeneration  classification of the complex algebras from the principal varieties of Poisson-type algebras in dimension two. In particular, we consider the varieties of Leibniz--Poisson algebras, 
transposed Leibniz--Poisson algebras, 
Novikov--Poisson algebras,
commutative pre-Lie algebras,
anti-Pre-Poisson algebras and 
pre-Poisson algebras.

\newpage
\section{Preliminaries: the algebraic and geometric  classifications}

All the algebras below will be over $\mathbb C$ and all the linear maps will be $\mathbb C$-linear.
For simplicity, every time we write the multiplication table of an algebra 
the products of basic elements whose values are zero or can be recovered from the commutativity  or from the anticommutativity are omitted. In this section, introduce the techniques used to obtain our main results.

\subsection{The algebraic classification of algebras}
In this paper, we work with various  types of algebras with two multiplications such that the  ``first" multiplication is associative and commutative (or no) and the ``second" multiplication depends on the particular type of algebras under consideration.
Let us review the method we will use to obtain the algebraic classification  for an arbitrary variety of Poisson-type algebras
(the present method, in the case of Poisson algebras, is given with more details in \cite{afm}).

\begin{definition} 
A Poisson-type algebra  is a  vector space $({\rm A}, \cdot, \bullet)$ equipped with 
one commutative associative multiplication $\cdot$ and another multiplication $\bullet$ satisfying a family 
of polynomial identities $\Omega^{\{\bullet\}}_i(x_1, \ldots, x_{n_i}),$
which includes only one multiplication $\bullet.$
These two operations are required to satisfy  a family of  Leibniz type  identities
$\Theta^{ \{\cdot, \bullet\}}_i(x_1, \ldots, x_{k_i}),$
which includes two of these  multiplications $\cdot$ and $\bullet.$
Main examples of Poisson-type algebras are the following:
Poisson algebras, 
Leibniz--Poisson algebras,
transposed Poisson algebras,
Novikov--Poisson algebras, 
commutative pre-Lie algebras or  
anti-pre-Lie Poisson algebras.
\end{definition}

\begin{definition}
Let $\left( {\rm A},\cdot \right) $ be a commutative associative
algebra. Define ${\rm Z}^{2}\left( {\rm A},{\rm A}\right) $ to be the
set of all bilinear maps 
$\theta :{\rm A}\times {\rm A}\longrightarrow {\rm A}$ such that:%
\begin{center}
$\Omega^{\{\bullet \to \theta\}}_i(x_1, \ldots, x_{n_i})=0$ and 
$\Theta^{ \{\cdot, \bullet \to \theta\}}_i(x_1, \ldots, x_{k_i})=0,$
\end{center}
where, by $\bullet \to \theta$ we denote the changing of the multiplication $\bullet$ to the cocycle $\theta.$ 
Then ${\rm Z}^{2}\left({\rm A},{\rm A}\right)\neq \mathcal{\emptyset }$ since $\theta=0\in {\rm Z}^{2}\left({\rm A},{\rm A}\right)$.
\end{definition}

Now, for $\theta \in {\rm Z}^{2}\left( {\rm A},{\rm A}\right)$, let us define a multiplication 
$\bullet_{\theta }$ on ${\rm A}$ by $x\bullet_{\theta}y=\theta \left( x,y\right) $ for all $x,y$ in ${\rm A}$. 
Then $\left( {\rm A},\cdot ,\bullet_{\theta }\right) $ is a Poisson-type
algebra satisfying the two  families of identities
$\Omega^{\{\bullet_\theta\}}_i(x_1, \ldots, x_{n_i})$ and 
$\Theta^{ \{\cdot, \bullet_\theta\}}_i(x_1, \ldots, x_{k_i})$. 
Conversely, if $\left( {\rm A},\cdot ,\bullet \right) $ is a Poisson-type algebra satisfying the two  families of identities
$\Omega^{\{\bullet\}}_i(x_1, \ldots, x_{n_i})$ and 
$\Theta^{ \{\cdot, \bullet\}}_i(x_1, \ldots, x_{k_i})$, 
then there exists $\theta \in {\rm Z}^{2}\left( 
{\rm A},{\rm A}\right) $ such that $\left( 
{\rm A},\cdot ,\bullet_{\theta }\right) \cong\left( {\rm A},\cdot ,\bullet
\right)$. To see this, consider the bilinear map $\theta :{\rm A}\times {\rm A}\longrightarrow {\rm A}$ defined by $
\theta \left( x,y\right) = x\bullet y$ for all $x,y$ in ${\rm A}$. Then $\theta \in {\rm Z}^{2}\left( 
{\rm A},{\rm A}\right) $ and $\left( 
{\rm A},\cdot ,\bullet_\theta \right) =\left( {\rm A},\cdot ,\bullet
\right)$.

Let $\left( {\rm A},\cdot \right) $ be a commutative associative
algebra. The automorphism group, ${\rm Aut}\left( {\rm A}\right) $,
of ${\rm A}$ acts on ${\rm Z}^{2}\left( {\rm A},{\rm A}\right) $ by 
\begin{longtable}{rcl}
$\left(\theta *\phi\right)  \left( x,y\right) $&$=$&$\phi ^{-1}\left( \theta \left( \phi \left(
x\right) ,\phi \left( y\right) \right) \right),$
\end{longtable}
for $\phi \in {\rm Aut}%
\left( {\rm A}\right) $, and $\theta \in {\rm Z}^{2}\left( {\rm A}, {\rm A}\right) $.

\begin{lemma}
\label{isom}Let $\left( {\rm A},\cdot \right) $ be a commutative associative algebra
and $\theta ,\vartheta \in {\rm Z}^{2}\left( {\rm A},{\rm A}\right) $.
Then $\left( {\rm A},\cdot ,\bullet_{\theta }\right) $ and $%
\left( \rm{A},\cdot ,\bullet_{\vartheta }\right) $ are
isomorphic if and only if there is a $\phi \in {\rm Aut}\left( {\rm A}
\right) $ with $\theta *\phi =\vartheta $.
\end{lemma}

\begin{proof}
If $\theta * \phi =\vartheta $, then $\phi :\left( \rm{A},\cdot ,\bullet_{\vartheta }\right) \longrightarrow $ $\left( \rm{A},\cdot
,\bullet_{\theta }\right) $ is an isomorphism since 
$\phi
\left( \vartheta \left( x,y\right) \right) =\theta \left( \phi \left(
x\right) ,\phi \left( y\right) \right) $. On the other hand, if $\phi
:\left( \rm{A},\cdot ,\bullet_{\vartheta }\right)
\longrightarrow $ $\left( \rm{A},\cdot ,\bullet _{\theta
}\right) $ is an isomorphism of Poisson-type algebras, then $\phi \in {\rm Aut}%
\left( \rm{A}\right) $ and $\phi \left(  x\bullet  _{\vartheta
}y \right) = \phi \left( x\right) \bullet  _{\theta } \phi \left( y\right)  
 $. Hence \begin{longtable}{lclcl}
    $\vartheta \left( x,y\right) $&$=$&$\phi ^{-1}\left( \theta
\left( \phi \left( x\right) ,\phi \left( y\right) \right) \right) $&$=$&$\left(\theta *
\phi \right)\left( x,y\right) $
 \end{longtable} and therefore $\theta * \phi=\vartheta $.
\end{proof}

Consequently, we have a procedure to classify the Poisson-type algebras with given
associated commutative associative algebra $\left( \rm{A},\cdot \right) 
$. It consists of three steps:

\begin{enumerate}
\item Compute ${\rm Z}^{2}\left( \rm{A},\rm{A}\right) $.

\item Find the orbits of ${\rm Aut}\left( \rm{A}\right) $ on $%
{\rm Z}^{2}\left( \rm{A},\rm{A}\right) $.
\item  Choose a representative $\theta$ from each orbit and then construct the Poisson-type algebra $\left( \rm{A},\cdot,\bullet  _{\theta }\right) $.
\end{enumerate}

\subsubsection{The algebraic classification of $2$-dimensional commutative associative algebras}
All algebras under our consideration are constructed on the base of $2$-dimensional commutative associative algebras. 
Hence,  the classification of them, which can be obtained by direct verification from \cite{kv16}, is given below.

\begin{theorem}
\label{asocc2}
Let $\left( {\rm A}, \cdot \right) $ be a  $2$-dimensional
commutative associative algebra. Then ${\rm A}$ is
isomorphic to one of the following algebras:
\begin{longtable}{lclcllcll} 

${\rm A}_{01}$&$:$&$e_1 \cdot e_1 $&$=$&$ e_1,$&$ e_2 \cdot e_2 $&$=$&$ e_2$\\ 

${\rm A}_{02}$&$:$& $ e_1 \cdot e_1  $&$=$&$ e_1,$&$ e_1 \cdot e_2 $&$ =$&$e_2$\\  

${\rm A}_{03}$&$:$ & $e_1 \cdot e_1 $&$= $&$e_1$\\ 
 
${\rm A}_{04}$&$:$ & $ e_1 \cdot e_1 $&$ =$&$ e_2$
 
 \end{longtable}

\end{theorem}

\begin{lemma}
\label{asocc2aut}
The description of the group of automorphisms of every $2$-dimensional commutative associative algebra is given.

\begin{enumerate}
    \item If $\phi \in {\rm Aut}({\rm A}_{01})$, then $\phi\in \mathbb S_2,$ i.e. $\phi(e_1)=e_1, \  \phi(e_2)=e_2$ or
    $\phi(e_1)=e_2, \  \phi(e_2)=e_1.$
    
    \item If $\phi \in{\rm Aut}({\rm A}_{02})$, then $\phi(e_1) = e_1$ and $\phi(e_2)= \xi e_2$, for $\xi\in \mathbb{C}^{*}$.
    
    \item If $\phi \in{\rm Aut}({\rm A}_{03})$, then $\phi(e_1) = e_1$ and $\phi(e_2)= \xi e_2$, for 
    $\xi\in \mathbb{C}^{*}$. 
    
    \item If $\phi \in{\rm Aut}({\rm A}_{04})$, then $\phi(e_1) = \xi e_1 + \nu e_2$ and $\phi(e_2)= \xi^2 e_2$, for $\xi\in \mathbb{C}^{*}$ and $\nu\in \mathbb{C}$. 
\end{enumerate}

\end{lemma}

 \subsection{Degenerations and the  geometric classification of algebras}
Let us introduce the techniques used to obtain the geometric classification for an arbitrary variety of Poisson-type algebras. Given a complex vector space ${\mathbb V}$ of dimension $n$, the set of bilinear maps \begin{center}$\textrm{Bil}({\mathbb V} \times {\mathbb V}, {\mathbb V}) \cong \textrm{Hom}({\mathbb V} ^{\otimes2}, {\mathbb V})\cong ({\mathbb V}^*)^{\otimes2} \otimes {\mathbb V}$
\end{center} is a vector space of dimension $n^3$. The set of pairs of bilinear maps 
\begin{center}$\textrm{Bil}({\mathbb V} \times {\mathbb V}, {\mathbb V}) \oplus \textrm{Bil}({\mathbb V}\times {\mathbb V}, {\mathbb V}) \cong ({\mathbb V}^*)^{\otimes2} \otimes {\mathbb V} \oplus ({\mathbb V}^*)^{\otimes2} \otimes {\mathbb V}$ \end{center} which is a vector space of dimension $2n^3$. This vector space has the structure of the affine space $\mathbb{C}^{2n^3}$ in the following sense:
fixed a basis $e_1, \ldots, e_n$ of ${\mathbb V}$, then any pair with multiplication $(\mu, \mu')$, is determined by some parameters $c_{ij}^k, c_{ij}'^k \in \mathbb{C}$,  called {structural constants},  such that
$$\mu(e_i, e_j) = \sum_{p=1}^n c_{ij}^k e_k \textrm{ and } \mu'(e_i, e_j) = \sum_{p=1}^n c_{ij}'^k e_k$$
which corresponds to a point in the affine space $\mathbb{C}^{2n^3}$. Then a set of bilinear pairs $\mathcal S$ corresponds to an algebraic variety, i.e., a Zariski closed set, if there are some polynomial equations in variables $c_{ij}^k, c_{ij}'^k$ with zero locus equal to the set of structural constants of the bilinear pairs in $\mathcal S$. Given the identities defining a particular class of Poisson-type algebras, we can obtain a set of polynomial equations in variables $c_{ij}^k, c_{ij}'^k$. This class of $n$-dimensional Poisson-type algebras is a variety. Denote it by $\mathcal{T}_{n}$.
Now, consider the following action of $\rm{GL}({\mathbb V})$ on ${\mathcal T}_{n}$:
$$(g*(\mu, \mu'))(x,y) := (g \mu (g^{-1} x, g^{-1} y), g \mu' (g^{-1} x, g^{-1} y))$$
for $g\in\rm{GL}({\mathbb V})$, $(\mu, \mu')\in \mathcal{T}_{n}$ and for any $x, y \in {\mathbb V}$. Observe that the $\textrm{GL}({\mathbb V})$-orbit of $(\mu, \mu')$, denoted $O((\mu, \mu'))$, contains all the structural constants of the bilinear pairs isomorphic to the Poisson-type algebra with structural constants $(\mu, \mu')$.

A geometric classification of a variety of algebras consists on describing the irreducible components of the variety. Recall that any affine variety can be represented as a finite union of its irreducible components uniquely.
Note that describing the irreducible components of  ${\mathcal{T}_{n}}$ gives us the rigid algebras of the variety, which are those bilinear pairs with an open $\textrm{GL}(\mathbb V)$-orbit. This is due to the  fact that a bilinear pair is rigid in a variety if and only if the closure of its orbit is an irreducible component of the variety. 
For this reason, the following notion is convenient. Denote by $\overline{O((\mu, \mu'))}$ the closure of the orbit of $(\mu, \mu')\in{\mathcal{T}_{n}}$.

\begin{definition}
\rm Let ${\rm T} $ and ${\rm T}'$ be two $n$-dimensional Poisson-type algebras of a fixed class corresponding to the variety $\mathcal{T}_{n}$ and $(\mu, \mu'), (\lambda,\lambda') \in \mathcal{T}_{n}$ be their representatives in the affine space, respectively. The algebra ${\rm T}$ is said to {degenerate}  to ${\rm T}'$, and we write ${\rm T} \to {\rm T} '$, if $(\lambda,\lambda')\in\overline{O((\mu, \mu'))}$. If ${\rm T}  \not\cong {\rm T}'$, then we call it a  {proper degeneration}.
Conversely, if $(\lambda,\lambda')\not\in\overline{O((\mu, \mu'))}$ then  we say that ${{\rm T} }$ does not degenerate to ${{\rm T} }'$
and we write ${{\rm T} }\not\to {{\rm T} }'$.
\end{definition}

Furthermore, for a parametric family of algebras we have the following notion.

\begin{definition}
\rm
Let ${{\rm T} }(*)=\{{{\rm T} }(\alpha): {\alpha\in I}\}$ be a family of $n$-dimensional Poisson-type algebras of a fixed class corresponding to ${\mathcal{T} }_n$ and let ${{\rm T} }'$ be another algebra. Suppose that ${{\rm T} }(\alpha)$ is represented by the structure $(\mu(\alpha),\mu'(\alpha))\in{\mathcal{T} }_n$ for $\alpha\in I$ and ${{\rm T} }'$ is represented by the structure $(\lambda, \lambda')\in{\mathcal{T} }_n$. We say that the family ${{\rm T} }(*)$ {degenerates}   to ${{\rm T} }'$, and write ${{\rm T} }(*)\to {{\rm T} }'$, if $(\lambda,\lambda')\in\overline{\{O((\mu(\alpha),\mu'(\alpha)))\}_{\alpha\in I}}$.
Conversely, if $(\lambda,\lambda')\not\in\overline{\{O((\mu(\alpha),\mu'(\alpha)))\}_{\alpha\in I}}$ then we call it a  {non-degeneration}, and we write ${{\rm T} }(*)\not\to {{\rm T} }'$.

\end{definition}

Observe that ${\rm T}'$ corresponds to an irreducible component of $\mathcal{T}_n$ (more precisely, $\overline{{\rm T}'}$ is an irreducible component) if and only if ${{\rm T} }\not\to {{\rm T} }'$ for any $n$-dimensional Poisson-type algebra ${\rm T}$ and ${{\rm T}(*) }\not\to {{\rm T} }'$ for any parametric family of $n$-dimensional Poisson-type algebras ${\rm T}(*)$. To prove a particular algebra corresponds to an irreducible compotent, we will use the next ideas.
Firstly, since $\mathrm{dim}\,O((\mu, \mu')) = n^2 - \mathrm{dim}\,\mathfrak{Der}({\rm T})$, then if $ {\rm T} \to  {\rm T} '$ and  ${\rm T} \not\cong  {\rm T} '$, we have that $\mathrm{dim}\,\mathfrak{Der}( {\rm T} )<\mathrm{dim}\,\mathfrak{Der}( {\rm T} ')$, where $\mathfrak{Der}( {\rm T} )$ denotes the Lie algebra of derivations of  ${\rm T} $. 
Secondly, to prove degenerations, let ${{\rm T} }$ and ${{\rm T} }'$ be two Poisson-type algebras represented by the structures $(\mu, \mu')$ and $(\lambda, \lambda')$ from ${{\mathcal T} }_n$, respectively. Let $c_{ij}^k, c_{ij}'^k$ be the structure constants of $(\lambda, \lambda')$ in a basis $e_1,\dots, e_n$ of ${\mathbb V}$. If there exist $n^2$ maps $a_i^j(t): \mathbb{C}^*\to \mathbb{C}$ such that $E_i(t)=\sum_{j=1}^na_i^j(t)e_j$ ($1\leq i \leq n$) form a basis of ${\mathbb V}$ for any $t\in\mathbb{C}^*$ and the structure constants $c_{ij}^k(t), c_{ij}'^k(t)$ of $(\mu, \mu')$ in the basis $E_1(t),\dots, E_n(t)$ satisfy $\lim\limits_{t\to 0}c_{ij}^k(t)=c_{ij}^k$ and $\lim\limits_{t\to 0}c_{ij}'^k(t)=c_{ij}'^k$, then ${{\rm T} }\to {{\rm T} }'$. In this case,  $E_1(t),\dots, E_n(t)$ is called a parametrized basis for ${{\rm T} }\to {{\rm T} }'$.
Thirdly, to prove non-degenerations we may use a remark that follows from this lemma, see \cite{afm}. 

\begin{lemma} 
Consider two Poisson-type algebras ${\rm T}$ and ${\rm T}'$. Suppose ${\rm T} \to {\rm T}'$. Let C be a Zariski closed in ${\mathcal T}_n$ that is stable by the action of the invertible upper (lower) triangular matrices. Then if there is a representation $(\mu, \mu')$ of ${\rm T}$ in C, then there is a representation $(\lambda, \lambda')$ of ${\rm T}'$ in C.
\end{lemma}

In order to apply this lemma, we will give the explicit definition of the appropriate stable Zariski closed $C$ in terms of the variables $c_{ij}^k, c_{ij}'^k$ in each case. For clarity, we assume by convention that $c_{ij}^k=0$ (resp. $c_{ij}'^k=0$) if $c_{ij}^k$ (resp. $c_{ij}'^k$) is not explicitly mentioned on the definition of $C$.

\begin{remark}
\label{redbil}

Moreover, let ${{\rm T} }$ and ${{\rm T} }'$ be two Poisson-type algebras represented by the structures $(\mu, \mu')$ and $(\lambda, \lambda')$ from ${{\rm T} }_n$. Suppose ${\rm T}\to{\rm T}'$. Then if $\mu, \mu', \lambda, \lambda'$ represents algebras ${\rm T}_{0}, {\rm T}_{1}, {\rm T}'_{0}, {\rm T}'_{1}$ in the affine space $\mathbb{C}^{n^3}$ of algebras with a single multiplication, respectively, we have ${\rm T}_{0}\to {\rm T}'_{0}$ and $ {\rm T}_{1}\to {\rm T}'_{1}$.
So for example, $(0, \mu)$ can not degenerate in $(\lambda, 0)$ unless $\lambda=0$. 

\end{remark}

Fourthly, to prove ${{\rm T} }(*)\to {{\rm T} }'$, suppose that ${{\rm T} }(\alpha)$ is represented by the structure $(\mu(\alpha),\mu'(\alpha))\in{\mathcal{T} }_n$ for $\alpha\in I$ and ${{\rm T} }'$ is represented by the structure $(\lambda, \lambda')\in{\mathcal{T} }_n$. Let $c_{ij}^k, c_{ij}'^k$ be the structure constants of $(\lambda, \lambda')$ in a basis  $e_1,\dots, e_n$ of ${\mathbb V}$. If there is a pair of maps $(f, (a_i^j))$, where $f:\mathbb{C}^*\to I$ and $a_i^j:\mathbb{C}^*\to \mathbb{C}$ are such that $E_i(t)=\sum_{j=1}^na_i^j(t)e_j$ ($1\le i\le n$) form a basis of ${\mathbb V}$ for any  $t\in\mathbb{C}^*$ and the structure constants $c_{ij}^k(t), c_{ij}'^k(t)$ of $(\mu\big(f(t)\big),\mu'\big(f(t)\big))$ in the basis $E_1(t),\dots, E_n(t)$ satisfy $\lim\limits_{t\to 0}c_{ij}^k(t)=c_{ij}^k$ and $\lim\limits_{t\to 0}c_{ij}'^k(t)=c_{ij}'^k$, then ${{\rm T} }(*)\to {{\rm T} }'$. In this case  $E_1(t),\dots, E_n(t)$ and $f(t)$ are called a parametrized basis and a parametrized index for ${{\rm T} }(*)\to {{\rm T} }'$, respectively.
Fithly, to prove ${{\rm T} }(*)\not \to {{\rm T} }'$, we may use an analogous of Remark \ref{redbil} for parametric families that follows from Lemma \ref{main2}.

\begin{lemma}\label{main2}
Consider the family of Poisson-type algebras ${\rm T}(*)$ and the Poisson-type algebra ${\rm T}'$. Suppose ${\rm T}(*) \to {\rm T}'$. Let C be a Zariski closed in $\mathcal{T}_n$ that is stable by the action of the invertible upper (lower) triangular matrices. Then if there is a representation $(\mu(\alpha), \mu'(\alpha))$ of ${\rm T}(\alpha)$ in C for every $\alpha\in I$, then there is a representation $(\lambda, \lambda')$ of ${\rm T}'$ in C.
\end{lemma}

Finally, the following remark simplifies the geometric problem.

\begin{remark}
 Let $(\mu, \mu')$ and $(\lambda, \lambda')$ represent two Poisson-type algebras. Suppose $(\lambda, 0)\not\in\overline{O((\mu, 0))}$, (resp. $(0, \lambda')\not\in\overline{O((0, \mu'))}$), then $(\lambda, \lambda')\not\in\overline{O((\mu, \mu'))}$.
  As we construct the classification of a given class of Poisson-type algebras from a certain class of algebras with a single multiplication which remains unchanged, this remark becomes very useful.  Namely, all the classes considered in this paper are constructed on a $2$-dimensional commutative associative algebra. The geometric classification of this class was given in \cite{kv16}.
\end{remark}

\section{ Leibniz--Poisson algebras}
The notion of  Leibniz--Poisson algebras is a non-anticommutative generalization of Poisson algebras and on the other hand, it is a 
particular case of 
algebras with bracket \cite{cp06}
and noncommutative Leibniz--Poisson algebras \cite{cd06}.
Other results related to Leibniz--Poisson algebras have appeared in \cite{R13}.

\subsection{The algebraic classification of $2$-dimensional Leibniz algebras}

\begin{definition}
An algebra  is called a Leibniz algebra if it satisfies the identity 
\begin{longtable}{rcl}
$[[ x,y] ,z] $&$=$&$[[ x,z],y] +[ x,[ y,z]].$
\end{longtable}%
\end{definition}

The algebraic classification of  Leibniz algebras was obtained by direct verification from \cite{kv16}.

\begin{theorem}
\label{leib2}
Let ${\rm L}$ be a nonzero $2$-dimensional Leibniz algebra. 
Then $L$  is
isomorphic to  one and only one  of the following algebras:
\begin{longtable}{lclcllcl}
${\rm L}_{1}$&$:$&$[ e_{1},e_{1}]$&$ =$&$e_{2}$\\

${\rm L}_{2}$&$:$&$[ e_{1},e_{2}] $&$=$&$e_{1}$\\

${\rm L}_{3}$&$:$&$[ e_{1},e_{2}] $&$=$&$ e_{2},$ &$ [e_{2},e_{1}] $&$= $&$-e_{2}$

\end{longtable}

\end{theorem}

\subsection{The algebraic classification of  $2$-dimensional Leibniz--Poisson algebras}

\begin{definition}
A Leibniz–Poisson  algebra is a vector space  equipped with 
a commutative associative multiplication $\cdot$
and     
a Leibniz multiplication $[-,-].$
These two operations are required to satisfy the following identity:
\begin{longtable}{rcl}
$[x\cdot y, z]$&$ =$&$ [x, z]\cdot y + x\cdot [y, z].$
\end{longtable}

\end{definition}

Note that the Leibniz–Poisson algebras $({\rm L}, \cdot, [-, -])$ with zero multiplication $ \cdot$  are the Leibniz algebras. The classification of  $2$-dimensional Leibniz algebras is given in Theorem \ref{leib2}.  
Hence, we are studying the Leibniz--Poisson algebras defined on every  commutative associative algebra from Theorem \ref{asocc2}. 
\begin{definition}
Let $\left( \rm{A},\cdot \right) $ be a commutative associative
algebra. Define ${\rm Z}_{\rm LP}^{2}\left( \rm{A},\rm{A}\right) $ to be the
set of all bilinear maps $\theta :\rm{A}\times \rm{A}%
\longrightarrow \rm{A}$ such that:%
\begin{longtable}{rcl}
$\theta \left( \theta \left( x,y\right) ,z\right)$ & $=$&$\theta \left( \theta
\left( x,z\right) ,y\right) +\theta \left( x,\theta \left( y,z\right)
\right),$ \\
$\theta \left( x\cdot y,z\right)$ & $=$&$\theta \left( x,z\right) \cdot y+x\cdot
\theta \left( y,z\right).$
\end{longtable}
If $\theta \in {\rm Z}_{\rm LP}^{2}\left( \rm{A},%
\rm{A}\right) $, then $\left( \rm{A},\cdot ,[ -,-]\right) $ is a Leibniz--Poisson algebra where $[ x,y]=\theta \left( x,y\right) $ for all $x,y\in \rm{A}$.
\end{definition}

\subsubsection{Leibniz--Poisson algebras defined on  ${\rm A}_{01}$}

Since ${\rm Z}_{\rm LP}^{2}({\rm A}_{01}, {\rm A}_{01}) = \left\{0\right\}$,  then there is only the trivial structure  ${\rm L}_{04} = {\rm A}_{01}$.
 
\subsubsection{Leibniz--Poisson algebras defined on  ${\rm A}_{02}$}

From the computation of ${\rm Z}_{\rm LP}^{2}({\rm A}_{02}, {\rm A}_{02})$, the Leibniz--Poisson structures defined on ${\rm A}_{02}$ are of the form:
$${\rm L}_{5}^{\alpha} : =\left\{ 
\begin{tabular}{lcllcl}
$e_1 \cdot e_1  $&$=$&$ e_1,$&$ e_1 \cdot e_2 $&$ =$&$e_2$ \\ 
$[ e_{2},e_{1}] $&$=$&$ \alpha_{1} e_{2}$%
\end{tabular}%
\right.$$
Since the action of ${\rm Aut}({\rm A}_{02})$ on this structure is the identity, we obtain there are no isomorphic algebras in this family.

\subsubsection{Leibniz--Poisson algebras defined on  ${\rm A}_{03}$}

The Leibniz--Poisson structures defined on ${\rm A}_{03}$ are:
$${\rm L}_{6}^{\alpha} : = \left\{ 
\begin{tabular}{lcl}
$e_1 \cdot e_1 $&$=$&$ e_1$ \\ 
$[ e_{2},e_{1}] $&$=$&$ \alpha e_{2}$%
\end{tabular}%
\right.$$
Again, the action of ${\rm Aut}({\rm A}_{03})$ on this structure is the identity, hence all algebras from this family are not isomorphic.

\subsubsection{Leibniz--Poisson algebras defined on  ${\rm A}_{04}$}

The Leibniz--Poisson structures defined on ${\rm A}_{04}$ are:
$${\rm L}_{7}^{\alpha} : = \left\{ 
\begin{tabular}{lcl}
$e_1 \cdot e_1  $&$=$&$ e_2$ \\
$[e_{1},e_{1}] $&$=$&$ \alpha e_{2}$%
\end{tabular}%
\right.$$

As in the previous cases of this section, the action of ${\rm Aut}({\rm A}_{04})$ on the structure above is the identity, hence all algebras from this family are not isomorphic.

\subsubsection{The algebraic classification of $2$-dimensional Leibniz--Poisson algebras}

\begin{theorem}
Let $\left( {\rm L},\cdot , [ -,-] \right) $ be a nonzero $2$-dimensional Leibniz--Poisson algebra. 
Then ${\rm L}$ is isomorphic to one Leibniz algebra listed in Theorem \ref{leib2} or  to one algebra listed below:

\begin{longtable}{lcllll}
 
${\rm L}_{4} $&$:$&$e_1 \cdot e_1 = e_1, e_2 \cdot e_2 = e_2$ \\

${\rm L}_{5}^{\alpha}$&$ :$&$ \left\{ 
\begin{tabular}{lcllcl}
$e_1 \cdot e_1  $&$=$&$ e_1,$&$ e_1 \cdot e_2  $&$=$&$e_2$ \\ 
$[ e_{2},e_{1}] $&$=$&$ \alpha e_{2}$%
\end{tabular}%
\right.$\\

 ${\rm L}_{6}^{\alpha}$&$ :$&$ \left\{ 
\begin{tabular}{lcl}
$e_1 \cdot e_1 $&$=$&$ e_1$ \\ 
$[ e_{2},e_{1}] $&$=$&$ \alpha e_{2}$%
\end{tabular}%
\right.$\\

 ${\rm L}_{7}^{\alpha} $&$:$&$ \left\{ 
\begin{tabular}{lcl}
$e_1 \cdot e_1  $&$=$&$ e_2$ \\ 
$[ e_{1},e_{1}] $&$=$&$ \alpha e_{2}$%
\end{tabular}%
\right.$

\end{longtable}

\end{theorem}

\subsubsection{The algebraic classification of $2$-dimensional generic Poisson algebras}

\begin{definition}
A generic Poisson  algebra (see, \cite{ksu18}) is a vector space  equipped with 
a commutative associative multiplication $\cdot$
and     
an anticommutative multiplication $[-,-].$
These two operations are required to satisfy the following identity:
\begin{longtable}{rcl}
$[x\cdot y, z]$&$ =$&$ [x, z]\cdot y + x\cdot [y, z].$
\end{longtable}
\end{definition}
 
 It is known \cite{kv16}, that each $2$-dimensional anticommutative algebra is a Lie algebra.
 Hence, each $2$-dimensional generic Poisson algebra is a Poisson algebra.
 
 \begin{corollary}
  Let $\left( {\rm G},\cdot , [ -,-] \right) $ be a nonzero $2$-dimensional generic Poisson algebra. 
Then ${\rm G}$ is isomorphic to one commutative associative  algebra listed in Theorem \ref{asocc2}
or to ${\rm L}_{3}.$

\end{corollary}

\subsection{Degenerations  of  $2$-dimensional Leibniz--Poisson algebras}

\begin{lemma} \label{th:degleib}
The graph of primary degenerations and non-degenerations of the variety of $2$-dimensional Leibniz--Poisson algebras is given in Figure  1, where the numbers on the right side are the dimensions of the corresponding orbits.

\end{lemma}
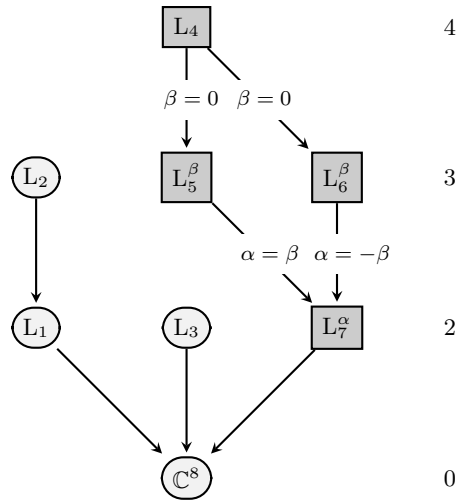
\begin{center}
	
	\begin{tikzpicture}[->,>=stealth,shorten >=0.05cm,auto,node distance=1.3cm,
	thick,
	main node/.style={rectangle,draw,fill=gray!10,rounded corners=1.5ex,font=\sffamily \scriptsize \bfseries },
	rigid node/.style={rectangle,draw,fill=black!20,rounded corners=0ex,font=\sffamily \scriptsize \bfseries }, 
	poisson node/.style={rectangle,draw,fill=black!20,rounded corners=0ex,font=\sffamily \scriptsize \bfseries },
	ac node/.style={rectangle,draw,fill=black!20,rounded corners=0ex,font=\sffamily \scriptsize \bfseries },
	lie node/.style={rectangle,draw,fill=black!20,rounded corners=0ex,font=\sffamily \scriptsize \bfseries },
	style={draw,font=\sffamily \scriptsize \bfseries }]

	\node (3) at (0.5,6) {$4$};
	\node (2) at (0.5,4) {$3$};
	\node (1) at (0.5,2) {$2$};
	\node (0)  at (0.5,0) {$0$};

    \node[main node] (c20) at (-3,0) {${\mathbb C^8}$};
  
    \node[main node] (c21) at (-5,2) {${\rm L}_{1}$};
    \node[main node] (c23) at (-3,2) {${\rm L}_{3}$};
    \node[ac node] (c27) at (-1,2) {${\rm L}_{7}^{\alpha}$};

    \node[main node] (c22) at (-5,4) {${\rm L}_{2}$};
    \node[ac node] (c25) at (-3,4) {${\rm L}_{5}^{\beta}$};
    \node[ac node] (c26) at (-1,4) {${\rm L}_{6}^{\beta}$};
    
    \node[ac node] (c24) at (-3,6) {${\rm L}_{4}$};
	\path[every node/.style={font=\sffamily\small}]

    
    
    (c24) edge node[above=0, right=-13, fill=white]{\tiny $\beta=0$}  (c25)
    (c24) edge node[above=0, right=-13, fill=white]{\tiny $\beta=0$}  (c26)
    (c25) edge node[above=0, right=-13, fill=white]{\tiny $\alpha=\beta$}  (c27)
    (c26) edge node[above=0, right=-13, fill=white]{\tiny $\alpha=-\beta$}  (c27)
    (c22) edge [bend left=0] node{}  (c21)
    
    (c21) edge [bend left=0] node{}  (c20)
    (c23) edge [bend left=0] node{}  (c20)
    (c27) edge [bend left=0] node{}  (c20);

	\end{tikzpicture}

{\tiny 
\begin{itemize}
\noindent Legend:
\begin{itemize}
    \item[--] Round nodes: Leibniz--Poisson algebras with trivial commutative associative multiplication $\cdot$.
    \item[--] Squared nodes: Leibniz--Poisson algebras with non-trivial commutative associative multiplication $\cdot$.

\end{itemize}
\end{itemize}}

{Figure 1.}  Graph of primary degenerations and non-degenerations.	
\end{center}
\bigskip

\bigskip

\begin{proof}
The dimensions of the orbits are deduced by computing the algebra of derivations.
The primary degenerations of the Leibniz--Poisson algebras with trivial commutative associative multiplication follow by \cite{kv16}, the others are proven using the parametric bases included in following Table:

    \begin{longtable}{|lcl|ll|}
\hline
\multicolumn{3}{|c|}{\textrm{Degeneration}}  & \multicolumn{2}{|c|}{\textrm{Parametrized basis}} \\
\hline
\hline


${\rm L}_{4} $ & $\to$ & ${\rm L}_{5}^{0}  $ & $
E_{1}(t)= e_1 + e_2$ & $
E_{2}(t)= te_2$  \\
\hline

${\rm L}_{4} $ & $\to$ & ${\rm L}_{6}^{0}  $ & $
E_{1}(t)= e_1$ & $
E_{2}(t)= te_2$  \\
\hline

${\rm L}_{5}^{\alpha} $ & $\to$ & ${\rm L}_{7}^{\alpha}  $ & $
E_{1}(t)= te_1 + t^{-1}e_2$ & $
E_{2}(t)= e_2$  \\
\hline

${\rm L}_{6}^{\alpha} $ & $\to$ & ${\rm L}_{7}^{-\alpha}  $ & $
E_{1}(t)= te_1 - t^{-1}e_2$ & $
E_{2}(t)= e_2$  \\
\hline 
    \end{longtable}

The primary non-degenerations of Leibniz--Poisson algebras with non-trivial commutative associative multiplication are proven using the sets from the following table:

    \begin{longtable}{|l|l|}
\hline
\multicolumn{1}{|c|}{\textrm{Non-degeneration}} & \multicolumn{1}{|c|}{\textrm{Arguments}}\\
\hline
\hline


$\begin{array}{cccc}
     {\rm L}_{4} &\not \to&  {\rm L}_{1}, 
     {\rm L}_{3}, 
     {\rm L}_{7}^{\alpha\neq0} & 
\end{array}$ &
${\mathcal R}= \left\{ \begin{array}{l}
c_{11}^{1}, c_{11}^{2}, c_{12}^{2}, c_{21}^{2}, c_{22}^{2} \in \mathbb{C}, 
c_{12}^{2} = c_{21}^{2}
\end{array} \right\}$\\
\hline

$\begin{array}{cccc}
     {\rm L}_{5}^{\alpha} &\not \to&  {\rm L}_{1}, {\rm L}_{3}, {\rm L}_{7}^{\beta\neq\alpha} & 
\end{array}$ &
${\mathcal R}= \left\{ \begin{array}{l}
c_{11}^{1}, c_{11}^{2}, c_{12}^{2}, c_{21}^{2}, c_{11}'^{2}, c_{21}'^{2} \in\mathbb{C},\\
c_{11}^{1}=c_{12}^{2},  c_{11}^{1}=c_{21}^{2},   c_{11}'^{2}=\alpha  c_{11}^{2},  c_{21}'^{2}=\alpha  c_{11}^{1}
\end{array} \right\}$\\
\hline

$\begin{array}{cccc}
     {\rm L}_{6}^{\alpha} &\not \to&  {\rm L}_{1}, {\rm L}_{3}, {\rm L}_{7}^{\beta\neq-\alpha} & 
\end{array}$ &
${\mathcal R}= \left\{ \begin{array}{l}
c_{11}^{1}, c_{11}^{2}, c_{11}'^{2}, c_{21}'^{2} \in \mathbb{C},
c_{11}'^{2}=-\alpha  c_{11}^{2},  c_{21}'^{2}=\alpha  c_{11}^{1}
\end{array} \right\}$\\
\hline

    \end{longtable}

\end{proof}

At this point, only the description of the closures of the orbits of the parametric families is missing. Although it is not necessary to study the closure of the orbits of each of the parametric families of
the variety of $2$-dimensional Leibniz--Poisson algebras in order to identify its irreducible components, we will study them to give a complete description of the variety.

\begin{lemma}\label{th:degleibf}
The description of the closure of the orbit of the parametric families in
the variety of $2$-dimensional Leibniz--Poisson algebras are given below:

\begin{longtable}{lcl}
$\overline{\{O({\rm L}_{5}^{*})\}}$ & ${\supseteq}$ &
$ \Big\{\overline{\{O({\rm L}_{7}^{\alpha})\}}, 
\overline{\{O({\rm L}_{2})\}}, 
\overline{\{O({\rm L}_{1})\}},
\overline{\{O({\mathbb C^8})\}} \Big\}$\\

$\overline{\{O({\rm L}_{6}^{*})\}}$ & ${\supseteq}$ &
$ \Big\{\overline{\{O({\rm L}_{7}^{\alpha})\}}, 
\overline{\{O({\rm L}_{2})\}}, 
\overline{\{O({\rm L}_{1})\}}, 
\overline{\{O({\mathbb C^8})\}} \Big\}$\\
    
$\overline{\{O({\rm L}_{7}^{*})\}}$ & ${\supseteq}$ &
$ \Big\{\overline{\{O({\rm L}_{1})\}},
\overline{\{O({\mathbb C^8})\}} \Big\}$\

\end{longtable}

\end{lemma}

\begin{proof}
Thanks to Theorem \ref{th:degleib} we have all necessarily degenerations.
The closure of the orbits of suitable families are given using the parametric bases and the parametric index included in following Table: 

    \begin{longtable}{|lcl|ll|}
\hline
\multicolumn{3}{|c|}{\textrm{Degeneration}}  & \multicolumn{2}{|c|}{\textrm{Parametrized basis}} \\
\hline
\hline

${\rm L}_{5}^{t^{-1}} $ & $\to$ & ${\rm L}_{2}  $ & $
E_{1}(t)= e_2$ & $
E_{2}(t)= te_1$  \\
\hline

${\rm L}_{6}^{t^{-1}} $ & $\to$ & ${\rm L}_{2}  $ & $
E_{1}(t)= e_2$ & $
E_{2}(t)= te_1$  \\
\hline

${\rm L}_{7}^{t^{-1}} $ & $\to$ & ${\rm L}_{1}  $ & $
E_{1}(t)= e_1$ & $
E_{2}(t)= t^{-1}e_2$  \\
\hline

    \end{longtable}

  \begin{longtable}{|l|l|}
\hline
\multicolumn{1}{|c|}{\textrm{Non-degeneration}} & \multicolumn{1}{|c|}{\textrm{Arguments}}\\
\hline
\hline

$\begin{array}{cccc}
     {\rm L}_{5}^{*} &\not \to&  {\rm L}_{3} & 
\end{array}$ &
${\mathcal R}= \left\{ \begin{array}{l}
c_{11}^{1}, c_{11}^{2}, c_{12}^2, c_{21}^2, c_{11}'^{2}, c_{21}'^2 \in \mathbb{C},
c_{11}^{1}=c_{12}^{2},  c_{11}^{1}=c_{21}^{2}
\end{array} \right\}$\\
\hline

$\begin{array}{cccc}
     {\rm L}_{6}^{*} &\not \to&  {\rm L}_{3} & 
\end{array}$ &
${\mathcal R}= \left\{ \begin{array}{l}
c_{11}^{1}, c_{11}^{2}, c_{11}'^{2}, c_{21}'^{2} \in \mathbb{C}
\end{array} \right\}$\\
\hline


    \end{longtable}

\end{proof}

By Lemma \ref{th:degleib} and Lemma \ref{th:degleibf}, we have the following result that summarizes the geometric classification of ${{\rm L}_{2}}$.

\begin{theorem}
The variety of $2$-dimensional Leibniz--Poisson algebras has four irreducible components, corresponding to the Leibniz--Poisson algebras ${\rm L}_{3}$ and ${\rm L}_{4}$ and the families of  Leibniz--Poisson algebras ${\rm L}_{5}^{*}$ and ${\rm L}_{6}^{*}$.
\end{theorem}

%





    


\section{Transposed Leibniz--Poisson algebras}
The notion of transposed Leibniz--Poisson algebras is a non-anticommutative generalization of transposed Poisson algebras. The notion of transposed Poisson algebras was introduced in \cite{bai20}. 
Let us note, that they are related to many other interesting classes of algebras,
such as commutative Gelfand-Dorfman algebras and $F$-manifold algebras \cite{kms}. 
The study of transposed Leibniz--Poisson algebras is motivated by a question posed in \cite{bfk22}.

\subsection{The algebraic classification of $2$-dimensional transposed Leibniz--Poisson algebras}

\begin{definition}

A transposed Leibniz–Poisson algebra   is a vector space   equipped with  a commutative associative multiplication, denoted by $\cdot$ and  a symmetric Leibniz multiplication $[-, -]$. 
These two operations are required to satisfy the following  identity:
\begin{longtable}{lcl}
$2 [x, y]  \cdot z $&$=$&$ [ x\cdot z, y] + [x, y\cdot z].$
\end{longtable}
\end{definition}

\begin{definition}
Let $\left( \rm{A},\cdot \right) $ be a commutative associative
algebra. Define ${\rm Z}_{\rm TLP}^{2}\left( \rm{A},\rm{A}\right) $ to be
the set of all bilinear maps $\theta :\rm{A}\times \rm{A}%
\longrightarrow \rm{A}$ such that:%
\begin{longtable}{rcl}
$\theta \left( \theta \left( x,y\right) ,z\right)$ & $=$&$ \theta \left( \theta
\left( x,z\right) ,y\right) +\theta \left( x,\theta \left( y,z\right)
\right),$ \\
$\theta \left( x,\theta \left( y,z\right) \right)$ & $=$&$ \theta \left( \theta
\left( x,y\right) ,z\right) +\theta \left( y,\theta \left( x,z\right)
\right),$ \\
$2\theta \left( x,y\right) \cdot z$& $=$&$\theta \left( x\cdot z,y\right) +\theta
\left( x,y\cdot z\right),$
\end{longtable}
for all $x,y,z$ in $\rm{A}$. If $\theta \in {\rm Z}_{\rm TLP}^{2}\left( \rm{A},\rm{A}\right) $, then $\left( \rm{A},\cdot , [ -,-] \right) $ is a transposed Leibniz--Poisson algebra where $[x,y]  =\theta \left( x,y\right) $ for all $x,y\in \rm{A}$.
\end{definition}

\subsubsection{The algebraic classification of $2$-dimensional symmertic Leibniz   algebras}\label{symLeib}
By Theorem \ref{leib2}, we obtain that there are only two  
transposed Leibniz--Poisson algebras with trivial multiplication $\cdot$ and non-trivial multiplication $[-,-]:$
${\rm T}_{1}=L_{1},{\rm T}_{2}=L_{3}$.

\subsubsection{Transposed Leibniz--Poisson algebras defined on ${{\rm A}}%
_{01}$}

Since ${\rm Z}_{\rm TLP}^{2}({{\rm A}}_{01},{{\rm A}}_{01})=\left\{ 0\right\} $%
, the only transposed Leibniz--Poisson structure is trivial, denote it by $%
{\rm T}_{3}$.

\subsubsection{Transposed Leibniz--Poisson algebras defined on ${{\rm A}}%
_{02}$}

From the computation of ${\rm Z}_{\rm TLP}^{2}({{\rm A}}_{02},{{\rm A}}_{02})$, 
the transposed Leibniz--Poisson structures defined on ${{\rm A}}_{02}$
are of the form: 
\begin{equation*}
\left\{ 
\begin{tabular}{lcllcl}
$e_{1}\cdot e_{1}$&$=$&$e_{1},$& $e_{1}\cdot e_{2}$&$=$&$e_{2},$ \\ 
$[ e_{1},e_{1}] $&$=$&$\alpha _{1}e_{2}.$%
\end{tabular}%
\right. \qquad   \left\{ 
\begin{tabular}{lcllcl}
$e_{1}\cdot e_{1}$&$=$&$e_{1}, $&$e_{1}\cdot e_{2}$&$=$&$e_{2},$ \\ 
$[ e_{1},e_{2}]$&$=$&$\alpha _{2}e_{2},$& $
[e_{2},e_{1}]  $&$=$&$-\alpha _{2}e_{2}.$%
\end{tabular}%
\right.
\end{equation*}%
We may assume $\alpha _{1}\neq 0$. Then, by Lemma \ref{isom}, we obtain the
following non-isomorphic transposed Leibniz--Poisson algebras:%
\begin{equation*}
{\rm T}_{4}:=\left\{ 
\begin{tabular}{lcllcl}
$e_{1}\cdot e_{1}$&$=$&$e_{1},$&$e_{1}\cdot e_{2}$&$=$&$e_{2},$ \\ 
$[ e_{1},e_{1}] $&$=$&$e_{2}.$%
\end{tabular}%
\right. \qquad {\rm T}_{5}^{\alpha }:=\left\{ 
\begin{tabular}{lcllcl}
$e_{1}\cdot e_{1}$&$=$&$e_{1},$& $e_{1}\cdot e_{2}$&$=$&$e_{2},$ \\ 
$\left\{ e_{1},e_{2}\right\} $&$=$&$\alpha e_{2},$& 
$[ e_{2},e_{1}]$&$=$&$-\alpha e_{2}.$%
\end{tabular}%
\right.
\end{equation*}

\subsubsection{Transposed Leibniz--Poisson algebras defined on ${{\rm A}}%
_{03}$}

Since ${\rm Z}_{\rm TLP}^{2}({{\rm A}}_{03},{{\rm A}}_{03})=\left\{ 0\right\} $%
, the only transposed Leibniz--Poisson structure is trivial, denote it by $%
{\rm T}_{6}$.

\subsubsection{Transposed Leibniz--Poisson algebras defined on ${{\rm A}}%
_{04}$}

The transposed Leibniz--Poisson structures defined on ${{\rm A}}_{04}$ are of the
form: 
\begin{equation*}
\left\{ 
\begin{tabular}{lcl}
$e_{1}\cdot e_{1}$&$=$&$e_{2},$ \\ 
$[ e_{1},e_{1}] $&$=$&$\alpha _{1}e_{2}.$%
\end{tabular}%
\right. \qquad \qquad \qquad \left\{ 
\begin{tabular}{lcllcl}
$e_{1}\cdot e_{1}$&$=$&$e_{2},$ \\ 
$[ e_{1},e_{2}] $&$=$&$\alpha _{2}e_{2}, $&$[ e_{2},e_{1}]$&$=$&$-\alpha _{2}e_{2}.$%
\end{tabular}%
\right.
\end{equation*}

We may assume $\alpha _{2}\neq 0$. Then, by Lemma \ref{isom}, we obtain the
following non-isomorphic transposed Leibniz--Poisson algebras:%
\begin{equation*}
\rm{T}_{7}^{\alpha }:=\left\{ 
\begin{tabular}{lcl}
$e_{1}\cdot e_{1}$&$=$&$e_{2},$ \\ 
$[ e_{1},e_{1}] $&$=$&$\alpha e_{2}.$%
\end{tabular}%
\right. \qquad \qquad \qquad {\rm T}_{8}:=\left\{ 
\begin{tabular}{lcllcl}
$e_{1}\cdot e_{1}$&$=$&$e_{2},$ \\ 
$[e_{1},e_{2}] $&$=$&$e_{2},$&$[ e_{2},e_{1}] $&$=$&$-e_{2}.$%
\end{tabular}%
\right. 
\end{equation*}

\subsubsection{The algebraic classification of $2$-dimensional transposed Leibniz--Poisson  algebras}

\begin{theorem}
Let $\left( {\rm T},\cdot , [-,-] \right) $ be a nonzero $2$-dimensional transposed Leibniz--Poisson algebra. Then ${\rm T}$ is isomorphic to one symmetric Leibniz algebra listed in Subsection \ref{symLeib} or  to one algebra listed below:

\begin{longtable}{lcllll}

${\rm T}_{3}$&$:$&$e_1 \cdot e_1  = e_1, e_2 \cdot e_2  =e_2$\\

${\rm T}_{4}$&$ :$&$ \left\{ 
\begin{tabular}{lcllcl}
$e_1 \cdot e_1  $&$=$&$ e_1,$&$ e_1 \cdot e_2 $&$ =$&$e_2,$ \\ 
$[ e_{1},e_{1}]  $&$=$&$  e_{2}.$%
\end{tabular}%
\right.$\\

${\rm T}_{5}^{\alpha}$&$ :$&$ \left\{ 
\begin{tabular}{lcllcl}
$e_1 \cdot e_1  $&$=$&$ e_1,$&$ e_1 \cdot e_2  $&$=$&$e_2,$ \\ 
$[e_{1},e_{2}]  $&$=$&$ \alpha e_{2},$&$[ e_{2},e_{1}] $&$=$&$ -\alpha e_{2}.$%
\end{tabular}%
\right.$\\

${\rm T}_{6}$&$ :$&$e_1 \cdot e_1 = e_1.$\\ 

${\rm T}_{7}^{\alpha}$&$ :$&$ \left\{ 
\begin{tabular}{lcl}
$e_1 \cdot e_1  $&$=$&$ e_2,$ \\ 
$[  e_{1},e_{1}] $&$=$&$ \alpha e_{2}.$%
\end{tabular}%
\right.$\\

${\rm T}_{8}$&$ :$&$ \left\{ 
\begin{tabular}{lcllcl}
$e_1 \cdot e_1  $&$=$&$ e_2,$ \\ 
$[  e_{1},e_{2}]  $&$=$&$ e_{2},$&$  [ e_{2},e_{1}] $&$=$&$ -e_{2}.$%
\end{tabular}%
\right.$
\end{longtable}
\end{theorem}

\subsubsection{The algebraic classification of $2$-dimensional transposed Poisson algebras}

\begin{definition}

A transposed Poisson algebra   is a vector space   equipped with  a commutative associative multiplication, denoted by $\cdot$ and  a Lie multiplication $[-, -]$. 
These two operations are required to satisfy the following  identity:
\begin{longtable}{lcl}
$2 [x, y]  \cdot z $&$=$&$ [ x\cdot z, y] + [x, y\cdot z].$
\end{longtable}
\end{definition}

\begin{corollary}
Let $\left( {\rm T},\cdot , [-,-] \right) $ be a nonzero $2$-dimensional transposed Poisson algebra. Then ${\rm T}$ is isomorphic to one  algebra listed below: 
\begin{center}
${\rm T}_{2},$ ${\rm T}_{3},$ ${\rm T}_{5}^\alpha,$ ${\rm T}_{6},$  ${\rm T}_{7}^0,$   ${\rm T}_{8}.$
\end{center}
\end{corollary}

\subsection{Degenerations of $2$-dimensional transposed Leibniz--Poisson algebras}

\begin{lemma} \label{th:degtleib}
The graph of primary degenerations and non-degenerations of the variety of $2$-dimensional transposed Leibniz--Poisson algebras is given in Figure  2, where the numbers on the right side are the dimensions of the corresponding orbits.

\end{lemma}
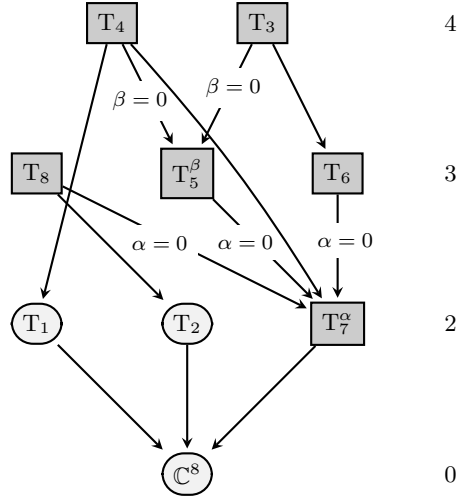
\begin{center}
	
	\begin{tikzpicture}[->,>=stealth,shorten >=0.05cm,auto,node distance=1.3cm,
	thick,
	main node/.style={rectangle,draw,fill=gray!10,rounded corners=1.5ex,font=\sffamily \scriptsize \bfseries },
	rigid node/.style={rectangle,draw,fill=black!20,rounded corners=0ex,font=\sffamily \scriptsize \bfseries }, 
	poisson node/.style={rectangle,draw,fill=black!20,rounded corners=0ex,font=\sffamily \scriptsize \bfseries },
	ac node/.style={rectangle,draw,fill=black!20,rounded corners=0ex,font=\sffamily \scriptsize \bfseries },
	lie node/.style={rectangle,draw,fill=black!20,rounded corners=0ex,font=\sffamily \scriptsize \bfseries },
	style={draw,font=\sffamily \scriptsize \bfseries }]

	\node (3) at (0.5,6) {$4$};
	\node (2) at (0.5,4) {$3$};
	\node (1) at (0.5,2) {$2$};
	\node (0)  at (0.5,0) {$0$};

    \node[main node] (c20) at (-3,0) {${\mathbb C^8}$};
  
    \node[main node] (c21) at (-5,2) {${\rm T}_{1}$};
    \node[main node] (c22) at (-3,2) {${\rm T}_{2}$};
    \node[ac node] (c27) at (-1,2) {${\rm T}_{7}^{\alpha}$};

    \node[ac node] (c25) at (-3,4) {${\rm T}_{5}^{\beta}$};
    \node[ac node] (c26) at (-1,4) {${\rm T}_{6}$};
    \node[ac node] (c28) at (-5,4) {${\rm T}_{8}$};
    
    \node[ac node] (c23) at (-2,6) {${\rm T}_{3}$};
    \node[ac node] (c24) at (-4,6) {${\rm T}_{4}$};
	\path[every node/.style={font=\sffamily\small}]

    
    
    
    (c23) edge node[above=3, right=-12, fill=white]{\tiny $\beta=0$} (c25)
    
    (c24) edge node[above=-2, right=-18, fill=white]{\tiny $\beta=0$} (c25)
    
    (c26) edge node[above=3, right=-12, fill=white]{\tiny $\alpha=0$} (c27)
    (c25) edge node[above=3, right=-22, fill=white]{\tiny $\alpha=0$} (c27)
    (c28) edge node[above=2, right=-25, fill=white]{\tiny $\alpha=0$} (c27)
    
    (c23) edge (c26)
    (c28) edge (c22)
  
        (c24) edge [bend left=0] (c21)
    (c24) edge [bend left=5]  (c27)

    (c21) edge (c20)
    (c22) edge (c20)
    (c27) edge (c20);

	\end{tikzpicture}

{\tiny 
\begin{itemize}
\noindent Legend:
\begin{itemize}
    \item[--] Round nodes: transposed Leibniz--Poisson algebras with trivial commutative associative multiplication $\cdot$.
    \item[--] Squared nodes: transposed Leibniz--Poisson algebras with nontrivial commutative associative multiplication $\cdot$.

\end{itemize}
\end{itemize}}

{Figure 2.}  Graph of primary degenerations and non-degenerations.	
\end{center}
\bigskip

\bigskip

\begin{proof}
The dimensions of the orbits are deduced by computing the algebra of derivations.
The primary degenerations are proven using the parametric bases included in following Table:

    \begin{longtable}{|lcl|ll|}
\hline
\multicolumn{3}{|c|}{\textrm{Degeneration}}  & \multicolumn{2}{|c|}{\textrm{Parametrized basis}} \\
\hline
\hline

${\rm T}_{3} $ & $\to$ & ${\rm T}_{5}^{0} $ & $
E_{1}(t)= e_1 + e_2$ & $
E_{2}(t)= t e_2$  \\
\hline

${\rm T}_{3} $ & $\to$ & ${\rm T}_{6} $ & $
E_{1}(t)= e_1$ & $
E_{2}(t)= t e_2$  \\
\hline

${\rm T}_{4} $ & $\to$ & ${\rm T}_{1} $ & $
E_{1}(t)= t e_1$ & $
E_{2}(t)= t^2 e_2$  \\
\hline

${\rm T}_{4} $ & $\to$ & ${\rm T}_{5}^{0} $ & $
E_{1}(t)= e_1$ & $
E_{2}(t)= t^{-1}e_2$  \\
\hline

${\rm T}_{4} $ & $\to$ & ${\rm T}_{7}^{\alpha\neq0} $ & $
E_{1}(t)= \alpha t e_1 + t e_2$ & $
E_{2}(t)= \alpha t^{2} e_2$  \\
\hline

${\rm T}_{5}^{\alpha} $ & $\to$ & ${\rm T}_{7}^{0} $ & $
E_{1}(t)= t e_1 + e_2$ & $
E_{2}(t)= t e_2$  \\
\hline

${\rm T}_{6} $ & $\to$ & ${\rm T}_{7}^{0} $ & $
E_{1}(t)= t e_1 + e_2$ & $
E_{2}(t)= - te_2$  \\
\hline

${\rm T}_{8} $ & $\to$ & ${\rm T}_{2} $ & $
E_{1}(t)= e_1$ & $
E_{2}(t)= t^{-1}e_2$  \\
\hline

${\rm T}_{8} $ & $\to$ & ${\rm T}_{7}^{0} $ & $
E_{1}(t)= t e_1$ & $
E_{2}(t)= t^2 e_2$  \\
\hline
 
    \end{longtable}

The primary non-degenerations are proven using the sets from the following table: 
  
    \begin{longtable}{|l|l|}
\hline
\multicolumn{1}{|c|}{\textrm{Non-degeneration}} & \multicolumn{1}{|c|}{\textrm{Arguments}}\\
\hline
\hline

$\begin{array}{cccc}
     {\rm T}_{3} &\not \to&  {\rm T}_{1}, {\rm T}_{2},      {\rm T}_{5}^{\alpha\neq0} 
, {\rm T}_{7}^{\alpha\neq0}
\end{array}$ &
${\mathcal R}= \left\{ \begin{array}{l} c_{11}^1, c_{11}^2, c_{12}^2, c_{21}^2, c_{22}^2 \in \mathbb{C},
c_{12}^2=c_{21}^2
\end{array} \right\}$\\
\hline

$\begin{array}{cccc}
     {\rm T}_{4} &\not \to&  {\rm T}_{2}, 
     {\rm T}_{5}^{\alpha\neq0}, {{\rm T}_{6}} & 
\end{array}$ &
${\mathcal R}= \left\{ \begin{array}{l} c_{11}^1, c_{11}^2, c_{12}^2, c_{21}^2, c_{11}'^2 \in \mathbb{C},
c_{11}^1=c_{12}^2,  c_{11}^1=c_{21}^2
\end{array} \right\}$\\
\hline

$\begin{array}{cccc}
     {\rm T}_{5}^{\alpha} &\not \to&  {\rm T}_{1}, {\rm T}_{2}, {\rm T}_{7}^{\beta\neq0} & 
\end{array}$ &
${\mathcal R}= \left\{ \begin{array}{l} c_{11}^1, c_{11}^2, c_{12}^2, c_{21}^2, c_{12}'^2, c_{21}'^2 \in \mathbb{C},\\
c_{11}^1=c_{12}^2,  c_{11}^1=c_{21}^2,  c_{12}'^2=\alpha  c_{11}^1,  c_{21}'^2=-c_{12}'^2
\end{array} \right\}$\\
\hline


$\begin{array}{cccc}
     {\rm T}_{8} &\not \to&  {\rm T}_{1}, {\rm T}_{7}^{\alpha\neq0}  & 
\end{array}$ &
${\mathcal R}= \left\{ \begin{array}{l} c_{11}^2, c_{12}'^2, c_{21}'^2 \in \mathbb{C},
c_{21}'^2=-c_{12}'^2
\end{array} \right\}$\\
\hline

    \end{longtable}

\end{proof}

At this point, only the description of the closures of the orbits of the parametric families is missing. Although it is not necessary to study the closure of the orbits of each of the parametric families of
the variety of $2$-dimensional transposed Leibniz--Poisson algebras in order to identify its irreducible components, we will study them to give a complete description of the variety.

\begin{lemma}\label{th:degtleibf}
The description of the closure of the orbit of the parametric families in
the variety of $2$-dimensional transposed Leibniz--Poisson algebras are given below:

\begin{longtable}{lcl}
$\overline{\{O({\rm T}_{5}^{*})\}}$ & ${\supseteq}$ &
$ \Big\{
\overline{\{O({\rm T}_{2})\}}, \overline{\{O({\rm T}_{7}^{0})\}}, \overline{\{O({\rm T}_{8})\}},
\overline{\{O({\mathbb C^8})\}} \Big\}$\\

$\overline{\{O({\rm T}_{7}^{*})\}}$ & ${\supseteq}$ &
$ \Big\{\overline{\{O({\rm T}_{1})\}},
\overline{\{O({\mathbb C^8})\}} \Big\}$\\

\end{longtable}

\end{lemma}

\begin{proof}
Thanks to Theorem \ref{th:degtleib} we have all necessarily degenerations.
The closure of the orbits of suitable families are given using the parametric bases and the parametric index included in following Table: 

    \begin{longtable}{|lcl|ll|}
\hline
\multicolumn{3}{|c|}{\textrm{Degeneration}}  & \multicolumn{2}{|c|}{\textrm{Parametrized basis}} \\
\hline
\hline

${\rm T}_{5}^{t^{-1}} $ & $\to$ & ${\rm T}_{8} $ & $
E_{1}(t)= t e_1 + e_2$ & $
E_{2}(t)= t e_2$  \\
\hline

${\rm T}_{7}^{t^{-2}} $ & $\to$ & ${\rm T}_{1} $ & $
E_{1}(t)= t e_1$ & $
E_{2}(t)= e_2$  \\
\hline
    \end{longtable}

  \begin{longtable}{|l|l|}
\hline
\multicolumn{1}{|c|}{\textrm{Non-degeneration}} & \multicolumn{1}{|c|}{\textrm{Arguments}}\\
\hline
\hline

$\begin{array}{cccc}
     {\rm T}_{5}^{*} &\not \to&  {\rm T}_{1}, {\rm T}_{6}, {\rm T}_{7}^{\alpha\neq0} & 
\end{array}$ &
${\mathcal R}= \left\{ \begin{array}{l} c_{11}^1, c_{11}^2, c_{12}^2, c_{21}^2, c_{12}'^2, c_{21}'^2 \in \mathbb{C},\\
c_{11}^1=c_{12}^2,  c_{11}^1=c_{21}^2, c_{21}'^2=-c_{12}'^2
\end{array} \right\}$\\
\hline

$\begin{array}{cccc}
     {\rm T}_{7}^{*} &\not \to&  {\rm T}_{2} & 
\end{array}$ &
${\mathcal R}= \left\{ \begin{array}{l} c_{11}^2, c_{11}'^2 \in \mathbb{C}
\end{array} \right\}$\\
\hline

    \end{longtable}

\end{proof}

By Lemma \ref{th:degtleib} and Lemma \ref{th:degtleibf}, we have the following result that summarizes the geometric classification.

\begin{theorem}
The variety of $2$-dimensional transposed Leibniz--Poisson algebras has three irreducible components, corresponding to the transposed Leibniz--Poisson algebras ${\rm T}_{3}$ and ${\rm T}_{4}$ and the family of transposed Leibniz--Poisson algebras ${\rm T}_{5}^{*}$.
\end{theorem}

 \subsubsection{The geometric classification of $2$-dimensional transposed
Poisson algebras}

\begin{corollary}

The variety of $2$-dimensional transposed Poisson algebras has two  irreducible components, corresponding to the transposed  Poisson algebra ${\rm T}_{3}$  and the family of transposed Poisson algebras ${\rm T}_{5}^{*}$.
\end{corollary}

\section{ Novikov--Poisson algebras}
The notion of Novikov--Poisson (also known as Gelfand-Dorfman-Novikov--Poisson \cite{B18}) algebras is introduced in a paper of Xu  in order to study Novikov algebras \cite{xu96}.
He uses these new algebras to give the first examples of simple Novikov algebras.
The next paper of Xu 
opens the systematic study of Novikov--Poisson algebras \cite{xu97}.
Novikov--Poisson algebras are related to Jordan superalgebras and transposed Poisson algebras \cite{bai20}.
A first tentative of the algebraic classification of $2$- and $3$-dimensional Novikov--Poisson algebras is given in \cite{z15},
but the author did not separate isomorphic algebras.
Our aim in the present section is to improve and extend the results results in \cite{z15}.

\subsection{The algebraic classification of $2$-dimensional Novikov algebras}
\begin{definition}
An algebra  is called
Novikov if 
\begin{longtable}{rcl}
$ ( x \circ y)\circ z-x \circ ( y \circ z) $ & $=$& $(y  \circ  x)  \circ  z-y  \circ  ( x  \circ z) $, \\
$ ( x  \circ  y)  \circ  z$& $=$ &$( x \circ z)  \circ y.$
\end{longtable}
\end{definition}
The algebraic classification of  $2$-dimensional Novikov algebras is given by \cite{Bai}.

\begin{theorem}
\label{nov2}
Let ${\rm N}$ be a nonzero $2$-dimensional Novikov algebra. Then ${\rm N}$ is
isomorphic to one and only one of the following algebras:

\begin{longtable}{lllcllcllcl}
${\rm N}_{01}$&$:$&$  e_{1} \circ e_{1} $&$=$&$e_{2}$\\

 ${\rm N}_{02}$&$:$&$ e_{2} \circ e_{1}  $&$=$&$-e_{1}$\\

${\rm N}_{03}$&$:$&$ e_{1} \circ e_{1} $&$=$&$e_{1},$ & $ e_{2}\circ e_{2}$&$=$&$e_{2}$\\

${\rm N}_{04}$&$ :$&$ e_{1}\circ e_{1}  $&$=$&$e_{1}$\\

${\rm N}_{05}$ & $:$&$ e_{1}\circ e_{2}  $&$=$&$ e_{1},$ & $e_{2}\circ e_{2}$&$=$&$e_{1}+e_{2}$\\

${\rm N}_{06}^{\gamma}$ & $:$ & $e_{1}\circ e_{2} $&$=$&$e_{1},$ & $e_{2}\circ  e_{1}  $&$=$&$\gamma e_{1},$ & $ e_{2}\circ e_{2} $&$=$&$e_{2}$
\end{longtable}

\end{theorem}

\subsection{The algebraic classification of  $2$-dimensional Novikov--Poisson algebras}

\begin{definition}
A Novikov--Poisson algebra is a vector space  equipped with 
a commutative associative multiplication $\cdot$
and     
a Novikov multiplication $\circ.$
These two operations are required to satisfy the following identities:
\begin{longtable}{rcl}
$(x\cdot  y) \circ z$ &$  =$&$ x \cdot( y \circ z),$\\
$(x \circ y)\cdot z - (y \circ x)\cdot z$&$ =$&$ x\circ (y \cdot z)  - y \circ (x \cdot z).$
\end{longtable}
\end{definition}

Novikov--Poisson algebras 
$({\rm N}, \cdot, \circ)$ with zero commutative associative multiplication  are the Novikov algebras given in Theorem \ref{nov2}.  Hence, we are studying the Novikov--Poisson algebras defined on every  commutative associative algebra from Theorem \ref{asocc2}.

\begin{definition}
Let $\left( {\rm A},\cdot \right) $ be a commutative associative
algebra. Define ${\rm Z}_{\rm NP}^{2}\left( {\rm A},{\rm A}\right) $ to be the
set of all bilinear maps $\theta :{\rm A}\times {\rm A} \longrightarrow {\rm A}$ such that:
\begin{longtable}{rcl}
$\theta \left( x,\theta \left( y,z\right) \right) -\theta \left( \theta
\left( x,y\right) ,z\right)$ & $=$&$\theta \left( y,\theta \left( x,z\right)
\right) -\theta \left( \theta \left( y,x\right) ,z\right),$ \\
$\theta \left( \theta \left( x,y\right) ,z\right)$ &$ =$&$\theta \left( \theta
\left( x,z\right) ,y\right),$ \\
$\theta \left( x\cdot y,z\right)$ & $=$ &$x\cdot \theta \left( y,z\right),$ \\
$\theta \left( x,y\right) \cdot z-\theta \left( y,x\right) \cdot z$& $=$&$\theta
\left( x,y\cdot z\right) -\theta \left( y,x\cdot z\right) ,$
\end{longtable}%
for all $x,y,z$ in ${\rm A}$. 
If $\theta \in {\rm Z}_{\rm NP}^{2}\left( {\rm A}, {\rm A}\right) $, then 
$\left( {\rm A},\cdot ,\circ \right) $ is a Novikov--Poisson algebra where $ x \circ y=\theta \left( x,y\right),$ for all $x,y\in {\rm A}$. 
\end{definition}

\subsubsection{Novikov--Poisson algebras defined on   ${\rm A}_{01}$}

The Novikov--Poisson structures defined on ${\rm A}_{01}$ are:
$${\rm N}_{07}^{\alpha, \beta} : = \left\{ 
\begin{tabular}{lcllcl}
$e_1 \cdot e_1 $&$=$&$ e_1,$&$ e_2 \cdot e_2 $&$=$&$ e_2$ \\ 
$e_{1} \circ e_{1}  $&$=$&$ \alpha e_{1},$ &$ e_{2} \circ e_{2} $&$=$&$ \beta e_{2}$%
\end{tabular}%
\right.$$
Since ${\rm Aut}({\rm A}_{01})=\mathbb S_2$, we have 
${\rm N}_{07}^{\alpha, \beta}\cong {\rm N}_{07}^{\beta, \alpha}$.

\subsubsection{Novikov--Poisson algebras defined on   ${\rm A}_{02}$}

The Novikov--Poisson structures defined on ${\rm A}_{02}$ are of the form:
$$\left\{ 
\begin{tabular}{lcllcllcl}
$e_1 \cdot e_1  $&$=$&$ e_1,$&$ e_1 \cdot e_2  $&$=$&$e_2$ \\ 
$ e_{1} \circ e_{1}  $&$=$&$ \alpha_{1} e_{1} + \beta_{1} e_{2},$&$   e_{1} \circ e_{2}  $&$=$&$ \gamma_{1} e_{2},$ & $ e_{2} \circ e_{1}  $&$=$&$ \alpha_{1} e_{2}$%
\end{tabular}%
\right.$$

By Lemma \ref{asocc2aut}, to study the action of ${\rm Aut}({\rm A}_{02})$, we have to distinguish two cases:
\begin{itemize}
    \item If $\beta_{1}\neq0$, then choose $\xi=\beta_{1}$ and obtain  the parametric family:
$${\rm N}_{08}^{\alpha, \beta} := \left\{ 
\begin{tabular}{lcllcllcl}
$e_1 \cdot e_1  $&$=$&$ e_1,$&$ e_1 \cdot e_2  $&$=$&$ e_2$ \\ 
$e_{1}\circ e_{1} $&$=$&$ \alpha e_{1} + e_{2},$ & $e_{1}\circ e_{2}  $&$=$&$ \beta e_{2},$ & $ e_{2}\circ e_{1} $&$=$&$ \alpha e_{2}$%
\end{tabular}%
\right.$$
    
    \item If $\beta_{1}=0$, then,  we have:
$${\rm N}_{09}^{\alpha, \beta} := \left\{ 
\begin{tabular}{lll}
$e_1 \cdot e_1  = e_1,$& $ e_1 \cdot e_2  =e_2$ \\ 
$ e_{1}\circ e_{1} = \alpha e_{1},$ & $  e_{1}\circ e_{2}  = \beta e_{2},$ & $  e_{2}\circ e_{1}  = \alpha e_{2}$%
\end{tabular}%
\right.$$
    
\end{itemize}

\subsubsection{Novikov--Poisson algebras defined on   ${\rm A}_{03}$}
The Novikov--Poisson structures defined on ${\rm A}_{03}$ are:
$$\left\{ 
\begin{tabular}{lll}
$e_1 \cdot e_1 = e_1,$ \\ 
$ e_{1}\circ e_{1}  = \alpha_{1} e_{1},$ & $ e_{2}\circ e_{2}= \beta_{1} e_{2}$%
\end{tabular}%
\right.$$

The action of ${\rm Aut}({\rm A}_{03})$ produces two cases:
\begin{itemize}
    \item If $\beta_{1}\neq0$, then set $\xi=\beta_{1}^{-1}$ and obtain the family:
    $${\rm N}_{10}^{\alpha} := \left\{ 
\begin{tabular}{lcllcl}
$e_1 \cdot e_1 $&$=$&$ e_1$ \\ 
$e_{1}\circ e_{1} $&$=$&$ \alpha e_{1},$ & $ e_{2}\circ e_{2} $&$=$&$ e_{2}$%
\end{tabular}%
\right.$$

    \item If $\beta_{1} = 0$, then we obtain the family:
    $${\rm N}_{11}^{\alpha} := \left\{ 
\begin{tabular}{lcllcl}
$e_1 \cdot e_1 $&$=$&$ e_1$ \\ 
$e_{1}\circ e_{1}  $&$=$&$ \alpha e_{1}$%
\end{tabular}%
\right.$$
    
\end{itemize}

\subsubsection{Novikov--Poisson algebras defined on   ${\rm A}_{04}$}

The Novikov--Poisson structures defined on ${\rm A}_{04}$ are:
$$\left\{ 
\begin{tabular}{lcllcllcl}
$e_1 \cdot e_1  $&$=$&$ e_2$ \\ 
$ e_{1}\circ e_{1}  $&$=$&$ \alpha_{1} e_{1} + \beta_{1} e_{2},$ & $ e_{1}\circ e_{2} $&$=$&$ \gamma_{1} e_{2},$ &$  e_{2}\circ e_{1}  $&$=$&$ \alpha_{1} e_{2}$%
\end{tabular}%
\right.$$

The following cases arise from the action of  ${\rm Aut}({\rm A}_{04})$ on this structure:
\begin{itemize}
    \item If $\gamma_{1}\neq0$, then choose $\xi=\gamma_{1}^{-1}$ and $\nu=-\beta_{1} \gamma_{1}^{-2}$ obtaining  the family:
    $${\rm N}_{12}^{\alpha} := \left\{ 
\begin{tabular}{lcllcllcl}
$e_1 \cdot e_1  $&$=$&$ e_2$ \\ 
$ e_{1}\circ e_{1}  $&$=$&$ \alpha e_{1},$& $  e_{1}\circ e_{2} $&$=$&$ e_{2},$&$  e_{2}\circ e_{1}  $&$=$&$ \alpha e_{2}$%
\end{tabular}%
\right.$$
    
    \item If $\gamma_{1}=0$, then 
    \begin{itemize}
        \item If $\alpha_{1}\neq0$, then for $\xi=\alpha_{1}^{-1}$, we obtain:        
$${\rm N}_{13}^{\alpha} := \left\{ 
\begin{tabular}{lcllcl}
$e_1 \cdot e_1  $&$=$&$ e_2$ \\ 
$ e_{1}\circ e_{1}  $&$=$&$ e_{1} + \alpha e_2,$ &$  e_{2}\circ e_{1} $&$=$&$ e_{2}$%
\end{tabular}%
\right.$$
        \item If $\alpha_{1}=0$, then we have  the algebra: 
$${\rm N}_{14}^{\alpha} := \left\{ 
\begin{tabular}{lcllcl}
$e_1 \cdot e_1  $&$=$&$ e_2$ \\ 
$ e_{1}\circ e_{1}  $&$=$&$ \alpha e_{2}$%
\end{tabular}%
\right.$$
    \end{itemize}
\end{itemize}

\subsubsection{The algebraic classification of  $2$-dimensional Novikov-Poisson 
 algebras}

\begin{theorem}
\label{2dim NP} Let $({\rm N},\cdot , \circ ) $ be a nonzero $2$-dimensional Novikov--Poisson algebra. Then ${\rm N}$ is isomorphic to one Novikov algebra listed in Theorem \ref{nov2} or  to one algebra listed below: 
\begin{longtable}{lcl} 
 ${\rm N}_{07}^{\alpha, \beta}$ &$ :$&$ 
 \left\{ 
\begin{tabular}{lcllcl}
$e_1 \cdot e_1 $&$=$&$ e_1,$&$ e_2 \cdot e_2 $&$=$&$ e_2$ \\ 
$ e_{1}\circ e_{1}  $&$=$&$ \alpha e_{1}, $&$  e_{2}\circ e_{2}  $&$=$&$ \beta e_{2}$%
\end{tabular}%
\right.$\\

${\rm N}_{08}^{\alpha, \beta} $&$:$&$ \left\{ 
\begin{tabular}{lcllcllcl}
$e_1 \cdot e_1  $&$=$&$ e_1,$&$ e_1 \cdot e_2  $&$=$&$e_2$ \\ 
$ e_{1}\circ e_{1}  $&$=$&$ \alpha e_{1} + e_{2},$&$  e_{1}\circ e_{2}  $&$=$&$ \beta e_{2},$&$e_{2}\circ e_{1}$&$=$&$ \alpha e_{2}$%
\end{tabular}%
\right.$\\

${\rm N}_{09}^{\alpha, \beta} $&$:$&$ \left\{ 
\begin{tabular}{lcllcllcllcl}
$e_1 \cdot e_1  $&$=$&$ e_1, $&$e_1 \cdot e_2  $&$=$&$e_2$ \\ 
$e_{1}\circ e_{1}  $&$=$&$ \alpha e_{1},$&$  e_{1}\circ e_{2} $&$=$&$ \beta e_{2},$&$ e_{2}\circ e_{1} $&$=$&$ \alpha e_{2}$%
\end{tabular}%
\right.$\\

${\rm N}_{10}^{\alpha}$&$ : $&$\left\{ 
\begin{tabular}{lcllcl}
$e_1 \cdot e_1 $&$=$&$ e_1$ \\ 
$ e_{1}\circ e_{1}  $&$=$&$ \alpha e_{1},$&$  e_{2}\circ e_{2} $&$=$&$ e_{2}$%
\end{tabular}%
\right.$\\

 ${\rm N}_{11}^{\alpha} $&$:$&$ \left\{ 
\begin{tabular}{lcllcl}
$e_1 \cdot e_1 $&$=$&$ e_1$ \\ 
$e_{1}\circ e_{1}  $&$=$&$ \alpha e_{1}$%
\end{tabular}%
\right.$\\

 ${\rm N}_{12}^{\alpha} $&$:$&$ \left\{ 
\begin{tabular}{lcllcllcl}
$e_1 \cdot e_1  $&$=$&$ e_2$ \\ 
$ e_{1}\circ e_{1}  $&$=$&$ \alpha e_{1}, $&$ e_{1}\circ e_{2} $&$ =$&$ e_{2},$&$ e_{2}\circ e_{1} $&$=$&$ \alpha e_{2}$%
\end{tabular}%
\right.$\\

 ${\rm N}_{13}^{\alpha}$&$ :$&$ \left\{ 
\begin{tabular}{lcllcl}
$e_1 \cdot e_1  $&$=$&$ e_2,$ \\ 
$e_{1}\circ e_{1}  $&$=$&$ e_{1} + \alpha e_2$&$ e_{2}\circ e_{1}  $&$=$&$ e_{2}$%
\end{tabular}%
\right.$\\

${\rm N}_{14}^{\alpha} $&$:$&$ \left\{ 
\begin{tabular}{lcllcl}
$e_1 \cdot e_1  $&$=$&$ e_2$ \\ 
$e_{1}\circ e_{1} $&$=$&$ \alpha e_{2}$%
\end{tabular}%
\right.$\\

\end{longtable}
The only non-trivial isomorphism between these families is ${\rm N}_{07}^{\alpha, \beta}\cong {\rm N}_{07}^{\beta, \alpha}$.

\end{theorem}

\subsubsection{The algebraic classification of $2$-dimensional pre-Lie
Poisson algebras}

\begin{definition}
A pre-Lie algebra is a vector space ${\rm P}$ equipped with 
a pre-Lie algebra multiplication  $\circ$  satisfying  the following condition: 
\begin{longtable}{rcl}
$\left( x\circ y\right) \circ z-\left( y\circ x\right) \circ z$ &$=$&
$x\circ \left(y\circ z\right) -y\circ \left( x\circ z\right).$ 
\end{longtable}
\end{definition}

\begin{definition}
A pre-Lie Poisson algebra is a vector space ${\rm P}$ equipped with 
a  commutative associative multiplication  $\cdot$ and 
a pre-Lie algebra multiplication  $\circ.$
These two operations are required to satisfy the following conditions: 
\begin{longtable}{rcl}
$\left( x\circ y\right) \cdot z-\left( y\circ x\right) \cdot z$ &$=$&$x\circ \left(
y\cdot z\right) -y\circ \left( x\cdot z\right),$ \\
$\left( x\cdot y\right) \circ z$ &$=$&$x\cdot \left( y\circ z\right).$
\end{longtable}
\end{definition}

\begin{definition}
Let $\left( \rm{P},\cdot \right) $ be a commutative associative
algebra. Define ${\rm Z}_{\rm PLP}^{2}\left( \rm{P},\rm{P}\right) $ to be
the set of all bilinear maps $\theta :\rm{P}\times \rm{P}%
\longrightarrow \mathcal{P}$ such that:%
\begin{longtable}{rcl}
$\theta \left( \theta(x,  y),z\right) -\theta ( x, \theta \left( y,z\right)) $&$=
 $&$ \theta \left( \theta(y,
 x),z\right) -\theta
\left( y, \theta( x ,z\right),$ \\
$\theta \left( x,y\right) \cdot z-\theta \left( y,x\right) \cdot z$&$=$&$\theta
\left( x,y\cdot z\right) -\theta \left( y,x\cdot z\right),$ \\
$\theta \left( x\cdot y,z\right)$&$=$&$x\cdot \theta \left( y,z\right),$
\end{longtable}%
for all $x,y,z$ in $\rm{P}$. If $\theta \in {\rm Z}_{\rm PLP}^{2}\left( 
\rm{P},\rm{P}\right) $, then $\left( \rm{P},\cdot ,\circ 
\right) $ is a pre-Lie Poisson algebra where $x\circ  y=\theta \left(
x,y\right) $ for all $x,y\in \rm{P}$.
\end{definition}

The main tool for the algebraic classification of $2$-dimensional pre-Lie Poisson algebras 
can be obtained by a direct computations and given in next lemma.

\begin{lemma}
Let ${\rm P}$ be a nonzero $2$-dimensional commutative associative algebra, 
then ${\rm Z}_{\rm PLP}^{2}\left( 
\rm{P},\rm{P}\right)={\rm Z}_{\rm NP}^{2}\left( 
\rm{P},\rm{P}\right) $
\end{lemma}

It is know, that each Novikov algebra is a Pre-Lie algebra.
The algebraic classification of  $2$-dimensional pre-Lie algebras can be found in  \cite{burde09}. Hence, we have the following statement.

\begin{corollary}\label{algprelieP}
Let $({\rm P},\cdot , \circ ) $ be a nonzero $2$-dimensional pre-Lie Poisson algebra. 
Then ${\rm P}$ is isomorphic to one Novikov--Poisson algebra listed in Theorem \ref{2dim NP} or  
to one pre-Lie (non-Novikov) algebra listed below: 

\begin{longtable}{lclcllcllcl}

${\rm P}_{01}^{\alpha\neq0}$&$:$&
$e_2 \circ e_1 $&$=$&$ -e_1,$&$ e_2 \circ e_2 $&$=$&$ \alpha e_2$\\

${\rm P}_{02}$&$:$&
$e_1 \circ e_1 $&$=$&$ e_2,$ &$ e_2 \circ e_1 $&$=$&$ -e_1, $&$e_2 \circ e_2 $&$=$&$ -2e_2$\\

${\rm P}_{03}$&$:$&
 $e_2 \circ e_1 $&$=$&$ -e_1,$ & $ e_2 \circ e_2 $&$=$&$ e_1 - e_2$\\ 

\end{longtable}

\end{corollary}

\subsection{Degenerations  of  $2$-dimensional Novikov--Poisson algebras}

\begin{theorem} \label{th:degnov}
The graph of primary degenerations and non-degenerations of the variety of $2$-dimensional Novikov--Poisson algebras is given in Figure 3, where the numbers on the right side are the dimensions of the corresponding orbits.

\end{theorem}
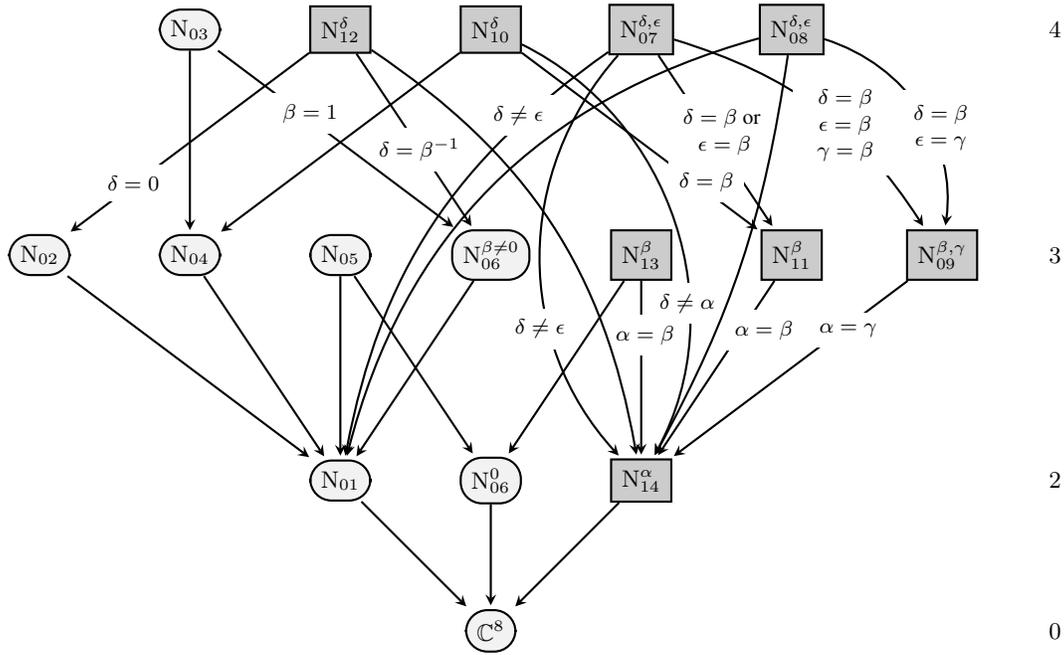
\begin{center}
	
	\begin{tikzpicture}[->,>=stealth,shorten >=0.05cm,auto,node distance=1.3cm,
	thick,
	main node/.style={rectangle,draw,fill=gray!10,rounded corners=1.5ex,font=\sffamily \scriptsize \bfseries },
	rigid node/.style={rectangle,draw,fill=black!20,rounded corners=0ex,font=\sffamily \scriptsize \bfseries }, 
	poisson node/.style={rectangle,draw,fill=black!20,rounded corners=0ex,font=\sffamily \scriptsize \bfseries },
	ac node/.style={rectangle,draw,fill=black!20,rounded corners=0ex,font=\sffamily \scriptsize \bfseries },
	lie node/.style={rectangle,draw,fill=black!20,rounded corners=0ex,font=\sffamily \scriptsize \bfseries },
	style={draw,font=\sffamily \scriptsize \bfseries }]

	\node (3) at (0.5,8) {$4$};
	\node (2) at (0.5,5) {$3$};
	\node (1) at (0.5,2) {$2$};
	\node (0)  at (0.5,0) {$0$};

    \node[main node] (c20) at (-7,0) {${\mathbb C}^8$};
    
    \node[ac node] (c214) at (-5,2) {${\rm N}_{14}^{\alpha}$};
    \node[main node] (c260) at (-7,2) {${\rm N}_{06}^{0}$};
    \node[main node] (c21) at (-9,2) {${\rm N}_{01}$};
    
    \node[ac node] (c213) at (-5,5) {${\rm N}_{13}^{\beta}$};
    \node[ac node] (c211) at (-3,5) {${\rm N}_{11}^{\beta}$};
    \node[ac node] (c29) at (-1,5) {${\rm N}_{09}^{\beta,\gamma}$};
    \node[main node] (c26a) at (-7,5) {${\rm N}_{06}^{\beta\neq0}$};
    \node[main node] (c25) at (-9,5) {${\rm N}_{05}$};
    \node[main node] (c24) at (-11,5) {${\rm N}_{04}$};
    \node[main node] (c22) at (-13,5) {${\rm N}_{02}$};
    
    \node[ac node] (c212) at (-9,8) {${\rm N}_{12}^{\delta}$};
    \node[ac node] (c210) at (-7,8) {${\rm N}_{10}^{\delta}$};
    \node[ac node] (c28) at (-3,8) {${\rm N}_{08}^{\delta,\epsilon}$};
    \node[ac node] (c27) at (-5,8) {${\rm N}_{07}^{\delta,\epsilon}$};
    \node[main node] (c23) at (-11,8) {${\rm N}_{03}$};
    
	\path[every node/.style={font=\sffamily\small}]

    
    (c27) edge node[above=4, right=-22, fill=white]{\tiny $\begin{array}{c}
        \delta=\beta \textrm{ or } \\
        \epsilon = \beta  
    \end{array}$}  (c211)
    
    (c210) edge[bend left=46] node[above=-30, right=-10, fill=white]{\tiny $\delta\neq\alpha$}  (c214)
    
    (c27) edge[bend left=-22] node[above=42, right=19, fill=white]{\tiny $\delta\neq\epsilon$}  (c21)
    
    (c27) edge[bend left=-42] node[above=-27, right=-14, fill=white]{\tiny $\delta\neq\epsilon$}  (c214)
    
    (c27) edge [bend left=20] node[above=-6, right=-10, fill=white]{\tiny $\begin{array}{c}
        \delta=\beta  \\
        \epsilon=\beta  \\
        \gamma=\beta 
    \end{array}$}  (c29)
    
    (c28) edge [bend left=40] node[above=-6, right=-12, fill=white]{\tiny $\begin{array}{c}
        \delta=\beta  \\
        \epsilon=\gamma 
    \end{array}$}  (c29)
    (c28) edge [bend left=-30] node{}  (c21)
     
    (c212) edge node[above=-15, right=-35, fill=white]{\tiny $\delta=0$}  (c22)
    (c212) edge node[above=-3, right=-18, fill=white]{\tiny $\delta=\beta^{-1}$}  (c26a)

    (c210) edge node[above=-15, right=10, fill=white]{\tiny $\delta=\beta$}  (c211)
    (c210) edge   node{}  (c24)
    
    (c29) edge node[above=15, right=7, fill=white]{\tiny $\alpha=\gamma$}  (c214)
    (c211) edge node[above=15, right=3, fill=white]{\tiny $\alpha=\beta$}  (c214)
    (c213) edge node[above=13, right=-14, fill=white]{\tiny $\alpha=\beta$}  (c214)
    (c213) edge   node{}  (c260)
    
    (c23) edge node[above=10, right=-25, fill=white]{\tiny $\beta=1$}  (c26a)
    (c23) edge   node{}  (c24)
    (c25) edge   node{}  (c260)
    (c26a) edge   node{}  (c21)
    (c25) edge   node{}  (c21)
    (c24) edge   node{}  (c21)
    (c22) edge   node{}  (c21)
    (c214) edge   node{}  (c20)
    (c260) edge   node{}  (c20)
    (c21) edge   node{}  (c20)
    (c212) edge [bend left=23]  node{}  (c214)
      (c28) edge [bend left=10] node{}  (c214)
  ;

	\end{tikzpicture}

{\tiny 
\begin{itemize}
\noindent Legend:
\begin{itemize}
    \item[--] Round nodes:  Novikov--Poisson algebras with trivial commutative associative multiplication $\cdot$.
    \item[--] Squared nodes: Novikov--Poisson algebras with non-trivial commutative associative multiplication $\cdot$.
\end{itemize}
\end{itemize}}

{Figure 3.}  Graph of primary degenerations and non-degenerations of $2$-dimensional Novikov--Poisson algebras.	
\end{center}
\bigskip

\bigskip

\begin{proof}
The dimensions of the orbits are deduced by computing the algebra of derivations.
The primary degenerations of the Novikov--Poisson algebras with trivial commutative associative multiplication follow by \cite{kv16}, the others are proven using the parametric bases included in following Table:

    \begin{longtable}{|lcl|ll|}
\hline
\multicolumn{3}{|c|}{\textrm{Degeneration}}  & \multicolumn{2}{|c|}{\textrm{Parametrized basis}} \\
\hline
\hline

%
%
%
%
%

%
%


 ${\rm N}_{07}^{\alpha, \beta\neq\alpha} $ & $\to$ & ${\rm N}_{01}  $ & $
E_{1}(t)= t e_1 + t e_2$ & $
E_{2}(t)= t^2 (\alpha - \beta)e_1$  \\
\hline

${\rm N}_{07}^{\alpha, \alpha} $ & $\to$ & ${\rm N}_{09}^{\alpha ,\alpha}  $ & $
E_{1}(t)= e_1 + e_2$ & $
E_{2}(t)= -te_2$  \\
\hline

${\rm N}_{07}^{\alpha,\beta} $ & $\to$ & ${\rm N}_{11}^{\alpha}  $ & $
E_{1}(t)= e_1$ & $
E_{2}(t)= te_2$  \\
\hline

${\rm N}_{07}^{\alpha,\alpha} $ & $\to$ & ${\rm N}_{14}^{\alpha}  $ & $
E_{1}(t)= t e_1 - (t-1)t e_2$ & $
E_{2}(t)= (t-1)t^3 e_2$  \\
\hline

${\rm N}_{08}^{\alpha,\beta} $ & $\to$ & ${\rm N}_{01}  $ & $
E_{1}(t)= t e_1$ & $
E_{2}(t)= t^2 e_2$  \\
\hline

${\rm N}_{08}^{\alpha,\beta} $ & $\to$ & ${\rm N}_{09}^{\alpha,\beta}  $ & $
E_{1}(t)= e_1$ & $
E_{2}(t)= t^{-1}e_2$  \\
\hline

${\rm N}_{08}^{\alpha,\beta} $ & $\to$ & ${\rm N}_{14}^{\gamma}  $ & $
E_{1}(t)= t e_1 + t (\gamma-\beta)^{-1} e_2$ & $
E_{2}(t)= t^2 (\gamma - \beta)^{-1} e_2$  \\
\hline

${\rm N}_{09}^{\alpha,\beta} $ & $\to$ & ${\rm N}_{14}^{\beta}  $ & $
E_{1}(t)= t e_1 + t^{-1}e_2$ & $
E_{2}(t)= e_2$  \\
\hline

${\rm N}_{10}^{\alpha} $ & $\to$ & ${\rm N}_{04}  $ & $
E_{1}(t)= e_2$ & $
E_{2}(t)= te_1$  \\
\hline

${\rm N}_{10}^{\alpha} $ & $\to$ & ${\rm N}_{11}^{\alpha}  $ & $
E_{1}(t)= e_1$ & $
E_{2}(t)= te_2$  \\
\hline
 
${\rm N}_{10}^{\alpha} $ & $\to$ & ${\rm N}_{14}^{\beta\neq\alpha}  $ & $
E_{1}(t)= t(\alpha-\beta)^{-1}e_1 + t e_2$ & $
E_{2}(t)= t^2(\alpha-\beta)^{-2}e_1$  \\
\hline

${\rm N}_{11}^{\alpha} $ & $\to$ & ${\rm N}_{14}^{\alpha}  $ & $
E_{1}(t)= t e_1 - t^{-1}e_2$ & $
E_{2}(t)= e_2$  \\
\hline

${\rm N}_{12}^{0} $ & $\to$ & ${\rm N}_{02}  $ & $
E_{1}(t)= t^{-1}e_2$ & $
E_{2}(t)= -e_1$  \\
\hline

${\rm N}_{12}^{\alpha} $ & $\to$ & ${\rm N}_{06}^{\alpha^{-1}}  $ & $
E_{1}(t)= t^{-1}e_2$ & $
E_{2}(t)=  \alpha^{-1}e_1$  \\
\hline

${\rm N}_{12}^{\alpha} $ & $\to$ & ${\rm N}_{14}^{\beta}  $ & $
E_{1}(t)= te_1 + \alpha t e_2$ & $
E_{2}(t)= t^2 e_2$  \\
\hline

${\rm N}_{13}^{\alpha} $ & $\to$ & ${\rm N}_{06}^{0}  $ & $
E_{1}(t)= t^{-1}e_2$ & $
E_{2}(t)= e_1$  \\
\hline

${\rm N}_{13}^{\alpha} $ & $\to$ & ${\rm N}_{14}^{\alpha}  $ & $
E_{1}(t)= te_1$ & $
E_{2}(t)= t^2e_2$  \\
\hline
\end{longtable}

The primary non-degenerations of Novikov--Poisson algebras with non-trivial commutative associative multiplication are proven using the sets from the following table: 

 \begin{longtable} {|lcl|l|}
\hline
\multicolumn{3}{|c|}{\textrm{Non-degeneration}} & \multicolumn{1}{|c|}{\textrm{Arguments}}\\
\hline
\hline



%

${{\rm N}_{07}^{\alpha,\beta }}$  &$\not \to$&  
$\Big\{ \begin{array}{l}
{\rm N}_{06}^{\gamma}, {\rm N}_{02}, {\rm N}_{04},\\ 
     {\rm N}_{11}^{\gamma\neq\alpha, \beta}, {\rm N}_{09}^{\gamma,\delta} [\beta\neq \alpha]
     \end{array} \Big\}$&
${\mathcal R}= \left\{ \begin{array}{l}
c_{11}^1, c_{11}^2, c_{12}^2, c_{21}^2, c_{22}^2, c_{11}'^1, c_{11}'^2, c_{12}'^2, c_{21}'^2, c_{22}'^2 \in \mathbb{C},\\
c_{21}^2=c_{12}^2, c_{12}'^2=c_{21}'^2,\\
c_{11}'^1=\alpha  c_{11}^1,  c_{12}'^2=\beta  c_{12}^2, c_{22}'^2=\beta  c_{22}^2
\end{array} \right\}$\\
\hline

${{\rm N}_{07}^{\alpha,\alpha}}$ &$\not \to$&  
$ \Big\{\begin{array}{l}    
     {\rm N}_{06}^{0},
     {\rm N}_{09}^{(\beta,\gamma)\neq (\alpha,\alpha)},\\
     {\rm N}_{01},
     {\rm N}_{14}^{\beta\neq\alpha} \end{array}\Big\}$&
${\mathcal R}= \left\{ \begin{array}{l}
c_{11}^1, c_{11}^2, c_{12}^2, c_{21}^2, c_{22}^2, c_{11}'^1, c_{11}'^2, c_{12}'^2, c_{21}'^2, c_{22}'^2 \in \mathbb{C},\\
c_{21}^2=c_{12}^2, c_{11}'^1=\alpha  c_{11}^1, c_{11}'^2=\alpha  c_{11}^2,\\ c_{12}'^2=c_{21}'^2, c_{12}'^2=\alpha  c_{12}^2, c_{22}'^2=\alpha  c_{22}^2
\end{array} \right\}$\\
\hline

${\rm N}_{08}^{\alpha,\beta}$ &$\not \to$&  
$\Big\{\begin{array}{l} {\rm N}_{06}^{\gamma}, {\rm N}_{02}, 
     {\rm N}_{04},\\
     {{\rm N}_{11}^{\gamma},} {\rm N}_{09}^{(\gamma,\delta)\neq(\alpha,\beta)} \end{array}\Big\}$ &
${\mathcal R}= \left\{ \begin{array}{l}
c_{11}^1, c_{11}^2, c_{12}^2, c_{21}^2, c_{11}'^1, c_{11}'^2, c_{12}'^2, c_{21}'^2 \in \mathbb{C},\\
c_{12}^2=c_{11}^1, c_{21}^2=c_{11}^1, c_{11}'^1=\alpha  c_{11}^1,\\
c_{12}'^2=\beta  c_{11}^1, c_{21}'^2=\alpha  c_{11}^1
\end{array} \right\}$\\
\hline

${\rm N}_{09}^{\alpha,\beta}$ &$\not \to$&  
${\rm N}_{01}, 
{\rm N}_{14}^{\gamma\neq\beta}$ &
${\mathcal R}= \left\{ \begin{array}{l}
c_{11}^1, c_{11}^2, c_{12}^2, c_{21}^2, c_{11}'^1, c_{11}'^2, c_{12}'^2, c_{21}'^2 \in \mathbb{C}, \\
c_{12}^2=c_{11}^1, c_{21}^2=c_{11}^1, c_{11}'^1=\alpha  c_{11}^1,\\ c_{21}'^2=\alpha  c_{11}^1, c_{12}'^2=\beta  c_{11}^1, c_{11}'^2=\beta  c_{11}^2
\end{array} \right\}$\\
\hline

${\rm N}_{10}^{\alpha}$ &$\not \to$&  
$ {\rm N}_{02},  {\rm N}_{06}^{\beta},
     {{\rm N}_{09}^{\beta,\gamma},} {\rm N}_{11}^{\beta\neq\alpha}$ & 
${\mathcal R}= \left\{ \begin{array}{l}
c_{11}^1, c_{11}^2, c_{11}'^1, c_{11}'^2, c_{12}'^2, c_{21}'^2, c_{22}'^2\in \mathbb{C}, \\c_{11}'^1=\alpha  c_{11}^1, c_{12}'^2=c_{21}'^2, c_{11}^1 c_{21}'^2=-c_{11}^2 c_{22}'^2
\end{array} \right\}$\\
\hline

${\rm N}_{11}^{\alpha}$ &$\not \to$&  
${\rm N}_{01}, 
{\rm N}_{14}^{\beta\neq\alpha}$ & 
${\mathcal R}= \left\{ \begin{array}{l}
c_{11}^1, c_{11}^2, c_{11}'^1, c_{11}'^2\in \mathbb{C},
c_{11}'^1=\alpha  c_{11}^1, c_{11}'^2=\alpha  c_{11}^2
\end{array} \right\}$\\
\hline

${\rm N}_{12}^{\alpha}$ &$\not \to$&  
$\Big\{
\begin{array}{l}
{{\rm N}_{09}^{\beta,\gamma}, {\rm N}_{11}^{\beta},} {\rm N}_{04}, \\ {\rm N}_{06}^{\beta\neq\alpha^{-1}}, {\rm N}_{02}(\alpha\neq0)
\end{array}\Big\}$ & 
${\mathcal R}= \left\{ \begin{array}{l}
c_{11}^2, c_{11}'^1, c_{11}'^2, c_{12}'^2, c_{21}'^2\in \mathbb{C}, 
c_{11}'^1=\alpha  c_{12}'^2, c_{21}'^2=\alpha  c_{12}'^2
\end{array} \right\}$\\
\hline

${\rm N}_{13}^{\alpha}$ &$\not \to$&  
${\rm N}_{01}, {\rm N}_{14}^{\beta\neq\alpha}$ & 
${\mathcal R}= \left\{ \begin{array}{l}
c_{11}^2, c_{11}'^1, c_{11}'^2, c_{21}'^2 \in \mathbb{C},
c_{11}'^1=c_{21}'^2, c_{11}'^2=\beta  c_{11}^2
\end{array} \right\}$\\
\hline

\end{longtable}

\end{proof}

At this point, only the description of the closures of the orbits of the parametric families is missing. 
 
\begin{lemma}\label{th:degnovf}
The description of the closure of the orbit of the parametric families in
the variety of $2$-dimensional Novikov--Poisson algebras are given below:

\begin{longtable}{lcl}
$\overline{\{O({\rm N}_{06}^{*}) \}}$ & ${\supseteq}$ &
$ \Big\{\overline{\{O({\rm N}_{05})\}}, \  
\overline{\{O({\rm N}_{02})\}}, \  
\overline{\{O({\rm N}_{01})\}}, \ 
\overline{\{O({\mathbb C^8})\}} \Big\}$\\

  $\overline{\{O({\rm N}_{07}^{*})\}}$ & ${\supseteq}$ &
  $\Big\{\begin{array}{l}\overline{\{O({\rm N}_{14}^{\alpha}\}}, \ 
  \overline{\{O({\rm N}_{12}^{1}\}}, \ 
  \overline{\{O({\rm N}_{11}^{\alpha}\}}, \ 
  \overline{\{O({\rm N}_{10}^{\alpha}\}}, \ 
  \overline{\{O({\rm N}_{09}^{\alpha, \alpha}\}}, \  
  \overline{\{O({\rm N}_{08}^{\alpha, \alpha}\}}, \\ 
  \overline{\{O({\rm N}_{06}^1\}},  \ 
  \overline{\{O({\rm N}_{04})\}},  \ 
  \overline{\{O({\rm N}_{03})\}}, \ 
  \overline{\{O({\rm N}_{01})\} }, \ 
  \overline{\{O({\mathbb C^8})\}} \end{array}\Big\}$\\

$\overline{\{O({\rm N}_{08}^{*})\}}$ & ${\supseteq}$ &
$\Big\{\begin{array}{l}\overline{\{O({\rm N}_{14}^{\alpha}) \}}, \ 
\overline{\{O({\rm N}_{13}^{\alpha}) \}}, \
\overline{\{O({\rm N}_{12}^{\alpha}) \}}, \ 
\overline{\{O({\rm N}_{09}^{\alpha, \beta})\}}, \ 
\overline{\{O({\rm N}_{06}^{\alpha})\}}, \\ 
\overline{\{O({\rm N}_{05})\}}, \ 
\overline{\{O({\rm N}_{02})\}}, \ 
\overline{\{O({\rm N}_{01})\}}, \ 
\overline{\{O({\mathbb C^8})\}}\end{array}\Big\}$\\

$\overline{\{O({\rm N}_{09}^{*})\}}$ & ${\supseteq}$ & 
$\Big\{\begin{array}{l}\overline{
\{O({\rm N}_{14}^{\alpha})\}}, \ 
\overline{\{O({\rm N}_{13}^{\alpha})\}}, \ 
\overline{\{O({\rm N}_{06}^{\alpha})\}},  \
\overline{\{O({\rm N}_{05}) \}}, \\ 
\overline{\{O({\rm N}_{02})\}}, \ 
\overline{\{O({\rm N}_{01})\}},
\overline{\{O({\mathbb C^8})\}}\end{array}\Big\}$\\

$\overline{\{O({\rm N}_{10}^{*})\}}$ & ${\supseteq}$ & 
$\Big\{\begin{array}{l}\overline{\{O({\rm N}_{14}^{\alpha})\}}, \ 
\overline{\{O({\rm N}_{12}^{1})\}}, \ 
\overline{\{O({\rm N}_{11}^{\alpha})\}}, \   
\overline{\{O({\rm N}_{06}^1)\}}, \\ 
\overline{\{O({\rm N}_{04})\}}, \ 
\overline{\{O({\rm N}_{03})\}}, \ 
\overline{\{O({\rm N}_{01})\}}, 
\overline{\{O({\mathbb C^8})\}}\end{array}\Big\}$\\

$\overline{\{O({\rm N}_{11}^{*})\}}$ & ${\supseteq}$ & 
$\Big\{\overline{\{O({\rm N}_{14}^{\alpha})\}}, \ 
\overline{\{O({\rm N}_{04})\}}, \ 
\overline{\{O({\rm N}_{01})\}}, 
\overline{\{O({\mathbb C^8})\}}\Big\}$\\

$\overline{\{O({\rm N}_{12}^{*})\}}$  & ${\supseteq}$ &  
$\Big\{\begin{array}{l}\overline{\{O({\rm N}_{14}^{\alpha})\}}, \ 
\overline{\{O({\rm N}_{13}^{\alpha})\}}, \
\overline{\{O({\rm N}_{06}^{\alpha})\}}, \ 
\overline{\{O({\rm N}_{05})\}}, \\ 
\overline{\{O({\rm N}_{02})\}}, \ 
\overline{\{O({\rm N}_{01})\}}, 
\overline{\{O({\mathbb C^8})\}}\end{array}\Big\}$\\

$\overline{\{O({\rm N}_{13}^{*})\}}$  & ${\supseteq}$ & 
$\Big\{\overline{\{O({\rm N}_{14}^{\alpha})\}}, \  
\overline{\{O({\rm N}_{06}^0)\}}, \ 
\overline{\{O({\rm N}_{05})\}}, \ 
\overline{\{O({\rm N}_{01})\}},
\overline{\{O({\mathbb C^8})\}}\Big\}$\\

$\overline{\{O({\rm N}_{14}^{*})\}}$  & ${\supseteq}$ & 
$\Big\{\overline{\{O({\rm N}_{01})\}}, \ 
\overline{\{O({\mathbb C^8})\}}\Big\}$ .
    
\end{longtable}

\end{lemma}

\begin{proof}
Thanks to Theorem \ref{th:degnov} we have all necessarily degenerations.
The closure of the orbits of suitable families are given using the parametric bases and the parametric index included in following Table: 
\begin{longtable}{|lcl|ll|}
\hline
%

${\rm N}_{07}^{t^{-1}, t^{-1}} $ & $\to$ & ${\rm N}_{03}  $ & $
E_{1}(t)= t e_1$ & $
E_{2}(t)= t e_2$  \\
\hline

${\rm N}_{07}^{t(1+\alpha t^{-1}), \alpha} $ & $\to$ & ${\rm N}_{08}^{\alpha \alpha}  $ & $
E_{1}(t)= e_1 + e_2$ & $
E_{2}(t)= -t e_1$  \\
\hline

${\rm N}_{07}^{\alpha, t^{-1}} $ & $\to$ & ${\rm N}_{10}^{\alpha}  $ & $
E_{1}(t)= e_1$ & $
E_{2}(t)= t e_2$  \\
\hline
${\rm N}_{08}^{\alpha t^{-1}, t^{-1}} $ & $\to$ & ${\rm N}_{12}^{\alpha}  $ & $
E_{1}(t)= t e_1 - t^{2}e_2$ & $
E_{2}(t)= -t^3e_2$  \\
\hline

${\rm N}_{09}^{t^{-1}, \alpha t^{-1}} $ & $\to$ & ${\rm N}_{06}^{\alpha}  $ & $
E_{1}(t)= te_2$ & $
E_{2}(t)= te_1$  \\
\hline

${\rm N}_{09}^{t^{-1}, \alpha} $ & $\to$ & ${\rm N}_{13}^{\alpha}  $ & $
E_{1}(t)= t e_1 + t^{-1}e_2$ & $
E_{2}(t)= e_2$  \\
\hline
${\rm N}_{10}^{t^{-1}} $ & $\to$ & ${\rm N}_{03}  $ & $
E_{1}(t)= t e_1$ & $
E_{2}(t)= e_2$  \\
\hline

${\rm N}_{10}^{t^{-1}} $ & $\to$ & ${\rm N}_{12}^1  $ & $
E_{1}(t)= t e_1 + e_2$ & $
E_{2}(t)= t^2 e_1$  \\
\hline

${\rm N}_{11}^{t^{-1}} $ & $\to$ & ${\rm N}_{04}  $ & $
E_{1}(t)= te_1$ & $
E_{2}(t)= e_2$  \\
\hline

${\rm N}_{12}^{t^{-1}} $ & $\to$ & ${\rm N}_{13}^{\alpha}  $ & $
E_{1}(t)= t e_1 + \alpha t  e_2$ & $
E_{2}(t)= t^{2}e_2$  \\
\hline

${\rm N}_{13}^{t^{-1}} $ & $\to$ & ${\rm N}_{05}  $ & $
E_{1}(t)= t^{-1}e_2$ & $
E_{2}(t)= e_1$  \\
\hline
${\rm N}_{14}^{t^{-1}} $ & $\to$ & ${\rm N}_{01}  $ & $
E_{1}(t)= e_1$ & $
E_{2}(t)= t^{-1}e_2$  \\
\hline
\end{longtable}

 \begin{longtable} {|lcl|l|}
\hline
\multicolumn{3}{|c|}{\textrm{Non-degeneration}} & \multicolumn{1}{|c|}{\textrm{Arguments}}\\
\hline
\hline


$ {{\rm N}_{07}^{*}}$ &$\not \to$ & 
${\rm N}_{02}, 
{\rm N}_{06}^{\alpha\neq1}, 
{\rm N}_{09}^{\alpha,\beta\neq\alpha}
$  & 
${\mathcal R}= \left\{ \begin{array}{l}
c_{11}^1, c_{11}^2, c_{12}^2, c_{21}^2, c_{22}^2, c_{11}'^1, c_{11}'^2, c_{12}'^2, c_{21}'^2, c_{22}'^2 \in \mathbb{C},\\
c_{21}^2=c_{12}^2, c_{12}'^2=c_{21}'^2, c_{22}^2 c_{12}'^2=c_{12}^2 c_{22}'^2
\end{array} \right\}$\\
\hline

${\rm N}_{08}^{*}$ &$\not \to$&  
$
{\rm N}_{04}, 
{\rm N}_{10}^{\alpha}, {\rm N}_{11}^{\alpha}$ & 
${\mathcal R}= \left\{ \begin{array}{l}
c_{11}^1, c_{11}^2, c_{12}^2, c_{21}^2, c_{11}'^1, c_{11}'^2, c_{12}'^2, c_{21}'^2 \in \mathbb{C},\\
c_{12}^2=c_{11}^1, c_{21}^2=c_{11}^1, c_{11}'^1=c_{21}'^2
\end{array} \right\}$\\
\hline

$ {\rm N}_{09}^{*}$ &$\not \to$&  
$
{\rm N}_{08}^{\alpha, \beta}, 
{\rm N}_{12}^{\alpha}$ & 
${\mathcal R}= \left\{ \begin{array}{l}
c_{11}^1, c_{11}^2, c_{12}^2, c_{21}^2, c_{11}'^1, c_{11}'^2, c_{12}'^2, c_{21}'^2 \in \mathbb{C}, \\
c_{12}^2=c_{11}^1, c_{21}^2=c_{11}^1,c_{11}'^1=c_{21}'^2, c_{11}^2 c_{12}'^2=c_{11}^1 c_{11}'^2
\end{array} \right\}$\\
\hline

$ {\rm N}_{10}^{*}$ &$\not \to$&  
$
{\rm N}_{07}^{\alpha, \beta}, 
{\rm N}_{09}^{\alpha, \alpha}
$  
& 
${\mathcal R}= \left\{ \begin{array}{l}
c_{11}^1, c_{11}^2, c_{11}'^1, c_{11}'^2, c_{12}'^2, c_{21}'^2, c_{22}'^2\in \mathbb{C},\\
c_{12}'^2=c_{21}'^2, c_{11}^1 c_{21}'^2=-c_{11}^2 c_{22}'^2
\end{array} \right\}$\\
\hline 

${\rm N}_{11}^{*}$ &$\not \to$& 
$
{\rm N}_{06}^{1}
$
& 
${\mathcal R}= \left\{ \begin{array}{l}
c_{11}^1, c_{11}^2, c_{11}'^1, c_{11}'^2\in \mathbb{C},
c_{11}^1 c_{11}'^2=c_{11}^2 c_{11}'^1
\end{array} \right\}$\\
\hline

${\rm N}_{12}^{*}$ &$\not \to$&  
$
{\rm N}_{09}^{\alpha, \beta} 
$ & 
${\mathcal R}= \left\{ \begin{array}{l}
c_{11}^2, c_{11}'^1, c_{11}'^2, c_{12}'^2, c_{21}'^2\in \mathbb{C}, 
c_{11}'^1=c_{21}'^2
\end{array} \right\}$\\

\hline
${\rm N}_{13}^{*}$ &$\not \to$&  
${\rm N}_{02}, 
{\rm N}_{06}^{\alpha\neq0}, 
$  
& 
${\mathcal R}= \left\{ \begin{array}{l}
c_{11}^2, c_{11}'^1, c_{11}'^2, c_{21}'^2 \in \mathbb{C},
c_{11}'^1=c_{21}'^2
\end{array} \right\}$\\
\hline


\end{longtable}

\end{proof}

By Lemma \ref{th:degnov} and Lemma \ref{th:degnovf}, we have the following result that summarizes the geometric classification of 
$2$-dimensional Novikov--Poisson algebras.

\begin{theorem}\label{geo_NP}
The variety of $2$-dimensional Novikov--Poisson algebras  has two irreducible components,  corresponding to the   closures of the orbits of the  families of  Novikov--Poisson algebras ${\rm N}_{07}^{*}$ and ${\rm N}_{08}^{*}$.
\end{theorem}

 \subsubsection{The geometric classification of $2$-dimensional pre-Lie
Poisson algebras}
The geometric classification of $2$-dimensional Novikov--Poisson algebras is given in theorem \ref{geo_NP};
a description of all degenerations of $2$-dimensional pre-Lie algebras is given in \cite{burde09};
a description of all orbit closures of families of $2$-dimensional algebras is given in \cite{kv16}.
Summarizing, these cited results and corollary \ref{algprelieP} we have the following statement.

\begin{corollary}
The variety of $2$-dimensional pre-Lie Poisson algebras has four irreducible components,  corresponding to the  the Pre-Lie Poisson algebra ${\rm P}_{02}$ and  the  families of pre-Lie Poisson algebras
${\rm P}_{01}^{*},$ ${\rm N}_{07}^{*}$ and ${\rm N}_{08}^{*}$.
\end{corollary}

\medskip 






\section{ Commutative pre-Lie algebras}

Commutative pre-Lie algebras first appeared in \cite{mansuy} and after that, they appeared in 
\cite{Foissycom,Foissy,dots,mammez,lsyb}.
A description of free commutative pre-Lie algebras with one generators in terms of partitioned trees is given in \cite{Foissycom} and a relation between commutative pre-Lie algebras and $F$-manifold algebras is established in \cite{dots}.

\subsection{The algebraic classification of $2$-dimensional pre-Lie algebras}

\begin{definition}
An algebra  is called a pre-Lie algebra if it satisfies the identity 
\begin{longtable}{rcl}
$\left( x\circ   y\right) \circ   z-\left( y\circ   x\right) \circ   z$&$=$&$x\circ   \left( y\circ  
z\right) -y\circ   \left( x\circ   z\right).$
\end{longtable}%
 \end{definition}

The algebraic classification of  $2$-dimensional pre-Lie algebras can be obtained by direct verification from \cite{kv16}.

\begin{theorem}
\label{pre2}
Let ${\rm C}$ be a nonzero $2$-dimensional pre-Lie algebra. 
Then ${\rm C}$  is
isomorphic to  one and only one  of the following algebras:

\begin{longtable}{lclcllcllcllcl}

${\rm C}_{01}$&$:$&$ e_{1} \circ   e_{1}  $&$=$&$ e_{1} + e_{2},$&$ e_{2} \circ   e_{1} $&$=$&$e_{2}$\\

${\rm C}_{02}$&$:$&$ e_{1}\circ   e_{1} $&$=$&$ e_{1} + e_{2},$&$ e_{1}\circ   e_{2}  $&$=$&$e_{2}$\\

${\rm C}_{03}$&$:$&$ e_{1}\circ   e_{1} $&$=$&$e_{2}$\\

${\rm C}_{04}$&$:$&$ e_{2} \circ   e_{1} $&$=$&$e_{1}$\\

${\rm C}_{05}^{\alpha}$&$:$&$ e_{1}\circ   e_{1}  $&$=$&$e_{1},$&$  e_{1}\circ   e_{2} $&$=$&$ \alpha e_{2}$\\

${\rm C}_{06}^{\alpha}$&$:$&$e_{1} \circ   e_{1}  $&$=$&$e_{1},$&$  e_{1}  \circ   e_{2}  $&$=$&$ \alpha e_{2}, $&$ e_{2} \circ   e_{1} $&$=$&$e_{2}$\\

${\rm C}_{07}$&$:$&$ e_{1}\circ   e_{1}  $&$=$&$e_{1},$&$  e_{2}\circ   e_{2}$&$=$&$e_{2}$\\

${\rm C}_{08}$&$:$&$  e_{1}\circ   e_{1}  $&$=$&$e_{1},$&$  e_{1}\circ   e_{2} $&$=$&$ 2 e_{2},$&$ e_{2}\circ   e_{1}  $&$=$&$\frac{1}{2} e_{1} + e_2, $&$  e_{2}\circ   e_{2}$&$=$&$e_{2}$

\end{longtable}

\end{theorem}

\subsection{The algebraic classification of  $2$-dimensional commutative pre-Lie algebras}

\begin{definition}
A commutative pre-Lie algebra   is a vector space  equipped with 
a commutative associative multiplication $\cdot$
and     
a pre-Lie multiplication $\circ  .$
These two operations are required to satisfy the following identities:
 
\begin{longtable}{rcl}
$x\circ   \left( y\cdot z\right)$&$ =
$&$\left( x\circ   y\right) \cdot z+y\cdot \left( x\circ   z\right).$
\end{longtable}
\end{definition}

Commutative pre-Lie  algebras 
$({\rm C}, \cdot, \circ  )$ with zero commutative associative multiplication  are   pre-Lie algebras given in Theorem \ref{pre2}.  Hence, we are studying commutative pre-Lie algebras defined on every  commutative associative algebra from Theorem \ref{asocc2}.

\begin{definition}
Let $\left( \rm{A},\cdot \right) $ be a commutative associative
algebra. Define ${\rm Z}_{\rm CPL}^{2}\left( \rm{A},\rm{A}\right) $ to be
the set of all bilinear maps $\theta :\rm{A}\times \rm{A}%
\longrightarrow \rm{A}$ such that:%
\begin{longtable}{rcl}
$\theta \left( \theta \left( x,y\right) ,z\right) -\theta \left( x,\theta \left( y,z\right)
\right)$ & $=$&$\theta \left( \theta
\left( y,x\right) ,z\right) -\theta \left( y,\theta \left( x,z\right) \right),$ \\
$\theta \left( x,y\cdot z\right)$ & $=$&$\theta \left( x,y\right) \cdot z+y\cdot
\theta \left( x,z\right).$
\end{longtable}%
If $\theta \in {\rm Z}_{\rm CPL}^{2}\left( \rm{A},%
\rm{A}\right) $, then $\left( \rm{A},\cdot ,\circ   \right) $ is a
commutative pre-Lie algebra where $x\circ   y=\theta \left( x,y\right) $ for
all $x,y\in \rm{A}$.
\end{definition}

\subsubsection{Commutative pre-Lie algebras defined on  ${\rm A}_{01}$}

Since ${\rm Z}_{\rm CPL}^{2}({\rm A}_{01}, {\rm A}_{01}) = \left\{0\right\}$, then there is only the trivial structure  ${\rm C}_{09} = {\rm A}_{01}$.

\subsubsection{Commutative pre-Lie algebras defined on  ${\rm A}_{02}$}

From the computation of ${\rm Z}_{\rm CPL}^{2}({\rm A}_{02}, {\rm A}_{02})$, the commutative pre-Lie structures defined on ${\rm A}_{02}$ are of the form:
$$\left\{ 
\begin{tabular}{lcllcl}
$e_1 \cdot e_1  $&$= $&$e_1,$&$ e_1 \cdot e_2  $&$=$&$e_2$ \\ 
$ e_{1}\circ   e_{2} $&$=$&$ \alpha_1 e_{2}$%
\end{tabular}%
\right. 
\qquad \qquad \qquad
\left\{ 
\begin{tabular}{lcllcl}
$e_1 \cdot e_1  $&$=$&$ e_1,$&$ e_1 \cdot e_2  $&$=$&$e_2$ \\ $ e_{2} \circ   e_{2}$&$= $&$\alpha_2 e_{2}$%
\end{tabular}%
\right. $$

If $\alpha_2=0$, then it is included in the first family. So we may assume $\alpha_2\neq0$. Further, if we choose $\xi = {\alpha_2}^{-1}$, 
we then obtain the following non-isomorphic algebras:
$${\rm C}_{10}^{\alpha} : = \left\{ 
\begin{tabular}{lcllcl}
$e_1 \cdot e_1 $&$ =$&$ e_1,$&$ e_1 \cdot e_2  $&$=$&$e_2$ \\ 
$ e_{1} \circ   e_{2}$&$ =$&$ \alpha e_{2}$%
\end{tabular}%
\right. 
\qquad \qquad \qquad
{\rm C}_{11} : = \left\{ 
\begin{tabular}{lcllcl}
$e_1 \cdot e_1  $&$=$&$ e_1,$ & $e_1 \cdot e_2  $&$=$&$e_2$ \\ 
$ e_{2} \circ   e_{2} $&$=$&$ e_{2}$%
\end{tabular}%
\right. $$

\subsubsection{Commutative pre-Lie algebras defined on  ${\rm A}_{03}$}

The commutative pre-Lie structures defined on ${\rm A}_{03}$ are:
$$\left\{ 
\begin{tabular}{lcl}
$e_1 \cdot e_1 $&$=$&$ e_1$ \\ 
$e_{1}\circ   e_{2} $&$=$&$ \alpha_1 e_{2}$%
\end{tabular}%
\right.
\qquad \qquad \qquad
\left\{ 
\begin{tabular}{lcl}
$e_1 \cdot e_1 $&$=$&$ e_1$ \\ 
$e_{2}\circ   e_{2} $&$=$&$ \alpha_2 e_{2}$%
\end{tabular}%
\right.$$

As above, we may assume $\alpha_2\neq0$. Then we obtain the following non-isomorphic algebras:
$${\rm C}_{12}^{\alpha} : = \left\{ 
\begin{tabular}{lcl}
$e_1 \cdot e_1  $&$=$&$ e_1$ \\
$ e_{1}\circ   e_{2} $&$=$&$ \alpha e_{2}$%
\end{tabular}%
\right. 
\qquad \qquad \qquad
{\rm C}_{13} : = \left\{ 
\begin{tabular}{lcl}
$e_1 \cdot e_1  $&$=$&$ e_1$ \\ 
$e_{2}\circ   e_{2} $&$=$&$ e_{2}$%
\end{tabular}%
\right. $$

\subsubsection{Commutative pre-Lie algebras defined on  ${\rm A}_{04}$}

The commutative pre-Lie structures defined on ${\rm A}_{04}$ are:
$$\left\{ 
\begin{tabular}{lcllcl}
$e_1 \cdot e_1  $&$=$&$ e_2$ \\ 
$e_{1}\circ   e_{1} $&$=$&$ \alpha_1 e_{1} + \beta_1 e_2,$&$  e_{1}\circ   e_{2} $&$=$&$ 2 \alpha_1 e_{2}$%
\end{tabular}%
\right.
\qquad 
\left\{ 
\begin{tabular}{lcllcl}
$e_1 \cdot e_1  $&$=$&$ e_2$ \\ 
$e_{1}\circ   e_{1} $&$=$&$ \alpha_2 e_{1} + \beta_2 e_2,$&$  e_{1}\circ   e_{2} $&$=$&$ 2 \alpha_2 e_{2}$ \\
$e_{2}\circ   e_{1} $&$=$&$ \gamma_2 e_1 + \alpha_2 e_{2}, $&$ e_{2}\circ   e_{2} $&$=$&$ 2\gamma_2 e_{2}$%
\end{tabular}%
\right.
$$

Note that if $\alpha_1=0$, then it is included in the second one. So assume $\alpha_1\neq0$. Then, we can choose $\xi={\alpha_1}^{-1}$ and  $\nu= -\beta_1\alpha_1^{-2}$ to simplify the first family into:
$${\rm C}_{14} : = \left\{ 
\begin{tabular}{ll}
$e_1 \cdot e_1  = e_2,$ \\ 
$e_{1}\circ   e_{1} = e_{1},$&$e_{1}\circ   e_{2} = 2 e_{2}.$%
\end{tabular}%
\right.$$

Now, for the second family we have the following cases.

\begin{itemize}
    \item If $\gamma_2\neq0$, then choose $\xi=-{\sqrt{\gamma_{2}^{-1}}}$ and $\nu = \alpha_2 \sqrt{\gamma_2^{-3}}-\sqrt{\gamma_2^{-3} \left(\alpha_2^2-\beta_2 \gamma_2\right)}$ to obtain the family:
    
$${\rm C}_{15}^{\alpha} : = \left\{ 
\begin{tabular}{lcllcllcllcl}
$e_1 \cdot e_1  $&$=$&$ e_2$ \\ 
$e_{1}\circ   e_{1} $&$=$&$ \alpha e_{1},$&$  e_{1}\circ   e_{2} $&$=$&$ 2 \alpha e_{2},$ &
$e_{2}\circ   e_{1} $&$=$&$ e_1 + \alpha e_{2}, $&$  e_{2}\circ   e_{2} $&$=$&$ 2 e_{2}$%
\end{tabular}%
\right.$$
    
    Moreover, ${\rm C}_{15}^{\alpha} \cong {\rm C}_{15}^{-\alpha}$, setting $\xi=-1$ and $\nu=0$.

    \item If $\gamma_2=0$, then if $\alpha_2\neq0$, by choosing $\xi={\alpha_2}^{-1}$ and $\nu = - \beta_2(2\alpha_{2}^2)^{-1}$, we obtain the algebra:
    
$${\rm C}_{16} : = \left\{ 
\begin{tabular}{lcllcllcl}
$e_1 \cdot e_1  $&$=$&$ e_2$ \\ 
$e_{1}\circ   e_{1} $&$=$&$ e_{1},$&$ e_{1}\circ   e_{2} $&$=$&$ 2 e_{2},$ &  $e_{2}\circ   e_{1} $&$=$&$ e_{2}$%
\end{tabular}%
\right.$$

Otherwise, we have the parametric family :
$${\rm C}_{17}^{\alpha} : = \left\{ 
\begin{tabular}{lcl}
$e_1 \cdot e_1  $&$=$&$ e_2$ \\ 
$e_{1}\circ   e_{1} $&$=$&$ \alpha e_{2}$%
\end{tabular}%
\right.$$

\end{itemize}

\begin{theorem}
\label{2dim PRE} 
Let $\left( {\rm C},\cdot , \circ   \right) $ be a nonzero $2$-dimensional commutative pre-Lie algebra. Then ${\rm C}$ is isomorphic to one
pre-Lie algebra listed in Theorem \ref{pre2} or to one algebra listed below: 

\begin{longtable}{lcl}

 ${\rm C}_{09}$&$:$&$ e_1 \cdot e_1 = e_1, \ \ e_2 \cdot e_2 = e_2$\\ 

 ${\rm C}_{10}^{\alpha} $&$:$&$ \left\{ 
\begin{tabular}{lcllcl}
$e_1 \cdot e_1  $&$=$&$ e_1, $&$ e_1 \cdot e_2  $&$=$&$e_2$ \\ 
$ e_{1} \circ   e_{2} $&$=$&$ \alpha e_{2}$%
\end{tabular}%
\right. $\\

 ${\rm C}_{11}$&$ : $&$\left\{ 
\begin{tabular}{lcllcl}
$e_1 \cdot e_1  $&$=$&$ e_1,$&$ e_1 \cdot e_2  $&$=$&$e_2$ \\ 
$ e_{2} \circ   e_{2} $&$=$&$ e_{2}$%
\end{tabular}%
\right.$\\

 ${\rm C}_{12}^{\alpha} $&$: $&$\left\{ 
\begin{tabular}{lcl}
$e_1 \cdot e_1  $&$=$&$ e_1$ \\
$ e_{1}\circ   e_{2} $&$=$&$ \alpha e_{2}$%
\end{tabular}%
\right.$\\

 ${\rm C}_{13}$&$ :$&$ \left\{ 
\begin{tabular}{lcl}
$e_1 \cdot e_1  $&$=$&$ e_1$ \\ 
$e_{2}\circ   e_{2} $&$=$&$ e_{2}$%
\end{tabular}%
\right.$\\

 ${\rm C}_{14}$&$ :$&$ \left\{ 
\begin{tabular}{lcllcllcllcl}
$e_1 \cdot e_1 $&$ =$&$ e_2$ \\ 
$e_{1}\circ   e_{1} $&$=$&$ e_{1},$&$  e_{1}\circ   e_{2} $&$=$&$ 2 e_{2}$%
\end{tabular}%
\right.$\\

 ${\rm C}_{15}^{\alpha}$&$ :$&$ \left\{ 
\begin{tabular}{lcllcllcllcl}
$e_1 \cdot e_1 $&$ =$&$ e_2$ \\ 
$e_{1}\circ   e_{1}$&$ = $&$\alpha e_{1},$&$  e_{1}\circ   e_{2} $&$=$&$ 2 \alpha e_{2},$ &$e_{2}\circ   e_{1} $&$=$&$ e_1 + \alpha e_{2}, $&$ e_{2}\circ   e_{2} $&$= $&$2 e_{2}$%
\end{tabular}%
\right.$\\

${\rm C}_{16}$&$ :$&$ \left\{ 
\begin{tabular}{lcllcllcl}
$e_1 \cdot e_1  $&$=$&$ e_2$ \\ 
$e_{1}\circ   e_{1} $&$= $&$e_{1},$&$  e_{1}\circ   e_{2}$&$ =$&$ 2 e_{2},$& $e_{2}\circ   e_{1} $&$= $&$e_{2}$%
\end{tabular}%
\right.$\\

 ${\rm C}_{17}^{\alpha} $&$:$&$ \left\{ 
\begin{tabular}{lcl}
$e_1 \cdot e_1  $&$= $&$e_2$ \\ 
$e_{1}\circ   e_{1} $&$= $&$\alpha e_{2}$%
\end{tabular}%
\right.$

\end{longtable}

Between these algebras, the only non-trivial isomorphism is ${\rm C}_{15}^{\alpha} \cong {\rm C}_{15}^{-\alpha}$.

\end{theorem}

\subsection{Degenerations of  $2$-dimensional commutative pre-Lie algebras}

\begin{theorem}

\label{th:degcpl}
The graph of primary degenerations and non-degenerations of the variety of $2$-dimensional commutative pre-Lie algebras is given in Figure 1, where the numbers on the right side are the dimensions of the corresponding orbits.

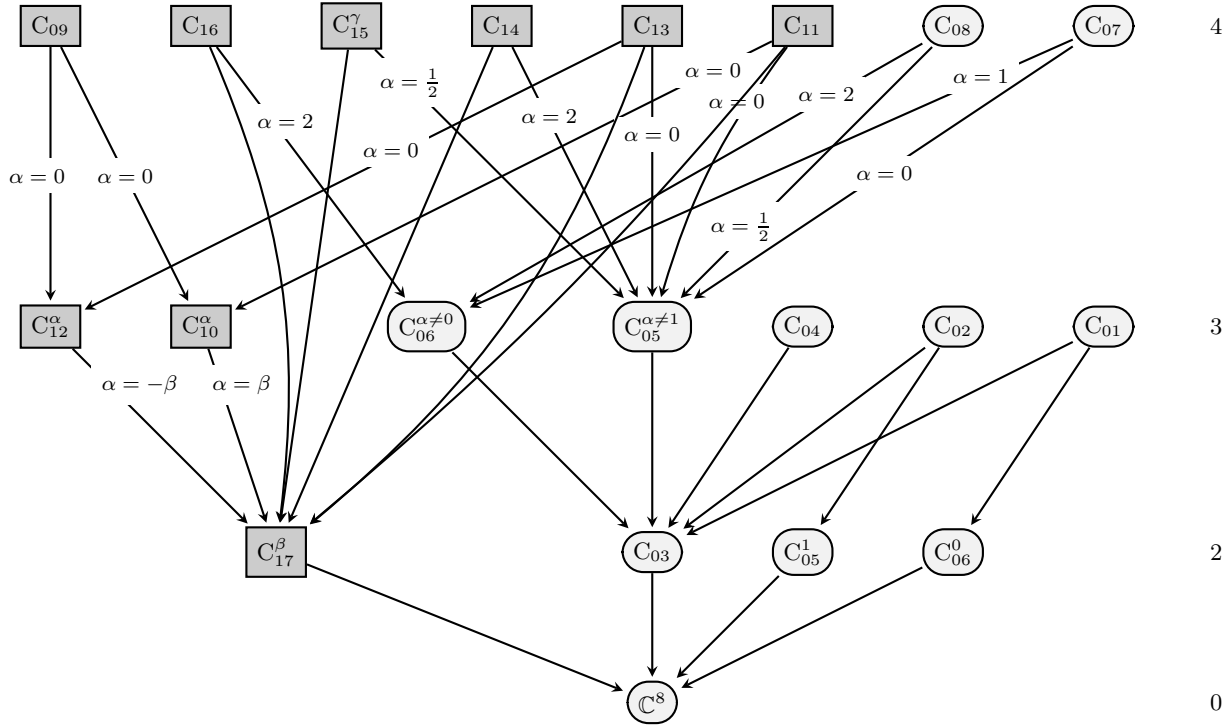
\begin{center}
	
	\begin{tikzpicture}[->,>=stealth,shorten >=0.05cm,auto,node distance=1.3cm,
	thick,
	main node/.style={rectangle,draw,fill=gray!10,rounded corners=1.5ex,font=\sffamily \scriptsize \bfseries },
	rigid node/.style={rectangle,draw,fill=black!20,rounded corners=0ex,font=\sffamily \scriptsize \bfseries }, 
	poisson node/.style={rectangle,draw,fill=black!20,rounded corners=0ex,font=\sffamily \scriptsize \bfseries },
	ac node/.style={rectangle,draw,fill=black!20,rounded corners=0ex,font=\sffamily \scriptsize \bfseries },
	lie node/.style={rectangle,draw,fill=black!20,rounded corners=0ex,font=\sffamily \scriptsize \bfseries },
	style={draw,font=\sffamily \scriptsize \bfseries }]

	\node (3) at (2.5,9) {$4$};
	\node (2) at (2.5,5) {$3$};
	\node (1) at (2.5,2) {$2$};
	\node (0)  at (2.5,0) {$0$};

    \node[main node] (c20) at (-5,0) {${\mathbb C^8}$};
    
    \node[main node] (c260) at (-1,2) {${\rm C}_{06}^{0}$};
    \node[main node] (c251) at (-3,2) {${\rm C}_{05}^{1}$};
    \node[main node] (c23) at (-5,2) {${\rm C}_{03}$};
    \node[ac node] (c217a) at (-10,2) {${\rm C}_{17}^{\beta}$};
    
    \node[main node] (c21) at (1,5) {${\rm C}_{01}$};
    \node[main node] (c22) at (-1,5) {${\rm C}_{02}$};
    \node[main node] (c24) at (-3,5) {${\rm C}_{04}$};
    \node[main node] (c25a) at (-5,5) {${\rm C}_{05}^{\alpha\neq1}$};
    \node[main node] (c26a) at (-8,5) {${\rm C}_{06}^{\alpha\neq0}$};
    \node[ac node] (c210) at (-11,5) {${\rm C}_{10}^{\alpha}$};
    \node[ac node] (c212) at (-13,5) {${\rm C}_{12}^{\alpha}$};
    
    \node[main node] (c27) at (1,9) {${\rm C}_{07}$};
    \node[main node] (c28) at (-1,9) {${\rm C}_{08}$};
    \node[ac node] (c211) at (-3,9) {${\rm C}_{11}$};
    \node[ac node] (c213) at (-5,9) {${\rm C}_{13}$};
    \node[ac node] (c214) at (-7,9) {${\rm C}_{14}$};
    \node[ac node] (c215) at (-9,9) {${\rm C}_{15}^{\gamma}$};
    \node[ac node] (c216) at (-11,9) {${\rm C}_{16}$};
    \node[ac node] (c29) at (-13,9) {${\rm C}_{09}$};
    
	\path[every node/.style={font=\sffamily\small}]

    (c28) edge [bend left=0] node[above=-20, right=-40, fill=white]{\tiny $\alpha=\frac{1}{2}$}  (c25a)

    (c28) edge [bend left=0] node[above=30, right=35, fill=white]{\tiny $\alpha=2$}  (c26a)    
    
    (c27) edge [bend left=0] node[above=37, right=65, fill=white]{\tiny $\alpha=1$}  (c26a)
  (c27) edge [bend left=0] node[above=0, right=-14, fill=white]{\tiny $\alpha=0$}  (c25a)

    (c215) edge [bend left=0] node[above=35, right=-50, fill=white]{\tiny $\alpha=\frac{1}{2}$}  (c25a)
 (c215) edge [bend left=0] node{}  (c217a)

    (c214) edge [bend left=0] node{}  (c217a)
    (c214) edge [bend left=0] node[above=22, right=-25, fill=white]{\tiny $\alpha=2$}  (c25a)
    (c216) edge [bend left=15] node{}  (c217a)
    (c216) edge [bend left=0] node[above=20, right=-25, fill=white]{\tiny $\alpha=2$}  (c26a)
    
    (c29) edge [bend left=0] node[above=0, right=-15, fill=white]{\tiny $\alpha=0$}  (c210)
    (c29) edge [bend left=0] node[above=0, right=-20, fill=white]{\tiny $\alpha=0$}  (c212)

    (c211) edge [bend left=-10] node[above=25, right=-5, fill=white]{\tiny $\alpha=0$}  (c25a)
    (c211) edge [bend left=0] node[above=40, right=65, fill=white]{\tiny $\alpha=0$}  (c210)
    (c211) edge [bend left=5] node{}  (c217a)
    
    (c213) edge [bend left=0] node[above=10, right=0, fill=white]{\tiny $\alpha=0$}  (c212)
    (c213) edge [bend left=0] node[above=15, right=-15, fill=white]{\tiny $\alpha=0$}  (c25a)
    (c213) edge [bend left=15] node{}  (c217a)

    (c212) edge [bend left=0] node[above=20, right=-27, fill=white]{\tiny $\alpha=-\beta$}  (c217a)
    (c210) edge [bend left=0] node[above=20, right=-14, fill=white]{\tiny $\alpha=\beta$} (c217a)
    (c22) edge [bend left=0] node{}  (c251)
    (c21) edge [bend left=0] node{}  (c260)
    (c26a) edge [bend left=0] node{}  (c23)
    (c25a) edge [bend left=0] node{}  (c23)
    (c24) edge [bend left=0] node{}  (c23)
    (c22) edge [bend left=0] node{}  (c23)
    (c21) edge [bend left=0] node{}  (c23)
    
    (c217a) edge [bend left=0] node{}  (c20)
    (c260) edge [bend left=0] node{}  (c20)
    (c251) edge [bend left=0] node{}  (c20)
    (c23) edge [bend left=0] node{}  (c20)

    ;

	\end{tikzpicture}

{\tiny 
\begin{itemize}
\noindent Legend:
\begin{itemize}
    \item[--] Round nodes: commutative pre-Lie algebras with trivial commutative associative multiplication $\cdot$.
    \item[--] Squared nodes: commutative pre-Lie algebras with non-trivial commutative associative multiplication $\cdot$.

\end{itemize}
\end{itemize}}

{Figure 1.}  Graph of primary degenerations and non-degenerations.	
\end{center}
\end{theorem}

\begin{proof} 
The dimensions of the orbits are deduced by computing the algebra of derivations.
The primary degenerations are proven using the parametric bases included in following Table:

    \begin{longtable}{|lcl|ll|}
\hline
\multicolumn{3}{|c|}{\textrm{Degeneration}}  & \multicolumn{2}{|c|}{\textrm{Parametrized basis}} \\
\hline
\hline

${\rm C}_{01} $ & $\to$ & ${\rm C}_{03}  $ & $
E_{1}(t)= te_1$ & $
E_{2}(t)= t^2 e_2$  \\
\hline

${\rm C}_{01} $ & $\to$ & ${\rm C}_{06}^{0}  $ & $
E_{1}(t)= e_1$ & $
E_{2}(t)= t^{-1}e_2$  \\
\hline

${\rm C}_{02} $ & $\to$ & ${\rm C}_{03}  $ & $
E_{1}(t)= te_1$ & $
E_{2}(t)= t^2e_2$  \\
\hline

${\rm C}_{02} $ & $\to$ & ${\rm C}_{05}^{1}  $ & $
E_{1}(t)= e_1$ & $
E_{2}(t)= t^{-1}e_2$  \\
\hline

${\rm C}_{04} $ & $\to$ & ${\rm C}_{03}  $ & $
E_{1}(t)= e_1 + t e_2$ & $
E_{2}(t)= - t^{2}e_2$  \\
\hline

${\rm C}_{05}^{\alpha\neq1} $ & $\to$ & ${\rm C}_{03}  $ & $
E_{1}(t)= t e_1 + t^{-1}(\alpha-1)^{-1}e_2$ & $
E_{2}(t)= e_2$  \\
\hline

${\rm C}_{06}^{\alpha\neq0} $ & $\to$ & ${\rm C}_{03}  $ & $
E_{1}(t)= t e_1 + t^{-1} \alpha^{-1} e_2$ & $
E_{2}(t)= e_2$  \\
\hline

${\rm C}_{07} $ & $\to$ & ${\rm C}_{05}^{0}  $ & $
E_{1}(t)= e_1$ & $
E_{2}(t)= te_2$  \\
\hline

${\rm C}_{07} $ & $\to$ & ${\rm C}_{06}^1  $ & $
E_{1}(t)= e_1 + e_2$ & $
E_{2}(t)= te_2$  \\
\hline

${\rm C}_{08} $ & $\to$ & ${\rm C}_{05}^{1/2}  $ & $
E_{1}(t)= e_2$ & $
E_{2}(t)= - \frac{t}{2} e_1 + t e_2$  \\
\hline

${\rm C}_{08} $ & $\to$ & ${\rm C}_{06}^{2}  $ & $
E_{1}(t)= e_1$ & $
E_{2}(t)= te_2$  \\
\hline

${\rm C}_{09} $ & $\to$ & ${\rm C}_{10}^0  $ & $
E_{1}(t)= e_1+ e_2$ & $
E_{2}(t)= te_1$  \\
\hline

${\rm C}_{09} $ & $\to$ & ${\rm C}_{12}^0  $ & $
E_{1}(t)= e_1$ & $
E_{2}(t)= te_2$  \\
\hline

${\rm C}_{11} $ & $\to$ & ${\rm C}_{05}^0  $ & $
E_{1}(t)= e_2$ & $
E_{2}(t)= te_1$  \\
\hline

${\rm C}_{11} $ & $\to$ & ${\rm C}_{10}^0  $ & $
E_{1}(t)= e_1$ & $
E_{2}(t)= te_2$  \\
\hline

${\rm C}_{13} $ & $\to$ & ${\rm C}_{05}^0  $ & $
E_{1}(t)= e_2$ & $
E_{2}(t)= te_1$  \\
\hline
 
 ${\rm C}_{13} $ & $\to$ & ${\rm C}_{12}^0  $ & $
E_{1}(t)= e_1$ & $
E_{2}(t)= te_2$  \\
\hline
 
${\rm C}_{14} $ & $\to$ & ${\rm C}_{05}^2  $ & $
E_{1}(t)= e_1$ & $
E_{2}(t)= t^{-1}e_2$  \\
\hline

${\rm C}_{15}^{\alpha} $ & $\to$ & ${\rm C}_{05}^{1/2}  $ & $
E_{1}(t)= \frac{1}{2}e_2$ & $
E_{2}(t)= te_1$  \\
\hline

${\rm C}_{16} $ & $\to$ & ${\rm C}_{06}^2  $ & $
E_{1}(t)= e_1$ & $
E_{2}(t)= t^{-1}e_2$  \\
\hline

${\rm C}_{10}^{\alpha} $ & $\to$ & ${\rm C}_{17}^{\alpha}  $ & $
E_{1}(t)= t e_1 + e_2$ & $
E_{2}(t)= t e_2$  \\
\hline

${\rm C}_{12}^{\alpha} $ & $\to$ & ${\rm C}_{17}^{-\alpha}  $ & $
E_{1}(t)= t e_1 + e_2$ & $
E_{2}(t)= t^2 e_1$  \\
\hline

${\rm C}_{11} $ & $\to$ & ${\rm C}_{17}^{\alpha}  $ & $
E_{1}(t)= te_1+ \beta t e_2$ & $
E_{2}(t)= \beta t^2 e_2$  \\
\hline

${\rm C}_{13} $ & $\to$ & ${\rm C}_{17}^{\alpha}  $ & $
E_{1}(t)= -te_1+ \beta t e_2$ & $
E_{2}(t)= \beta t^2 e_2$  \\
\hline

${\rm C}_{14} $ & $\to$ & ${\rm C}_{17}^{\alpha}  $ & $
E_{1}(t)= t e_1 + \alpha t e_2$ & $
E_{2}(t)= t^2 e_2$  \\
\hline

${\rm C}_{15}^{\alpha} $ & $\to$ & ${\rm C}_{17}^{\beta}  $ & $
E_{1}(t)= t e_1 - (\alpha t - t\sqrt{\alpha^2+\beta})e_2 $ & $
E_{2}(t)= t^2 e_2$  \\
\hline

${\rm C}_{16} $ & $\to$ & ${\rm C}_{17}^{\alpha}  $ & $
E_{1}(t)= te_1 + \frac{\alpha t}{2}e_2 $ & $
E_{2}(t)= t^2 e_2$  \\
\hline

    \end{longtable}
     
The primary non-degenerations are proven using the sets from the following table: 

    \begin{longtable}{|lcl|l|}
\hline
\multicolumn{3}{|c|}{\textrm{Non-degeneration}} & \multicolumn{1}{|c|}{\textrm{Arguments}}\\
\hline
\hline

$    {\rm C}_{01}$ &$\not \to$&  ${\rm C}_{05}^{1} $ &
${\mathcal R}= \left\{ \begin{array}{l}
c_{11}'^1, c_{11}'^2, c_{21}'^2\in\mathbb{C},
c_{11}'^1=c_{21}'^2
\end{array} \right\}$\\
\hline

${\rm C}_{02}$ &$\not \to$&  ${\rm C}_{06}^{0}$&
${\mathcal R}= \left\{ \begin{array}{l}
c_{11}'^{1}, c_{11}'^{2}, c_{12}'^{2}\in\mathbb{C},
c_{12}'^2=c_{11}'^1
\end{array} \right\}$\\
\hline

${\rm C}_{04}$ &$\not \to$&  ${\rm C}_{06}^{0}, {\rm C}_{05}^{1}$&
${\mathcal R}= \left\{ \begin{array}{l}
c_{11}'^{1}, c_{11}'^{2}, c_{21}'^{1}, c_{21}'^{2}\in \mathbb{C}, \\
c_{11}'^1=-c_{21}'^2, -c_{11}'^2 c_{21}'^1=(c_{11}'^1)^2 
\end{array} \right\}$\\
\hline

${\rm C}_{05}^{\alpha\neq1}$ &$\not \to$&  ${\rm C}_{05}^{1}, {\rm C}_{06}^{0}$ &
${\mathcal R}= \left\{ \begin{array}{l}
c_{11}'^{1}, c_{11}'^2, c_{12}'^2 \in \mathbb{C},
c_{12}'^2=\alpha  c_{11}'^1
\end{array} \right\}$\\
\hline

${\rm C}_{06}^{\alpha\neq0}$ &$\not \to$&  ${\rm C}_{05}^{1}, {\rm C}_{06}^{0}$  &
${\mathcal R}= \left\{ \begin{array}{l}
c_{11}'^1, c_{11}'^2, c_{12}'^2, c_{21}'^2 \in \mathbb{C}, 
c_{12}'^2=\alpha  c_{11}'^1, c_{21}'^2=c_{11}'^1
\end{array} \right\}$\\
\hline

$ {\rm C}_{07}$ &$\not \to$& $\begin{array}{l}
          {\rm C}_{06}^{\alpha\neq1}, {\rm C}_{05}^{\alpha\neq0}, {\rm C}_{04}  \\
     \end{array} $&
${\mathcal R}= \left\{ \begin{array}{l}
c_{11}'^1, c_{11}'^2, c_{12}'^2, c_{21}'^2, c_{22}'^2 \in \mathbb{C},
c_{21}'^2=c_{12}'^2
\end{array} \right\}$\\
\hline

$ {\rm C}_{08}$ &$\not \to$& $\begin{array}{l}
          {\rm C}_{06}^{\alpha\neq2}, {\rm C}_{05}^{\alpha\neq\frac{1}{2}}, {\rm C}_{04} \\
     \end{array}$&
${\mathcal R}= \left\{ \begin{array}{l}
c_{11}'^1, c_{11}'^2, c_{12}'^2, c_{21}'^1, c_{21}'^2, c_{22}^2 \in \mathbb{C},\\
2 c_{21}'^1=c_{22}'^2,  c_{11}'^1=c_{21}'^2, 2 c_{11}'^1=c_{12}'^2
\end{array} \right\}$\\
\hline

$ {\rm C}_{10}^{\alpha}$ &$\not \to$& $\begin{array}{l}
          {\rm C}_{17}^{\beta\neq \alpha}, {\rm C}_{06}^0,  
          {\rm C}_{05}^1, {\rm C}_{03} 
     \end{array}$  &
${\mathcal R}= \left\{ \begin{array}{l}
c_{11}^1, c_{11}^2, c_{12}^2, c_{21}^2, c_{11}'^2, c_{12}'^2 \in \mathbb{C},
c_{12}^2=c_{11}^1, \\c_{21}^2=c_{11}^1,
c_{12}'^2=\alpha  c_{11}^1, c_{11}'^2=\alpha  c_{11}^2
\end{array} \right\}$\\
\hline

$   {\rm C}_{12}^{\alpha}$ &$\not \to$& $\begin{array}{l}
          {\rm C}_{17}^{\beta\neq-\alpha}, {\rm C}_{06}^{0}, 
          {\rm C}_{05}^{1}, {\rm C}_{03}  
     \end{array}$ &
${\mathcal R}= \left\{ \begin{array}{l}
c_{11}^1, c_{11}^2, c_{11}'^2, c_{12}'^2\in \mathbb{C},
c_{12}'^2=\alpha  c_{11}^1, c_{11}'^2=-\alpha  c_{11}^2
\end{array} \right\}$\\
\hline

$ {\rm C}_{09}$ &$\not \to$&  $\begin{array}{l}
          {\rm C}_{17}^{\alpha\neq0}, {\rm C}_{05}^{\alpha}, {\rm C}_{06}^{\alpha}, {\rm C}_{03} \\
     \end{array}  $&
${\mathcal R}= \left\{ \begin{array}{l}
c_{11}^1, c_{11}^2, c_{12}^2, c_{21}^2, c_{22}^2 \in \mathbb{C},
c_{12}^2=c_{21}^2
\end{array} \right\}$\\
\hline

$   {\rm C}_{11}$ &$\not \to$&  $\begin{array}{l}
                    {\rm C}_{04},
{\rm C}_{05}^{\alpha\neq0},
{\rm C}_{06}^{\alpha}, {\rm C}_{12}^{\alpha}, {\rm C}_{10}^{\alpha\neq0} 
     \end{array} $ &
${\mathcal R}= \left\{ \begin{array}{l}
c_{11}^1, c_{11}^2, c_{12}^2, c_{21}^2, c_{11}'^2, c_{12}'^2, c_{21}'^2, c_{22}'^2 \in \mathbb{C},\\ c_{12}^2=c_{11}^1,  c_{21}^2=c_{11}^1, c_{12}'^2=c_{21}'^2,\\ 
(c_{12}'^2)^2=c_{11}'^2 c_{22}'^2, c_{11}^2 c_{22}'^2=c_{11}^1 c_{21}'^2
\end{array} \right\}$\\
\hline

$ {\rm C}_{13}$ &$\not \to$&  $\begin{array}{l}
               {\rm C}_{04}, {\rm C}_{05}^{\alpha\neq0}, {\rm C}_{06}^{\alpha}, {\rm C}_{12}^{\alpha\neq0}, {\rm C}_{10}^{\alpha} 
        
     \end{array}$&
${\mathcal R}= \left\{ \begin{array}{l}
c_{11}^1, c_{11}^2, c_{11}'^2, c_{12}'^2, c_{21}'^2, c_{22}'^2 \in \mathbb{C},
c_{12}'^2=c_{21}'^2, \\c_{11}'^2 c_{22}'^2=(c_{12}'^2)^2,
c_{11}^2 c_{22}'^2=-c_{11}^1 c_{12}'^2
\end{array} \right\}$\\
\hline

$
     {\rm C}_{14}$ &$\not \to$&  $\begin{array}{l}
          {\rm C}_{05}^{\alpha\neq2}, {\rm C}_{06}^{\alpha}, {\rm C}_{12}^{\alpha}, {\rm C}_{10}^{\alpha}, {\rm C}_{04} 
     \end{array} $&
${\mathcal R}= \left\{ \begin{array}{l}
c_{11}^2, c_{11}'^1, c_{11}'^2, c_{12}'^2 \in \mathbb{C},
c_{12}'^2=2 c_{11}'^1
\end{array} \right\}$\\
\hline

$
     {\rm C}_{15}^{\alpha}$ &$\not \to$& $ \begin{array}{l}
          {\rm C}_{05}^{\beta\neq\frac{1}{2}}, {\rm C}_{06}^{\beta}, {\rm C}_{12}^{\beta}, {\rm C}_{10}^{\beta},  {\rm C}_{04} 
     \end{array} $&
${\mathcal R}= \left\{ \begin{array}{l}
c_{11}^2, c_{11}'^1, c_{11}'^2, c_{12}'^2, c_{21}'^1, c_{21}'^2, c_{22}'^2 \in \mathbb{C},\\ c_{11}'^1=c_{21}'^2, 2 c_{21}'^1=c_{22}'^2,  c_{12}'^2=2 c_{11}'^1,\\ 
(c_{11}'^1)^2 = c_{11}'^2 c_{21}'^1 + \alpha ^2 c_{11}^2 c_{21}'^1
\end{array} \right\}$\\
\hline

$
     {\rm C}_{16}$ &$\not \to$& $ \begin{array}{l}
          {\rm C}_{05}^{\alpha}, {\rm C}_{06}^{\alpha\neq2}, {\rm C}_{12}^{\alpha}, {\rm C}_{10}^{\alpha}, {\rm C}_{04}
     \end{array}$ &
${\mathcal R}= \left\{ \begin{array}{l}
c_{11}^2, c_{11}'^1, c_{11}'^2, c_{12}'^2, c_{21}'^2 \in \mathbb{C}, 
c_{21}'^2=c_{11}'^1, c_{12}'^2=2 c_{11}'^1
\end{array} \right\}$\\
\hline

    \end{longtable}

  \end{proof}
 
 At this point, only the description of the closures of the orbits of the parametric families is missing. Although it is not necessary to study the closure of the orbits of each of the parametric families of
the variety of $2$-dimensional commutative pre-Lie algebras in order to identify its irreducible components, we will study them to give a complete description of the variety.

\begin{corollary}\label{th:degcplf}
The description of the closure of the orbit of the parametric families in the variety of $2$-dimensional commutative pre-Lie algebras.

\begin{longtable}{lcl}
  $\overline{\{O({\rm C}_{05}^{*}) \}}$ & ${\supseteq}$ &
 $ \Big\{\overline{\{O({\rm C}_{04})\}}, \  
 \overline{\{O({\rm C}_{03})\}}, \  
 \overline{\{O({\rm C}_{02})\}}, \  
 \overline{\{O({\mathbb C^8})\}} \Big\}$\\
   
   $\overline{\{O({\rm C}_{06}^{*}) \}}$ & ${\supseteq}$ & 
   $ \Big\{\overline{\{O({\rm C}_{04})\}}, \  
   \overline{\{O({\rm C}_{03})\}}, \ 
   \overline{\{O({\rm C}_{01})\}}, \ 
   \overline{\{O({\mathbb C^8})\}} \Big\}$\\
  
  $\overline{\{O({\rm C}_{10}^{*}) \}}$ & ${\supseteq}$ &
 $ \Big\{\overline{\{O({\rm C}_{17}^\alpha)\}}, \  
 \overline{\{O({\rm C}_{04})\}}, \ 
 \overline{\{O({\rm C}_{03})\}}, \ 
 \overline{\{O({\mathbb C^8})\}} \Big\}$\\
 
 $\overline{\{O({\rm C}_{12}^{*}) \}}$ & ${\supseteq}$ &
 $ \Big\{\overline{\{O({\rm C}_{17}^\alpha)\}}, \ 
 \overline{\{O({\rm C}_{04})\}}, \ 
 \overline{\{O({\rm C}_{03})\}}, \ 
 \overline{\{O({\mathbb C^8})\}} \Big\}$\\
  
  $\overline{\{O({\rm C}_{15}^{*}) \}}$ & ${\supseteq}$ &
  $ \Big\{\overline{\{O({\rm C}_{17}^\alpha)\}}, \ 
  \overline{\{O({\rm C}_{16})\}}, \  
  \overline{\{O({\rm C}_{08})\}}, \  
  \overline{\{O({\rm C}_{06}^{2})\}}, \  
  \overline{\{O({\rm C}_{05}^{1/2})\}}, \  
  \overline{\{O({\rm C}_{03})\}}, \ 
  \overline{\{O({\mathbb C^8})\}} \Big\}$\\

  $\overline{\{O({\rm C}_{17}^{*}) \}}$ & ${\supseteq}$ &
  $ \Big\{ \overline{\{O({\rm C}_{03})\}},
  \overline{\{O({\mathbb C^8})\}} \Big\}$\\

\end{longtable}

\end{corollary}

\begin{proof}
Thanks to Theorem \ref{th:degcpl} we have all necessarily degenerations.
All necessarily the closure of the orbits of suitable families are given using the parametric bases and the parametric index included in following Table:

    \begin{longtable}{|lcl|ll|}
\hline
\multicolumn{3}{|c|}{\textrm{Degeneration}}  & \multicolumn{2}{|c|}{\textrm{Parametrized basis}} \\
\hline
\hline

${\rm C}_{05}^{t^{-1}} $ & $\to$ & ${\rm C}_{04}  $ & $
E_{1}(t)= e_2$ & $
E_{2}(t)= te_1$  \\
\hline

${\rm C}_{05}^{t+1} $ & $\to$ & ${\rm C}_{02}  $ & $
E_{1}(t)= e_1 + e_2$ & $
E_{2}(t)= te_2$  \\
\hline

${\rm C}_{06}^{t^{-1}} $ & $\to$ & ${\rm C}_{04}  $ & $
E_{1}(t)= e_2$ & $
E_{2}(t)= te_1$  \\
\hline

${\rm C}_{06}^t $ & $\to$ & ${\rm C}_{01}  $ & $
E_{1}(t)= e_1 + e_2$ & $
E_{2}(t)= t e_2$  \\
\hline

${\rm C}_{10}^{t^{-1}} $ & $\to$ & ${\rm C}_{04}  $ & $
E_{1}(t)= e_2$ & $
E_{2}(t)= te_1$  \\
\hline

${\rm C}_{12}^{t^{-1}} $ & $\to$ & ${\rm C}_{04}  $ & $
E_{1}(t)= e_2$ & $
E_{2}(t)= te_1$  \\
\hline

${\rm C}_{15}^{t^{-1}} $ & $\to$ & ${\rm C}_{16}  $ & $
E_{1}(t)= te_1$ & $
E_{2}(t)= t^2e_2$  \\
\hline

${\rm C}_{15}^{t^{-1}} $ & $\to$ & ${\rm C}_{08}  $ & $
E_{1}(t)= te_1$ & $
E_{2}(t)= \frac{1}{2}e_2$  \\
\hline

${\rm C}_{17}^{t^{-1}} $ & $\to$ & ${\rm C}_{03}  $ & $
E_{1}(t)= e_1$ & $
E_{2}(t)= t^{-1}e_2$  \\
\hline

    \end{longtable}

  \begin{longtable}{|lcl|l|}
\hline
\multicolumn{3}{|c|}{\textrm{Non-degeneration}} & \multicolumn{1}{|c|}{\textrm{Arguments}}\\
\hline
\hline

$ 
     {\rm C}_{05}^{*}$ &$\not \to$& $\begin{array}{l}
          {\rm C}_{17}^{\alpha}, {\rm C}_{06}^{\alpha}
     \end{array} $ &
${\mathcal R}= \left\{ \begin{array}{l}
c_{11}'^1, c_{11}'^2, c_{12}'^2\in \mathbb{C}
\end{array} \right\}$\\
\hline

$ 
     {\rm C}_{06}^{*} $&$\not \to$& $\begin{array}{l}
          {\rm C}_{17}^{\alpha}, {\rm C}_{05}^{\alpha}
     \end{array} $ &
${\mathcal R}= \left\{ \begin{array}{l}
c_{11}'^1, c_{11}'^2, c_{12}'^2, c_{21}'^2\in \mathbb{C},
c_{21}'^2=c_{11}'^1
\end{array} \right\}$\\
\hline

$ 
     {\rm C}_{10}^{*}$ &$\not \to$& $\begin{array}{l}
         {\rm C}_{05}^{\alpha},{\rm C}_{06}^{\alpha}, {\rm C}_{12}^{\alpha}
     \end{array} $  &
${\mathcal R}= \left\{ \begin{array}{l}
c_{11}^1, c_{11}^2, c_{12}^2, c_{21}^2, c_{11}'^2, c_{12}'^2\in \mathbb{C}, \\
c_{12}^2=c_{11}^1, c_{21}^2=c_{11}^1
\end{array} \right\}$\\
\hline

$ 
     {\rm C}_{12}^{*} $&$\not \to$& $\begin{array}{l}
          {\rm C}_{05}^{\alpha},{\rm C}_{06}^{\alpha}, {\rm C}_{10}^{\alpha}
     \end{array} $  &
${\mathcal R}= \left\{ \begin{array}{l}
c_{11}^1, c_{11}^2, c_{11}'^2, c_{12}'^2\in \mathbb{C},
c_{11}^1 c_{11}'^2=-c_{11}^2 c_{12}'^2
\end{array} \right\}$\\
\hline

$ 
     {\rm C}_{15}^{*}$ &$\not \to$&  $\begin{array}{l}
     {\rm C}_{05}^{\alpha\neq\frac{1}{2}},
     {\rm C}_{06}^{\alpha\neq2},  \\
          {\rm C}_{12}^{\alpha},
     {\rm C}_{10}^{\alpha},  {\rm C}_{04}
     \end{array}$
  &
${\mathcal R}= \left\{ \begin{array}{l}
c_{11}^2, c_{11}'^1, c_{11}'^2, c_{12}'^2, c_{21}'^1, c_{21}'^2, c_{22}'^2\in \mathbb{C},\\
c_{11}'^1=c_{21}'^2, 2 c_{21}'^1=c_{22}'^2,  c_{12}'^2=2 c_{11}'^1
\end{array} \right\}$\\
\hline

$ 
     {\rm C}_{17}^{*}$ &$\not \to$&$ \begin{array}{l}
               {\rm C}_{05}^{\alpha},
     {\rm C}_{06}^{\alpha}
     \end{array}$
     
      &
${\mathcal R}= \left\{ \begin{array}{l}
c_{11}^2, c_{11}'^2, c_{12}'^2, c_{21}'^2, c_{22}'^2\in \mathbb{C},\\
c_{12}'^2=c_{21}'^2, c_{22}'^2 c_{11}'^2=(c_{12}'^2)^2
\end{array} \right\}$\\
\hline

    \end{longtable}

\end{proof}

The geometric classification of $2$-dimensional commutative pre-Lie algebras follows by 
Theorem  \ref{th:degcpl} and Corollary \ref{th:degcplf}.

\begin{theorem}
The variety of $2$-dimensional commutative pre-Lie algebras has eleven irreducible components, corresponding to the rigid algebras ${\rm C}_{07}$, ${\rm C}_{09}$, ${\rm C}_{11}$, ${\rm C}_{13}$ and ${\rm C}_{14},$  and the families of  algebras ${\rm C}_{05}^{*}$, ${\rm C}_{06}^{*}$, ${\rm C}_{10}^{*}$, ${\rm C}_{12}^{*}$, ${\rm C}_{15}^{*}$ and ${\rm C}_{17}^{*}$.
\end{theorem}

\section{Anti-pre-Lie Poisson algebras}

The notion of anti-pre-Lie Poisson algebras is introduced in \cite{BL22}. 

\subsection{The algebraic classification of the 2-dimensional anti-pre-Lie
algebras}

\begin{definition}
An algebra $\left( {\rm A},\circ \right) $ is called anti-pre-Lie algebra
if the following equations are satisfied:%
\begin{longtable}{rcl}
$x\circ \left( y\circ z\right) -y\circ \left( x\circ z\right) -
(y \circ x - x \circ y)  \circ z$ &$=$&$0,$ \\
$( x\circ y - y \circ x)  \circ z+( y \circ z - z \circ y) \circ x+(z \circ x - x\circ z)  \circ
y $&$=$&$0.$
\end{longtable}
\end{definition}

It is obvious that any commutative algebra ${\rm A}$ is anti-pre-Lie algebra if and only if ${\rm A}$ is associative. So the algebraic classification of  $2$-dimensional commutative anti-pre-Lie algebras is given in Theorem \ref{asocc2}. The algebraic classification of
the $2$-dimensional non-commutative anti-pre-Lie algebras is given by \cite{BL22}.

\begin{theorem} \label{2dim anti-pre} 
Let $\left( {\rm A},\circ \right) $ be a nonzero $2$-dimensional 
anti-pre-Lie algebra. Then $\left( {\rm A},\circ \right) $ is isomorphic to one
and only one of the following algebras:
\begin{longtable}{lclcllcllcl} 

${\rm A}_{01}$&$:$&$e_1 \circ e_1 $&$=$&$ e_1,$&$ e_2 \circ e_2 $&$= $&$e_2$\\ 

${\rm A}_{02}$&$:$& $ e_1 \circ e_1  $&$=$&$ e_1,$&$ e_1 \circ e_2  $&$=$&$ e_2,$ & $e_2 \circ e_1 $&$=$&$ e_2$\\  

${\rm A}_{03}$&$:$ & $e_1 \circ e_1 $&$=$&$ e_1$\\ 
 
${\rm A}_{04}$&$:$ & $ e_1 \circ e_1  $&$=$&$ e_2$\\
 
 $ {\rm A}_{05}$&$:$ &$e_{1}\circ e_{1} $&$=$&$ -e_{2},$&$ e_{2}\circ e_{1} $&$=$&$ -e_{1}$\\

${\rm A}_{06}^{\lambda}$&$:$&$e_{2}\circ e_{1}$&$=$&$-e_{1},$&$e_{2}\circ
e_{2}$&$=$&$\lambda e_{2}$\\

 ${\rm A}_{07}$&$:$&$  e_{2}\circ e_{1}$&$=$&$-e_{1},$&$e_{2}\circ
e_{2}$&$=$&$e_{1}-e_{2}$\\

 ${\rm A}_{08}^{\lambda \neq -1}$&$:$& $e_{1}\circ e_{2}$&$=$&$\left(
\lambda +1\right) e_{1},$&$e_{2}\circ e_{1}$&$=$&$\lambda e_{1},$&$e_{2}\circ
e_{2}$&$=$&$\left( \lambda -1\right) e_{2}$\\

${\rm A}_{09} $&$:$ &$e_{1}\circ e_{2}$&$=$&$-e_{1},$&$e_{2}\circ
e_{1}$&$=$&$-2e_{1},$&$e_{2}\circ e_{2}$&$=$&$e_{1}-3e_{2}.$
\end{longtable}
\end{theorem}
\subsection{The algebraic classification of the 2-dimensional anti-pre-Lie
Poisson algebras}
\begin{definition}
An  anti-pre-Lie Poisson algebra  is a %
 vector space   equipped with 
 a commutative associative multiplication  $\cdot$ and
an anti-pre-Lie  multiplication $\circ$.
These two operations are required to satisfy the following conditions: 
\begin{longtable}{rcl}
$2\left( x\circ y\right) \cdot z-2\left( y\circ x\right) \cdot z$&$=$&$y\cdot
\left( x\circ z\right) -x\cdot \left( y\circ z\right),$ \\
$2 \ x\circ\left( y\cdot z\right)$  &$=$&$\left( z\cdot x\right) \circ y+z\cdot
\left( x\circ y\right).$
\end{longtable}%
\end{definition}
Anti-pre-Lie algebras $\left( \rm{A},\circ \right) $ 
with zero commutative associative multiplication are isomorphic to   
 anti-pre-Lie Poisson algebras given in Theorem \ref{2dim anti-prepois}.
 Hence,   we only have to study the anti-pre-Lie Poisson structures defined on every commutative associative algebra from Theorem \ref{asocc2}  to obtain the algebraic classification of the 2-dimensional anti-pre-Lie Poisson algebras. To do this, we first need the following definition.

\begin{definition}
Let $\left( \rm{A},\cdot \right) $ be a commutative associative
algebra. Define ${\rm Z}_{\rm APL}^{2}\left( \rm{A},\rm{A}\right) $ to be
the set of all bilinear maps $\theta :\rm{A}\times \rm{A}%
\longrightarrow \rm{A}$ such that:%
\begin{longtable}{rcl}
$\theta \left( x,\theta \left( y,z\right) \right) -\theta \left( y,\theta
\left( x,z\right) \right) -\theta \left( \theta \left( y,x\right) -\theta
\left( x,y\right) ,z\right)$ & $=$&$0,$ \\
$\theta \left( \theta \left( x,y\right) -\theta \left( y,x\right) ,z\right)
+\theta \left( \theta \left( y,z\right) -\theta \left( z,y\right) ,x\right)
+\theta \left( \theta \left( z,x\right) -\theta \left( x,z\right) ,y\right)$
& $=$&$0,$ \\
$2\theta \left( x,y\right) \cdot z-2\theta \left( y,x\right) \cdot z-y\cdot
\theta \left( x,z\right) +x\cdot \theta \left( y,z\right)$ & $=$&$0,$ \\
$2\theta \left( x,y\cdot z\right) -\theta \left( z\cdot x,y\right) -z\cdot
\theta \left( x,y\right)$ & $=$&$0.$
\end{longtable}%
If $\theta \in {\rm Z}_{\rm APL}^{2}\left( \rm{A},\rm{A}\right) $, then $\left( \rm{A},\cdot ,\circ 
\right) $ is an anti-pre-Lie Poisson algebra where $x\circ y=\theta \left(
x,y\right) $ for all $x,y\in \rm{A}$.
\end{definition}

\subsubsection{Anti-pre-Lie Poisson algebras defined on ${\rm A}_{01}$}

The anti-pre-Lie Poisson defined on ${\rm A}_{01}$ are: 
\begin{equation*}
\mathrm{A}_{10}^{\alpha ,\beta }:=\left\{ 
\begin{tabular}{lcllcl}
$e_{1}\cdot e_{1}$&$=$&$e_{1},$ & $e_{2}\cdot e_{2}$&$=$&$e_{2},$ \\ 
$e_{1}\circ e_{1}$&$=$&$\alpha e_{1},$ & $e_{2}\circ e_{2}$&$=$&$\beta e_{2}.$%
\end{tabular}%
\right. 
\end{equation*}%
Since $\mathrm{Aut}({\rm A}_{01})={\mathbb S}_{2}$, we have $\mathrm{A}%
_{10}^{\alpha ,\beta }\cong \mathrm{A}_{10}^{\beta ,\alpha }$.

\subsubsection{Anti-pre-Lie Poisson algebras defined on ${\rm{A}}_{02}$}

The anti-pre-Lie Poisson structures defined on ${\rm A}_{02}$ are: 
\begin{equation*}
\left\{ 
\begin{tabular}{lcllcl lcl}
$e_{1}\cdot e_{1}$&$=$&$e_{1},$ & $e_{1}\cdot e_{2}$&$=$&$e_{2},$ &  \\ 
$e_{1}\circ e_{1} $&$=$&$\left( 2\alpha _{2}-\alpha _{3}\right) e_{1}+\alpha
_{1}e_{2},$ & $e_{1}\circ e_{2}$&$=$&$\alpha _{2}e_{2},$ & $e_{2}\circ
e_{1}$&$=$&$\alpha _{3}e_{2}.$%
\end{tabular}%
\right. 
\end{equation*}

By Lemma \ref{isom}, 
to study the action of $\mathrm{Aut}(\mathrm{A}%
_{02})$, we have to distinguish two cases:

\begin{itemize}
\item If $\alpha _{1}\neq 0$, then choose $\lambda _{22}=\alpha _{1}$ and
obtain the parametric family: 
\begin{equation*}
\mathrm{A}_{11}^{\alpha ,\beta }:=\left\{ 
\begin{tabular}{lcllcl lcl}
$e_{1}\cdot e_{1}$&$=$&$e_{1},$ & $e_{1}\cdot e_{2}$&$=$&$e_{2},$ &  \\ 
$e_{1}\circ e_{1}$&$=$&$\left( 2\alpha -\beta \right) e_{1}+e_{2},$ & $e_{1}\circ
e_{2}$&$=$&$\alpha e_{2},$ & $e_{2}\circ e_{1}$&$=$&$\beta e_{2}.$%
\end{tabular}%
\right. 
\end{equation*}

\item If $\alpha _{1}=0$, then we have: 
\begin{equation*}
\mathrm{A}_{12}^{\alpha ,\beta }:=\left\{ 
\begin{tabular}{lcllcllcl}
$e_{1}\cdot e_{1}$&$=$&$e_{1},$ & $e_{1}\cdot e_{2}$&$=$&$e_{2},$ &  \\ 
$e_{1}\circ e_{1}$&$=$&$\left( 2\alpha -\beta \right) e_{1},$ & $e_{1}\circ
e_{2}$&$=$&$\alpha e_{2},$ & $e_{2}\circ e_{1}$&$=$&$\beta e_{2}.$%
\end{tabular}%
\right. 
\end{equation*}
\end{itemize}

\subsubsection{Anti-pre-Lie Poisson algebras defined on ${\rm A}_{03}$}

The anti-pre-Lie Poisson structures defined on ${\rm A}_{03}$ are: 
\begin{equation*}
\left\{ 
\begin{tabular}{lcllcllcl}
$e_{1}\cdot e_{1}$&$=$&$e_{1},$ &  &  \\ 
$e_{1}\circ e_{1}$&$=$&$\alpha _{1}e_{1},$ & $e_{2}\circ e_{2}$&$=$&$\alpha _{2}e_{2}.$
& 
\end{tabular}%
\right. 
\end{equation*}

The action of $\mathrm{Aut}({\rm A}_{03})$ produces two cases:

\begin{itemize}
\item If $\alpha _{2}\neq 0$, then set $\lambda _{22}=\alpha _{2}^{-1}$ and
hence we obtain the algebra: 
\begin{equation*}
\mathrm{A}_{13}^{\alpha }:=\left\{ 
\begin{tabular}{lcllcllcl}
$e_{1}\cdot e_{1}$&$=$&$e_{1},$ &  &  \\ 
$e_{1}\circ e_{1}$&$=$&$\alpha e_{1},$ & $e_{2}\circ e_{2}$&$=$&$e_{2}.$ & 
\end{tabular}%
\right. 
\end{equation*}

\item If $\alpha _{2}=0$, then we obtain the family: 
\begin{equation*}
\mathrm{A}_{14}^{\alpha }:=\left\{ 
\begin{tabular}{lcllcl}
$e_{1}\cdot e_{1}$&$=$&$e_{1},$ &  \\ 
$e_{1}\circ e_{1}$&$=$&$\alpha e_{1}.$ & 
\end{tabular}%
\right. 
\end{equation*}
\end{itemize}

\subsubsection{Anti-pre-Lie Poisson algebras defined on ${\rm A}_{04}$}

The anti-pre-Lie Poisson structures defined on ${\rm A}_{04}$ are: 
\begin{equation*}
\left\{ 
\begin{tabular}{lcllcllcl}
$e_{1}\cdot e_{1}$&$=$&$e_{2},$ &  &  \\ 
$e_{1}\circ e_{1}$&$=$&$\left( 2\alpha _{2}-\alpha _{3}\right) e_{1}+\alpha
_{1}e_{2},$ & $e_{1}\circ e_{2}$&$=$&$\alpha _{2}e_{2},$ & $e_{2}\circ
e_{1}$&$=$&$\alpha _{3}e_{2}.$%
\end{tabular}%
\right. 
\end{equation*}%
The following cases arise from the action of $\mathrm{Aut}({\rm A}%
_{04})$ on this structure:

\begin{itemize}
\item If $\alpha _{2}\neq 0$, then 

\begin{itemize}
\item If $2\alpha _{3}-\alpha _{2}\neq 0$, choose $\lambda _{11}=\alpha
_{2}^{-1}$ and $\lambda _{21}=\frac{\alpha _{1}}{\alpha _{2}\left( \alpha
_{2}-2\alpha _{3}\right) }$ obtaining the family: 
\begin{equation*}
\mathrm{A}_{15}^{\alpha\neq\frac{1}{2}}:=\left\{ 
\begin{tabular}{lcllcllcl}
$e_{1}\cdot e_{1}$&$=$&$e_{2},$ &  &  \\ 
$e_{1}\circ e_{1}$&$=$&$\left( 2-\alpha \right) e_{1},$ & $e_{1}\circ e_{2}$&$=$&$e_{2},$
& $e_{2}\circ e_{1}$&$=$&$\alpha e_{2}.$%
\end{tabular}%
\right. 
\end{equation*}

Note that here $\alpha\neq\frac{1}{2}$ follows from $2\alpha _{3}-\alpha _{2}\neq 0$.

\item If $2\alpha _{3}-\alpha _{2}=0$, choose $\lambda _{11}=\alpha _{2}^{-1}
$ to obtain the family:%
\begin{equation*}
{
\mathrm{A}_{16}^{\alpha}:=\left\{ 
\begin{tabular}{lcllcllcl}
$e_{1}\cdot e_{1}$&$=$&$e_{2},$ &  &  \\ 
$e_{1}\circ e_{1}$&$=$&$\frac{3}{2} e_{1}+\alpha e_{2},$ & $e_{1}\circ
e_{2}$&$=$&$e_{2},$ & 
$e_{2}\circ e_{1}$&$=$&$\frac{1}{2} e_{2}.$%
\end{tabular}%
\right. }
\end{equation*}

Moreover, observe that $\mathrm{A}_{15}^{\frac{1}{2}}= \mathrm{A}_{16}^{0}$.

\end{itemize}

\item If $\alpha _{2}=0$, then 

\begin{itemize}
\item If $\alpha _{3}\neq 0$, choose $\lambda _{11}=\alpha _{3}^{-1}$ and $%
\lambda _{21}=- \frac{\alpha _{1}}{2\alpha _{3}^{2}}$ obtaining the
algebra: 
\begin{equation*}
\mathrm{A}_{17}:=\left\{ 
\begin{tabular}{lcllcllcl}
$e_{1}\cdot e_{1}$&$=$&$e_{2},$ \\ 
$ e_{1}\circ e_{1}$&$=$&$-e_{1},$ & $ e_{2}\circ e_{1}$&$=$&$e_{2}.$%
\end{tabular}%
\right. 
\end{equation*}

\item If $\alpha _{3}=0$, then we have the family: 
\begin{equation*}
\mathrm{A}_{18}^{\alpha }:=\left\{ 
\begin{tabular}{lcl}
$e_{1}\cdot e_{1}$&$=$&$e_{2},$ \\ 
$e_{1}\circ e_{1}$&$=$&$\alpha e_{2}.$%
\end{tabular}%
\right. 
\end{equation*}
\end{itemize}
\end{itemize}

\begin{theorem}
\label{2dim anti-prepois} 
Let $\left( {\rm A},\cdot , \circ \right) $ be a nonzero $2$-dimensional anti-pre-Lie Poisson algebra. Then ${\rm A}$ is isomorphic to one
anti-pre-Lie algebra listed in   Theorem \ref{2dim anti-pre} or to one algebra listed below: 

\begin{longtable}{lcl}

 ${\rm A}_{10}^{\alpha, \beta} $&$:$&$ \left\{ 
\begin{tabular}{lcllcllcl}
$e_{1}\cdot e_{1}$&$=$&$e_{1},$ & $e_{2}\cdot e_{2}$&$=$&$e_{2},$ \\ 
$e_{1}\circ e_{1}$&$=$&$\alpha e_{1},$ & $e_{2}\circ e_{2}$&$=$&$\beta e_{2}.$%
\end{tabular}%
\right. $\\

 ${\rm A}_{11}^{\alpha, \beta}$&$ : $&$\left\{ 
\begin{tabular}{lcllcllcl}
$e_{1}\cdot e_{1}$&$=$&$e_{1},$ & $e_{1}\cdot e_{2}$&$=$&$e_{2},$ &  \\ 
$e_{1}\circ e_{1}$&$=$&$\left( 2\alpha -\beta \right) e_{1}+e_{2},$ & $e_{1}\circ
e_{2}$&$=$&$\alpha e_{2},$ & $e_{2}\circ e_{1}$&$=$&$\beta e_{2}.$%
\end{tabular}%
\right.$\\

 ${\rm A}_{12}^{\alpha, \beta} $&$: $&$\left\{ 
\begin{tabular}{lcllcllcl}
$e_{1}\cdot e_{1}$&$=$&$e_{1},$ & $e_{1}\cdot e_{2}$&$=$&$e_{2},$ &  \\ 
$e_{1}\circ e_{1}$&$=$&$\left( 2\alpha -\beta \right) e_{1},$ & $e_{1}\circ
e_{2}$&$=$&$\alpha e_{2},$ & $e_{2}\circ e_{1}$&$=$&$\beta e_{2}.$%
\end{tabular}%
\right.$\\

 ${\rm A}_{13}^{\alpha}$&$ :$&$ \left\{ 
\begin{tabular}{lcllcllcl}
$e_{1}\cdot e_{1}$&$=$&$e_{1},$ &  &  \\ 
$e_{1}\circ e_{1}$&$=$&$\alpha e_{1},$ & $e_{2}\circ e_{2}$&$=$&$e_{2}.$ & 
\end{tabular}%
\right.$\\

 ${\rm A}_{14}^{\alpha}$&$ :$&$ \left\{ 
\begin{tabular}{lcllcl}
$e_{1}\cdot e_{1}$&$=$&$e_{1},$ &  \\ 
$e_{1}\circ e_{1}$&$=$&$\alpha e_{1}.$ & 
\end{tabular}%
\right.$\\

 ${\rm A}_{15}^{\alpha\neq\frac{1}{2}}$&$ :$&$ \left\{ 
\begin{tabular}{lcllcllcl}
$e_{1}\cdot e_{1}$&$=$&$e_{2},$ &  &  \\ 
$e_{1}\circ e_{1}$&$=$&$\left( 2-\alpha \right) e_{1},$ & $e_{1}\circ e_{2}$&$=$&$e_{2},$
& $e_{2}\circ e_{1}$&$=$&$\alpha e_{2}.$%
\end{tabular}%
\right.$\\

${\rm A}_{16}^{\alpha}$&$ :$&$ \left\{ 
\begin{tabular}{lcllcllcl}
$e_{1}\cdot e_{1}$&$=$&$e_{2},$ &  &  \\ 
$e_{1}\circ e_{1}$&$=$&$\frac{3}{2} e_{1}+\alpha e_{2},$ & $e_{1}\circ
e_{2}$&$=$&$e_{2},$ & $e_{2}\circ e_{1}$&$=$&$\frac{1}{2} e_{2}.$%
\end{tabular}%
\right.$\\

 ${\rm A}_{17} $&$:$&$ \left\{ 
\begin{tabular}{lcllcl}
$e_{1}\cdot e_{1}$&$=$&$e_{2},$ \\ 
$ e_{1}\circ e_{1}$&$=$&$-e_{1},$ & $ e_{2}\circ e_{1}$&$=$&$e_{2}.$%
\end{tabular}%
\right.$\\

 ${\rm A}_{18}^{\alpha} $&$:$&$ \left\{ 
\begin{tabular}{lcl}
$e_{1}\cdot e_{1}$&$=$&$e_{2},$ \\ 
$e_{1}\circ e_{1}$&$=$&$\alpha e_{2}.$%
\end{tabular}%
\right.$

\end{longtable}

Between these algebras, the only non-trivial isomorphism is ${\rm A}_{10}^{\alpha,\beta} \cong {\rm A}_{10}^{\beta,\alpha}$.

\end{theorem}

\black
\subsection{Degenerations of $2$-dimensional anti-pre-Lie Poisson algebras}

\begin{lemma} \label{th:degapre}
The graph of primary degenerations and non-degenerations of the variety of $2$-dimensional anti-pre-Lie Poisson algebras is given in Figure  5, where the numbers on the right side are the dimensions of the corresponding orbits.

\end{lemma}
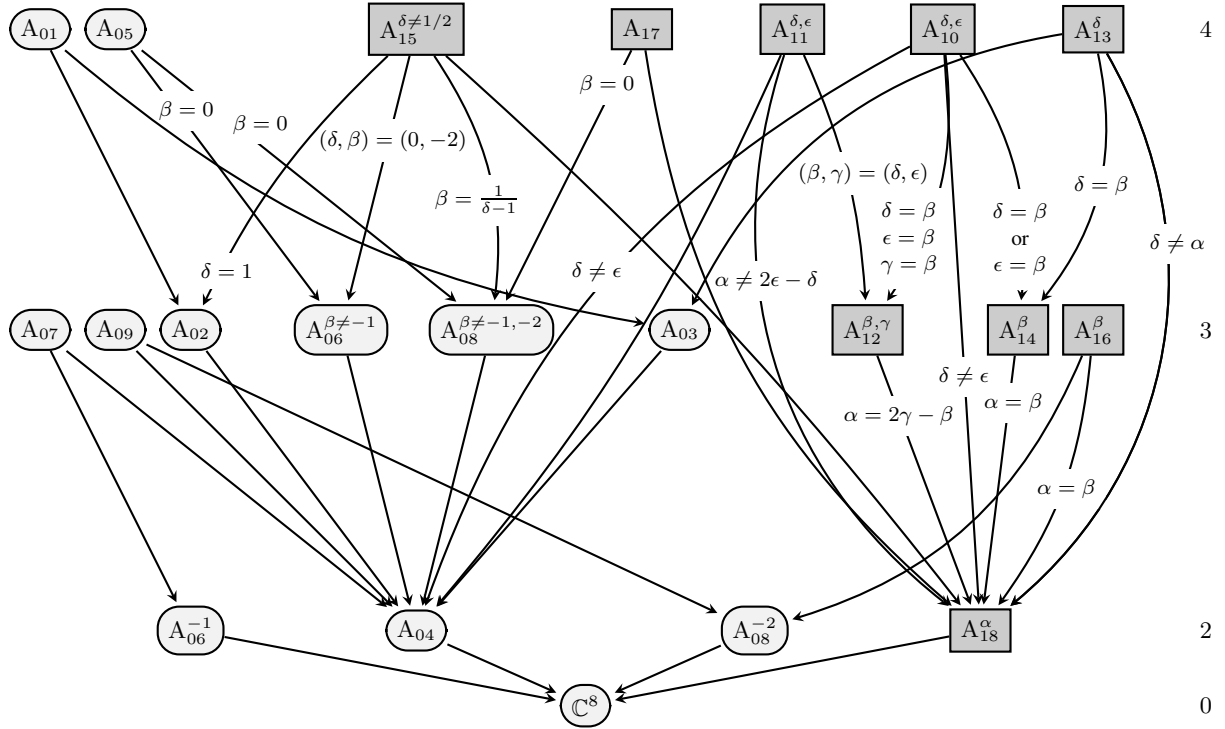
\begin{center}
	
	\begin{tikzpicture}[->,>=stealth,shorten >=0.05cm,auto,node distance=1.3cm,
	thick,
	main node/.style={rectangle,draw,fill=gray!10,rounded corners=1.5ex,font=\sffamily \scriptsize \bfseries },
	rigid node/.style={rectangle,draw,fill=black!20,rounded corners=0ex,font=\sffamily \scriptsize \bfseries }, 
	poisson node/.style={rectangle,draw,fill=black!20,rounded corners=0ex,font=\sffamily \scriptsize \bfseries },
	ac node/.style={rectangle,draw,fill=black!20,rounded corners=0ex,font=\sffamily \scriptsize \bfseries },
	lie node/.style={rectangle,draw,fill=black!20,rounded corners=0ex,font=\sffamily \scriptsize \bfseries },
	style={draw,font=\sffamily \scriptsize \bfseries }]

	\node (3) at (0.5,9) {$4$};
	\node (2) at (0.5,5) {$3$};
	\node (1) at (0.5,1) {$2$};
	\node (0)  at (0.5,0) {$0$};

    \node[main node] (c20) at (-7.75,0) {${\mathbb C^8}$};
    
    \node[main node] (c24) at (-10,1) {$\mathrm{A}_{04}$};
    \node[main node] (c26p) at (-13,1) {$\mathrm{A}_{06}^{-1}$};
    \node[main node] (c28p) at (-5.5,1) {$\mathrm{A}_{08}^{-2}$};
    \node[ac node] (c218) at (-2.5,1) {$\mathrm{A}_{18}^{\alpha}$};

    \node[main node] (c27) at (-15,5) {$\mathrm{A}_{07}$};
    \node[main node] (c29) at (-14,5) {$\mathrm{A}_{09}$};
    \node[main node] (c22) at (-13,5) {$\mathrm{A}_{02}$};
    \node[main node] (c26) at (-11,5) {$\mathrm{A}_{06}^{\beta\neq-1}$};
    \node[main node] (c28) at (-9,5) {$\mathrm{A}_{08}^{\beta\neq-1,-2}$};
    \node[main node] (c23) at (-6.5,5) {$\mathrm{A}_{03}$};
    \node[ac node] (c212) at (-4,5) {$\mathrm{A}_{12}^{\beta,\gamma}$};
    \node[ac node] (c214) at (-2,5) {$\mathrm{A}_{14}^{\beta}$};

    \node[ac node] (c216) at (-1,5) {$\mathrm{A}_{16}^{\beta}$};
    
    \node[main node] (c21) at (-15,9) {$\mathrm{A}_{01}$};
    \node[main node] (c25) at (-14,9) {$\mathrm{A}_{05}$};
    \node[ac node] (c210) at (-3,9) {$\mathrm{A}_{10}^{\delta,\epsilon}$};
    \node[ac node] (c211) at (-5,9) {$\mathrm{A}_{11}^{\delta,\epsilon}$};
    \node[ac node] (c213) at (-1,9) {$\mathrm{A}_{13}^{\delta}$};
    \node[ac node] (c215) at (-10,9) {$\mathrm{A}_{15}^{\delta\neq1/2}$};
    \node[ac node] (c217) at (-7,9) {$\mathrm{A}_{17}$};
    
	\path[every node/.style={font=\sffamily\small}]

    
    (c210) edge [bend left=20]  node[above=-20, right=-32, fill=white]{\tiny $\begin{array}{c}
        \delta=\beta \\
        \epsilon = \beta  \\
        \gamma = \beta  
    \end{array}$} (c212)
    (c210) edge [bend left=20] node[above=-25, right=-18, fill=white]{\tiny $\begin{array}{c}
        \delta=\beta \\\textrm{ or } \\
        \epsilon = \beta  
    \end{array}$} (c214)
    (c210) edge [bend left=-21]  node[above=2, right=-18, fill=white]{\tiny $\delta\neq\epsilon$} (c24)
    (c210) edge [bend left=0] node[above=-17, right=-13, fill=white]{\tiny $\delta\neq\epsilon$} (c218)

    (c211) edge [bend left=-35]  node[above=30, right=-27, fill=white]{\tiny $\alpha\neq2\epsilon-\delta$} (c218)
     (c211) edge [bend left=10] node[above=1, right=-23, fill=white]{\tiny $(\beta,\gamma)=(\delta,\epsilon)$} (c212)
    (c211) edge [bend left=8] (c24)
    
    (c213) edge [bend left=40] node[above=40, right=-9, fill=white]{\tiny $\delta\neq\alpha$} (c218)
    (c213) edge [bend left=25] node[above=0, right=-14, fill=white]{\tiny $\delta=\beta$} (c214)
    (c213) edge [bend left=-27] (c23)
    
    (c215) edge [bend right=0] node[above=15, right=-27, fill=white]{\tiny $(\delta,\beta)=(0,-2)$} (c26)
    (c215) edge [bend right=10] node[above=-37, right=-35, fill=white]{\tiny $\delta=1$} (c22)
    (c215) edge [bend right=-20] node[above=-10, right=-25, fill=white]{\tiny $\beta=\frac{1}{\delta-1}$} (c28)
    (c215) edge [bend right=-5] (c218)
    
    (c217) edge node[above=35, right=0, fill=white]{\tiny $\beta=0$} (c28)
    (c217) edge [bend right=20] (c218)
    
    (c212) edge node[above=25, right=-35, fill=white]{\tiny $\alpha=2\gamma - \beta$} (c218)
    (c214) edge node[above=30, right=-10, fill=white]{\tiny $\alpha=\beta$} (c218)
    (c216) edge [bend right= -25]  (c28p)
    (c216) edge [bend right=-15]  node[above=0, right=-14, fill=white]{\tiny $\alpha=\beta$} (c218)
    
    (c25) edge [bend right=0] node[above=20, right=-30, fill=white]{\tiny $\beta=0$} (c28)
    (c25) edge [bend right=0] node[above=25, right=-30, fill=white]{\tiny $\beta=0$} (c26)
  
    (c21) edge (c22)
    (c21) edge [bend right=15] (c23)
    
    (c22) edge (c24)
    (c23) edge (c24)
    
    (c26) edge (c24)
    (c28) edge (c24)
    
    (c27) edge (c24)
    (c27) edge (c26p)
    
    (c29) edge (c24)
    (c29) edge (c28p)
    
    (c24) edge (c20)
    (c26p) edge (c20)
    (c28p) edge (c20)
    (c218) edge (c20) 
     (c213) edge [bend left=40] node[above=40, right=-9, fill=white]{\tiny $\delta\neq\alpha$} (c218)  ;
    
	\end{tikzpicture}

{\tiny 
\begin{itemize}
\noindent Legend:
\begin{itemize}
    \item[--] Round nodes: anti-pre-Lie Poisson algebras with trivial commutative associative multiplication $\cdot$.
    \item[--] Squared nodes: anti-pre-Lie Poisson algebras with non-trivial commutative associative multiplication $\cdot$.

\end{itemize}
\end{itemize}}

{Figure 5.}  Graph of primary degenerations and non-degenerations.	
\end{center}
\bigskip

\bigskip

\begin{proof}
The dimensions of the orbits are deduced by computing the algebra of derivations.
The primary degenerations are proven using the parametric bases included in following Table:

    \begin{longtable}{|lll|ll|}
\hline
\multicolumn{3}{|c|}{\textrm{Degeneration}}  & \multicolumn{2}{|c|}{\textrm{Parametrized basis}} \\
\hline
\hline

$\mathrm{A}_{01} $ & $\to$ & $\mathrm{A}_{02} $ & $
E_{1}(t)= e_1 + e_2$ & $
E_{2}(t)= t e_2$  \\
\hline

$\mathrm{A}_{01} $ & $\to$ & $\mathrm{A}_{03} $ & $
E_{1}(t)= e_1$ & $
E_{2}(t)= t e_2$  \\
\hline

$\mathrm{A}_{02} $ & $\to$ & $\mathrm{A}_{04} $ & $
E_{1}(t)= t e_1 + e_2$ & $
E_{2}(t)= t e_2$  \\
\hline

$\mathrm{A}_{03} $ & $\to$ & $\mathrm{A}_{04} $ & $
E_{1}(t)= t e_1 + e_2$ & $
E_{2}(t)= t^2 e_1$  \\
\hline

$\mathrm{A}_{05} $ & $\to$ & $\mathrm{A}_{06}^{0} $ & $
E_{1}(t)= te_1$ & $
E_{2}(t)= e_2$  \\
\hline

$\mathrm{A}_{05} $ & $\to$ & $\mathrm{A}_{08}^{0} $ & $
E_{1}(t)= t e_2$ & $
E_{2}(t)= e_1 + e_2$  \\
\hline

$\mathrm{A}_{06}^{\alpha\neq-1} $ & $\to$ & $\mathrm{A}_{04} $ & $
E_{1}(t)= e_1 + t e_2$ & $
E_{2}(t)= -t (1+\alpha) e_1$  \\
\hline

$\mathrm{A}_{07} $ & $\to$ & $\mathrm{A}_{04} $ & $
E_{1}(t)= t e_2$ & $
E_{2}(t)= t^2 e_1$  \\
\hline

$\mathrm{A}_{07} $ & $\to$ & $\mathrm{A}_{06}^{-1} $ & $
E_{1}(t)= t^{-1}e_1$ & $
E_{2}(t)= e_2$  \\
\hline

$\mathrm{A}_{08}^{\alpha\neq-2} $ & $\to$ & $\mathrm{A}_{04} $ & $
E_{1}(t)= e_1 + t e_2$ & $
E_{2}(t)= -t^2 (2+\alpha) e_2$  \\
\hline

$\mathrm{A}_{09} $ & $\to$ & $\mathrm{A}_{04} $ & $
E_{1}(t)= te_2$ & $
E_{2}(t)= t^2e_1$  \\
\hline

$\mathrm{A}_{09} $ & $\to$ & $\mathrm{A}_{08}^{-2} $ & $
E_{1}(t)= t^{-1}e_1 $ & $
E_{2}(t)= e_2 $  \\
\hline

{ 
$\mathrm{A}_{10}^{\alpha,\beta\neq\alpha} $} & $\to$ & $\mathrm{A}_{04}$ & $
E_{1}(t)= t e_1 + t e_2$ & $
E_{2}(t)= t^2 (\alpha-\beta)e_1$   \\
\hline

$\mathrm{A}_{10}^{\alpha,\alpha} $ & $\to$ & $\mathrm{A}_{12}^{\alpha,\alpha} $ & $
E_{1}(t)= e_1 + e_2$ & $
E_{2}(t)= t e_2$  \\
\hline

$\mathrm{A}_{10}^{\alpha,\beta} $ & $\to$ & $\mathrm{A}_{14}^{\alpha} $ & $
E_{1}(t)= e_1$ & $
E_{2}(t)= t e_2$  \\
\hline

$\mathrm{A}_{10}^{\alpha,\beta} $ & 
$ \to
\footnote{Here $\alpha\neq\beta, \beta\neq \gamma$ and $\alpha\neq\gamma$.}$ & $\mathrm{A}_{18}^{\gamma}  $ & $
E_{1}(t)= \frac{t(\beta-\gamma)}{\alpha-\gamma}e_1 + t e_2$ & $
E_{2}(t)= \frac{t^2(\alpha-\beta)}{\alpha-\gamma}e_2$  \\
\hline

$\mathrm{A}_{11}^{\alpha,\beta} $ & $\to$ & $\mathrm{A}_{04} $ & $
E_{1}(t)= t e_1$ & $
E_{2}(t)= t^2 e_2$  \\
\hline

$\mathrm{A}_{11}^{\alpha,\beta} $ & $\to$ & $\mathrm{A}_{18}^{\gamma\neq2\beta-\alpha} $ & $
E_{1}(t)= t(\alpha-2\beta +\gamma)e_1 + t e_2$ & $
E_{2}(t)= t^2(\alpha-2\beta+\gamma)e_2$  \\
\hline

$\mathrm{A}_{11}^{\alpha,\beta} $ & $\to$ & $\mathrm{A}_{12}^{\alpha,\beta} $ & $
E_{1}(t)= e_1$ & $
E_{2}(t)= t^{-1}e_2$  \\
\hline

$\mathrm{A}_{12}^{\alpha,\beta} $ & $\to$ & $\mathrm{A}_{18}^{2\beta-\alpha} $ & $
E_{1}(t)= t e_1 + e_2$ & $
E_{2}(t)= t e_2$  \\
\hline

$\mathrm{A}_{13}^{\alpha} $ & $\to$ & $\mathrm{A}_{14}^{\alpha} $ & $
E_{1}(t)= e_1 $ & $
E_{2}(t)= t e_2 $  \\
\hline

$\mathrm{A}_{13}^{\alpha} $ & $\to$ & $\mathrm{A}_{03} $ & $
E_{1}(t)= t^2 e_1 + e_2$ & $
E_{2}(t)= t e_1$  \\
\hline

$\mathrm{A}_{13}^{\alpha} $ & $\to$ & $\mathrm{A}_{18}^{\beta\neq\alpha} $ & $
E_{1}(t)= t e_1 + t(\alpha-\beta)e_2$ & $
E_{2}(t)= -t^2(\alpha-\beta)e_2$  \\
\hline

$\mathrm{A}_{14}^{\alpha} $ & $\to$ & $\mathrm{A}_{18}^{\alpha} $ & $
E_{1}(t)= t e_1 + e_2$ & $
E_{2}(t)= - t e_2$  \\
\hline

$\mathrm{A}_{15}^{1} $ & $\to$ & $\mathrm{A}_{02} $ & $
E_{1}(t)= e_1$ & $
E_{2}(t)= t^{-1}e_2$  \\
\hline

$\mathrm{A}_{15}^0 $ & $\to$ & $\mathrm{A}_{06}^{-2} $ & $
E_{1}(t)= t^{-1}e_2$ & $
E_{2}(t)= -e_1$  \\
\hline

$\mathrm{A}_{15}^{\alpha\neq1} $ & $\to$ & $\mathrm{A}_{08}^{\frac{1}{\alpha-1}} $ & $
E_{1}(t)= t^{-1}e_2$ & $
E_{2}(t)= \frac{1}{\alpha-1}e_1 + \frac{1}{\alpha-1} e_2$  \\
\hline

$\mathrm{A}_{15}^{\alpha\neq\frac{1}{2}} $ & $\to$ & $\mathrm{A}_{18}^{\beta} $ & $
E_{1}(t)= t e_1 + \frac{t \beta}{2\alpha-1}e_2$ & $
E_{2}(t)= t^2 e_2$  \\
\hline

$\mathrm{A}_{16}^{\alpha} $ & $\to$ & $\mathrm{A}_{18}^{\alpha} $ & $
E_{1}(t)= t e_1$ & $
E_{2}(t)= t^2 e_2$  \\
\hline

$\mathrm{A}_{16}^{\alpha} $ & $\to$ & $\mathrm{A}_{08}^{-2} $ & $
E_{1}(t)= t^{-1}e_2$ & $
E_{2}(t)= -2 e_1$  \\
\hline

$\mathrm{A}_{17} $ & $\to$ & $\mathrm{A}_{18}^{\alpha\neq0} $ & $
E_{1}(t)= t e_1 + \frac{t \alpha}{2}e_2$ & $
E_{2}(t)= -\frac{t^2}{\alpha}e_1 + \frac{t^2}{2}e_2 $  \\
\hline

$\mathrm{A}_{17} $ & $\to$ & $\mathrm{A}_{18}^0 $ & $
E_{1}(t)= te_1$ & $
E_{2}(t)= t^2e_2$  \\
\hline

$\mathrm{A}_{17} $ & $\to$ & $\mathrm{A}_{08}^{0} $ & $
E_{1}(t)=t^{-1}e_2 $ & $
E_{2}(t)=e_1 $  \\
\hline

 
    \end{longtable}

The primary non-degenerations are proven using the sets from the following table: 

  {\small
    \begin{longtable}{|l|l|}
\hline
\multicolumn{1}{|c|}{\textrm{Non-degeneration}} & \multicolumn{1}{|c|}{\textrm{Arguments}}\\
\hline
\hline

$\begin{array}{cccc}
     \mathrm{A}_{07} &\not \to&  \mathrm{A}_{08}^{-2} & 
\end{array}$ &
${\mathcal R}= \left\{ \begin{array}{l} c_{ij}'^{k} \in \mathbb{C}, c_{11}'^{1}+2 c_{21}'^{2}=c_{12}'^{2},  2 c_{12}'^{1}+c_{22}'^{2}=c_{21}'^{1}
\end{array} \right\}$\\
\hline

$\begin{array}{cccc}
     \mathrm{A}_{09} &\not \to&  \mathrm{A}_{06}^{-1} & 
\end{array}$ &
${\mathcal R}= \left\{ \begin{array}{l} c_{ij}'^{k} \in \mathbb{C}, c_{11}'^{1}+5 c_{21}'^{2}=4 c_{12}'^{2},  5 c_{12}'^{1}+c_{22}'^{2}=4 c_{21}'^{1}
\end{array} \right\}$\\
\hline

$\begin{array}{cccc}
     \mathrm{A}_{06}^{\alpha} &\not \to&  \mathrm{A}_{06}^{-1}, \mathrm{A}_{08}^{-2} & 
\end{array}$ &
${\mathcal R}= \left\{ \begin{array}{l} c_{11}'^{1}, c_{11}'^{2}, c_{12}'^{2}, c_{21}'^{1}, c_{21}'^{2}, c_{22}'^{2} \in \mathbb{C}, -\alpha  c_{11}'^{1}=c_{12}'^{2},\\  -(\alpha +1) c_{11}'^{1}=c_{21}'^{2},  -\alpha  c_{21}'^{1}=c_{22}'^{2}
\end{array} \right\}$\\
\hline

$\begin{array}{cccc}
     \mathrm{A}_{08}^{\alpha} &\not \to&  \mathrm{A}_{06}^{-1}, \mathrm{A}_{08}^{-2} & 
\end{array}$ &
${\mathcal R}= \left\{ \begin{array}{l} c_{11}'^{1}, c_{11}'^{2}, c_{12}'^{1}, c_{12}'^{2}, c_{21}'^{1}, c_{21}'^{2}, c_{22}'^{2}  \in \mathbb{C}, \\
(2 \alpha +1) c_{21}'^{2}+c_{11}'^{1}=0,  (\alpha -1) c_{21}'^{1}=\alpha  c_{22}'^{2}, \\ c_{12}'^{1}+c_{22}'^{2}=2 c_{21}'^{1}
\end{array} \right\}$\\
\hline

$\begin{array}{cccc}
     \mathrm{A}_{05} &\not \to&  \mathrm{A}_{02},\mathrm{A}_{03},\mathrm{A}_{06}^{\alpha \neq 0}, 
     \mathrm{A}_{08}^{\alpha\neq0}
     & 
\end{array}$ &
${\mathcal R}= \left\{ \begin{array}{l} c_{11}'^{1}, c_{11}'^{2}, c_{21}'^{1}, c_{21}'^{2},  \in \mathbb{C}, c_{21}'^{2}=-c_{11}'^{1}
\end{array} \right\}$\\
\hline

$\begin{array}{cccc}
     \mathrm{A}_{12}^{\alpha,\beta} &\not \to&  \mathrm{A}_{04}, \mathrm{A}_{06}^{-1}, \mathrm{A}_{08}^{-2}, \mathrm{A}_{18}^{\gamma\neq 2\beta-\alpha} & 
\end{array}$ &
${\mathcal R}= \left\{ \begin{array}{l} c_{11}^{1}, c_{11}^{2}, c_{12}^{2}, c_{21}^{2},  c_{11}'^{1}, c_{11}'^{2}, c_{12}'^{2}, c_{21}'^{1}  \in \mathbb{C}, \\ c_{12}^{2}=c_{11}^{1},  c_{21}^{2}=c_{11}^{1}, c_{12}'^{2}=\alpha  c_{11}^{1},  c_{21}'^{2}=\beta  c_{11}^{1},   \\   c_{11}'^{2}=-(\alpha -2 \beta )c_{11}^{2} 
, c_{11}'^{1}=c_{11}^{1} (2 \alpha -\beta )
\end{array} \right\}$\\
\hline

$\begin{array}{cccc}
     \mathrm{A}_{14}^{\alpha} &\not \to&  \mathrm{A}_{04}, \mathrm{A}_{06}^{-1}, \mathrm{A}_{08}^{-2}, \mathrm{A}_{18}^{\beta\neq\alpha}  & 
\end{array}$ &
${\mathcal R}= \left\{ \begin{array}{l} c_{11}^{1}, c_{11}^{2}, c_{11}'^{1}, c_{11}'^{2} \in \mathbb{C}, c_{11}'^{1}=\alpha  c_{11}^{1},  c_{11}'^{2}=\alpha  c_{11}^{2}
\end{array} \right\}$\\
\hline

$\begin{array}{cccc}
     \mathrm{A}_{16}^{\alpha} &\not \to&  \mathrm{A}_{04}, \mathrm{A}_{06}^{-1}, \mathrm{A}_{18}^{\beta\neq\alpha} & 
\end{array}$ &
${\mathcal R}= \left\{ \begin{array}{l} c_{11}^{2}, c_{11}'^{1}, c_{11}'^{2}, c_{12}'^{2}, c_{21}'^{2} \in \mathbb{C},\\ c_{11}'^{1}=\frac{3}{2} c_{12}'^{2}, \\  c_{21}'^{2}=\frac{1}{2} c_{12}'^{2},  \alpha  c_{11}^{2}=c_{11}'^{2}
\end{array} \right\}$\\
\hline

$\begin{array}{cccc}
     \mathrm{A}_{10}^{\alpha, \beta} &\not \to&  \Big\{ \begin{array}{l} \mathrm{A}_{02}, \mathrm{A}_{03}, \mathrm{A}_{06}^{\gamma}, 
     \mathrm{A}_{08}^{\gamma}, 
     \\ \mathrm{A}_{12}^{\gamma,\delta}[\text{Not}(\alpha=\beta=\gamma=\delta)], \\ \mathrm{A}_{14}^{\gamma}[\alpha\neq\gamma \textrm{ and } \beta\neq\gamma] 
     \end{array}\Big\}& 
\end{array}$ &
${\mathcal R}= \left\{ \begin{array}{l} 
c_{11}^{1}, c_{ij}^2, 
c_{11}'^{1},  c_{ij}'^{2}
\in \mathbb{C} \mbox{ for } 1\leq i,j \leq 2\\
\alpha  c_{11}^{1}=c_{11}'^{1},  c_{12}'^{2}=c_{21}'^{2},  
\beta  c_{22}^{2}=c_{22}'^{2},  \\ c_{12}^{2}=c_{21}^{2},  \beta  c_{12}^{2}=c_{12}'^{2}
\end{array} \right\}$\\
\hline

$\begin{array}{cccc}
     \mathrm{A}_{10}^{\alpha, \alpha} &\not \to& 
     \Big\{ \begin{array}{l} 
     \mathrm{A}_{04}, 
     \mathrm{A}_{06}^{-1},
     \mathrm{A}_{08}^{-2}, \mathrm{A}_{18}^{\beta\neq\alpha}, \\
     \mathrm{A}_{12}^{\beta,\gamma}[\text{Not}(\alpha=\beta=\gamma)] 
      \end{array}\Big\} & 
\end{array}$ &
${\mathcal R}= \left\{ \begin{array}{l} 
c_{11}^{1}, c_{ij}^2, 
c_{11}'^{1},  c_{ij}'^{2}
\in \mathbb{C} \mbox{ for } 1\leq i,j \leq 2\\ \alpha  c_{11}^{1}=c_{11}'^{1},  \alpha  c_{11}^{2}=c_{11}'^{2},  \alpha  c_{12}^{2}=c_{12}'^{2},  \\ \alpha  c_{21}^{2}=c_{21}'^{2},  \alpha  c_{22}^{2}=c_{22}'^{2},  c_{12}^{2}=c_{21}^{2}
\end{array} \right\}$\\
\hline

$\begin{array}{cccc}
     \mathrm{A}_{11}^{\alpha, \beta} &\not \to& 
     \Big\{ \begin{array}{l} \mathrm{A}_{02}, \mathrm{A}_{03}, \mathrm{A}_{06}^{\gamma},  \mathrm{A}_{14}^{\gamma}, 
     \mathrm{A}_{08}^{\gamma}, \\ 
     \mathrm{A}_{12}^{\gamma,\delta}[\text{Not} (\alpha =\gamma \textrm{ and } \beta =\delta)] 
     \end{array}\Big\}
     & 
\end{array}$ &
${\mathcal R}= \left\{ \begin{array}{l} c_{11}^{1}, c_{11}^{2}, c_{12}^{2}, c_{21}^{2}, c_{11}'^{1}, c_{11}'^{2}, c_{12}'^{2}, c_{21}'^{2}  \in \mathbb{C}, \\ c_{11}^{1}=c_{12}^{2},  c_{21}^{2}=c_{11}^{1},  c_{11}'^{1}=c_{11}^{1} (2 \alpha -\beta ), \\ c_{12}'^{2}=\alpha  c_{11}^{1},  c_{21}'^{2}=\beta  c_{11}^{1}
\end{array} \right\}$\\
\hline

$\begin{array}{cccc}
     \mathrm{A}_{13}^{\alpha} &\not \to&  \mathrm{A}_{02}, 
     \mathrm{A}_{06}^{\beta}, 
     \mathrm{A}_{08}^{\beta}, 
     \mathrm{A}_{12}^{\beta,\gamma}, \mathrm{A}_{14}^{\beta\neq\alpha} 
     & 
\end{array}$ &
${\mathcal R}= \left\{ \begin{array}{l} c_{11}^{1}, c_{11}^{2}, c_{11}'^{1}, c_{11}'^{2}, c_{12}'^{2}, c_{21}'^{2}, c_{22}'^{2} \in \mathbb{C},\\ 
c_{12}'^{2}=c_{21}'^{2},  \alpha  c_{11}^{1}=c_{11}'^{1}
\end{array} \right\}$\\
\hline

$\begin{array}{cccc}
     \mathrm{A}_{15}^{\alpha\neq\frac{1}{2}} &\not \to& \Big\{ \begin{array}{l}  \mathrm{A}_{02}[\alpha\neq1], \mathrm{A}_{03}, \\ \mathrm{A}_{06}^{\beta}[\alpha\neq0 \textrm{ or } \beta \neq -2], \\ 
     \mathrm{A}_{08}^{\beta\neq\frac{1}{\alpha-1}}, \mathrm{A}_{08}^{0}, 
     \mathrm{A}_{12}^{\beta,\gamma}, \mathrm{A}_{14}^{\beta} 
     \end{array}\Big\}
     & 
\end{array}$ &
${\mathcal R}= \left\{ \begin{array}{l} c_{11}^{2}, c_{11}'^{1}, c_{11}'^{2}, c_{12}'^{2}, c_{21}'^{2} \in \mathbb{C},\\ 
c_{11}'^{1}=(2-\alpha ) c_{12}'^{2},   c_{21}'^{2}=\alpha  c_{12}'^{2}
\end{array} \right\}$\\
\hline

$\begin{array}{cccc}
     \mathrm{A}_{17} &\not \to&  \mathrm{A}_{02}, \mathrm{A}_{03}, \mathrm{A}_{06}^{\alpha}, 
     \mathrm{A}_{08}^{\alpha\neq0}, 
     \mathrm{A}_{12}^{\alpha,\beta}, \mathrm{A}_{14}^{\alpha} 
     & 
\end{array}$ &
${\mathcal R}= \left\{ \begin{array}{l} c_{11}^{2}, c_{11}'^{1}, c_{11}'^{2}, c_{21}'^{2} \in \mathbb{C}, c_{21}'^{2}=-c_{11}'^{1}
\end{array} \right\}$\\
\hline

    \end{longtable}
    
    }

\end{proof}

At this point, only the description of the closures of the orbits of the parametric families is missing. Although it is not necessary to study the closure of the orbits of each of the parametric families of
the variety of $2$-dimensional anti-pre-Lie Poisson algebras in order to identify its irreducible components, we will study them to give a complete description of the variety.

\begin{lemma}\label{th:degaprebf}
The description of the closure of the orbit of the parametric families in
the variety of $2$-dimensional anti-pre-Lie Poisson algebras are given below:

\begin{longtable}{lcl}
$\overline{\{O(\mathrm{A}_{06}^{*})\}}$ & ${\supseteq}$ &
$ \Big\{
\overline{\{O(\mathrm{A}_{03})\}}, \overline{\{O(\mathrm{A}_{04})\}}, \overline{\{O(\mathrm{A}_{07})\}}, \overline{\{O(\mathrm{A}_{08}^{-1})\}},
\overline{\{O({\mathbb C^8})\}} \Big\}$\\

$\overline{\{O(\mathrm{A}_{08}^{*})\}}$ & ${\supseteq}$ &
$ \Big\{
\overline{\{O(\mathrm{A}_{02})\}}, \overline{\{O(\mathrm{A}_{04})\}}, \overline{\{O(\mathrm{A}_{06}^{-2})\}}, \overline{\{O(\mathrm{A}_{09})\}}, 
\overline{\{O({\mathbb C^8})\}} \Big\}$\\

$\overline{\{O(\mathrm{A}_{10}^{*})\}}$ & ${\supseteq}$ &
$ \Big\{\begin{array}{l}
\overline{\{O(\mathrm{A}_{01})\}}, \overline{\{O(\mathrm{A}_{02})\}}, \overline{\{O(\mathrm{A}_{03})\}}, \overline{\{O(\mathrm{A}_{04})\}}, \overline{\{O(\mathrm{A}_{11}^{\alpha,\alpha})\}}, \overline{\{O(\mathrm{A}_{12}^{\alpha,\alpha})\}},\\ \overline{\{O(\mathrm{A}_{13}^{\alpha})\}}, \overline{\{O(\mathrm{A}_{14}^{\alpha})\}},
\overline{\{O(\mathrm{A}_{15}^{1})\}},
\overline{\{O(\mathrm{A}_{18}^{\alpha})\}}, 
\overline{\{O({\mathbb C^8})\}} \end{array}\Big\}$\\

$\overline{\{O(\mathrm{A}_{11}^{*})\}}$ & ${\supseteq}$ &
$ \Big\{\begin{array}{l}
\overline{\{O(\mathrm{A}_{02})\}}, \overline{\{O(\mathrm{A}_{04})\}}, \overline{\{O(\mathrm{A}_{08}^{\alpha})\}}^{1}, \overline{\{O(\mathrm{A}_{09})\}}, \overline{\{O(\mathrm{A}_{12}^{\alpha,\beta})\}},\\ \overline{\{O(\mathrm{A}_{15}^{\alpha})\}}, \overline{\{O(\mathrm{A}_{16}^{\alpha})\}}, \overline{\{O(\mathrm{A}_{17})\}}, \overline{\{O(\mathrm{A}_{18}^{\alpha})\}}, 
\overline{\{O({\mathbb C^8})\}} \end{array}\Big\}$\\

$\overline{\{O(\mathrm{A}_{12}^{*})\}}$ & ${\supseteq}$ &
$ \Big\{\begin{array}{l}
\overline{\{O(\mathrm{A}_{02})\}},
\overline{\{O(\mathrm{A}_{04})\}},
\overline{\{O(\mathrm{A}_{08}^{\alpha})\}}^{1},  \overline{\{O(\mathrm{A}_{09})\}}, \\ \overline{\{O(\mathrm{A}_{16}^{\alpha})\}}, \overline{\{O(\mathrm{A}_{18}^{\alpha})\}}, 
\overline{\{O({\mathbb C^8})\}} \end{array}\Big\}$\\

$\overline{\{O(\mathrm{A}_{13}^{*})\}}$ & ${\supseteq}$ &
$ \Big\{\begin{array}{l}
\overline{\{O(\mathrm{A}_{01})\}}, \overline{\{O(\mathrm{A}_{02})\}}, \overline{\{O(\mathrm{A}_{03})\}}, \overline{\{O(\mathrm{A}_{04})\}}, \\ \overline{\{O(\mathrm{A}_{14}^{\alpha})\}}, \overline{\{O(\mathrm{A}_{15}^{1})\}}, \overline{\{O(\mathrm{A}_{18}^{\alpha})\}}, 
\overline{\{O({\mathbb C^8})\}} \end{array}\Big\}$\\

$\overline{\{O(\mathrm{A}_{14}^{*})\}}$ & ${\supseteq}$ &
$ \Big\{
\overline{\{O(\mathrm{A}_{03})\}}, \overline{\{O(\mathrm{A}_{04})\}}, \overline{\{O(\mathrm{A}_{18}^{\alpha})\}}, 
\overline{\{O({\mathbb C^8})\}} \Big\}$\\

$\overline{\{O(\mathrm{A}_{15}^{*})\}}$ & ${\supseteq}$ &
$ \Big\{\begin{array}{l}
\overline{\{O(\mathrm{A}_{02})\}}, \overline{\{O(\mathrm{A}_{04})\}}, \overline{\{O(\mathrm{A}_{08}^{\alpha})\}}^{1}, \overline{\{O(\mathrm{A}_{09})\}}, \\ \overline{\{O(\mathrm{A}_{16}^{\alpha})\}}, \overline{\{O(\mathrm{A}_{17})\}}, \overline{\{O(\mathrm{A}_{18}^{\alpha})\}},
\overline{\{O({\mathbb C^8})\}} \end{array}\Big\}$\\

$\overline{\{O(\mathrm{A}_{16}^{*})\}}$ & ${\supseteq}$ &
$ \Big\{
\overline{\{O(\mathrm{A}_{04})\}}, \overline{\{O(\mathrm{A}_{08}^{-2})\}}, \overline{\{O(\mathrm{A}_{09})\}}, \overline{\{O(\mathrm{A}_{18}^{\alpha})\}}, 
\overline{\{O({\mathbb C^8})\}} \Big\}$\\

$\overline{\{O(\mathrm{A}_{18}^{*})\}}$ & ${\supseteq}$ &
$ \Big\{
\overline{\{O(\mathrm{A}_{04})\}}, 
\overline{\{O({\mathbb C^8})\}} \Big\}$\\

\end{longtable}

$^{1}$ Note that $\mathrm{A}_{08}^{-1}\cong \mathrm{A}_{06}^{-2}$.

\end{lemma}

\begin{proof}
Thanks to Theorem \ref{th:degapre} we have all necessary degenerations between algebras.
All necessarily the closure of the orbits of suitable families are given using the parametric bases and the parametric index included in following Table: 

    \begin{longtable}{|lcl|ll|}
\hline
\multicolumn{3}{|c|}{\textrm{Degeneration}}  & \multicolumn{2}{|c|}{\textrm{Parametrized basis}} \\
\hline
\hline 

$\mathrm{A}_{06}^{t^{-1}} $ & $\to$ & $\mathrm{A}_{03} $ & $
E_{1}(t)= te_2$ & $
E_{2}(t)= te_1$  \\
\hline

$\mathrm{A}_{06}^{t-1} $ & $\to$ & $\mathrm{A}_{07} $ & $
E_{1}(t)= e_1$ & $
E_{2}(t)= -t^{-1}e_1 + e_2$  \\
\hline

$\mathrm{A}_{08}^{t^{-1}} $ & $\to$ & $\mathrm{A}_{02} $ & $
E_{1}(t)= t e_2$ & $
E_{2}(t)= e_1$  \\
\hline

$\mathrm{A}_{08}^{t-2} $ & $\to$ & $\mathrm{A}_{09} $ & $
E_{1}(t)= -te_2$ & $
E_{2}(t)= e_1 + e_2$  \\
\hline

$\mathrm{A}_{10}^{t^{-1},t^{-1}} $ & $\to$ & $\mathrm{A}_{01} $ & $
E_{1}(t)= t e_1$ & $
E_{2}(t)= t e_2$  \\
\hline

$\mathrm{A}_{10}^{\frac{t^{-1}}{2}, t^{-1}} $ & $\to$ & $\mathrm{A}_{15}^{1} $ & $
E_{1}(t)= 2t e_1 + t e_2$ & $
E_{2}(t)= -t^{2} e_2$  \\
\hline

$\mathrm{A}_{10}^{\alpha, t^{-1}} $ & $\to$ & $\mathrm{A}_{13}^{\alpha} $ & $
E_{1}(t)= e_1 $ & $
E_{2}(t)= t e_2$  \\
\hline

$\mathrm{A}_{10}^{\alpha, t + \alpha} $ & $\to$ & $\mathrm{A}_{11}^{\alpha, \alpha} $ & $
E_{1}(t)= e_1 + e_2 $ & $
E_{2}(t)= t e_2$  \\
\hline

$\mathrm{A}_{11}^{0, t^{-1}} $ & $\to$ & $\mathrm{A}_{17} $ & $
E_{1}(t)= te_1 - \frac{t^{2}}{2}e_2 $ & $
E_{2}(t)= -\frac{t^3}{2}e_2$  \\
\hline

$\mathrm{A}_{11}^{-\frac{t^{-1}}{\sqrt{1-2\alpha}}, -\frac{t^{-1}\alpha}{\sqrt{1-2\alpha}}} $ & $\to$ & $\mathrm{A}_{15}^{\alpha\neq\frac{1}{2}} $ & $
E_{1}(t)= - t \sqrt{1-2\alpha} e_1 + t^{2}e_2$ & $
E_{2}(t)= -t^{3}\sqrt{1-2\alpha} e_2$  \\
\hline

$\mathrm{A}_{11}^{t^{-1}, t^{-1}} $ & $\to$ & $\mathrm{A}_{02} $ & $
E_{1}(t)= t e_1 $ & $
E_{2}(t)= e_2$  \\
\hline

$\mathrm{A}_{11}^{-2 t^{-1}, -t^{-1}} $ & $\to$ & $\mathrm{A}_{09} $ & $
E_{1}(t)= t^{2}e_2$ & $
E_{2}(t)= t e_1$  \\
\hline

$\mathrm{A}_{11}^{t^{-1}, \frac{t^{-1}}{2}} $ & $\to$ & $\mathrm{A}_{16}^{\alpha\neq0} $ & $
E_{1}(t)= t e_1 + \frac{t}{\alpha}e_2$ & $
E_{2}(t)= \frac{t^2}{\alpha}e_2$  \\
\hline

$\mathrm{A}_{11}^{t^{-1}, \frac{t^{-1}}{2}} $ & $\to$ & $\mathrm{A}_{16}^{0} $ & $
E_{1}(t)= t e_1 + e_2$ & $
E_{2}(t)= t e_2$  \\
\hline

$\mathrm{A}_{12}^{t^{-1},t^{-1}} $ & $\to$ & $\mathrm{A}_{02} $ & $
E_{1}(t)= te_1 $ & $
E_{2}(t)= e_2$  \\
\hline

$\mathrm{A}_{12}^{-2t^{-2}, -(1+t)t^{-2}} $ & $\to$ & $\mathrm{A}_{09} $ & $
E_{1}(t)= t^{3}e_1 - t e_2$ & $
E_{2}(t)= t^{2}e_1 + e_2$  \\
\hline

$\mathrm{A}_{12}^{t^{-1},\frac{(1+t\alpha)t^{-1}}{2}} $ & $\to$ & $\mathrm{A}_{16}^{\alpha} $ & $
E_{1}(t)= t e_1 + e_2$ & $
E_{2}(t)= -\frac{t^2}{2} e_1 + \frac{t}{2}e_2$  \\
\hline

$\mathrm{A}_{12}^{\alpha t^{-1}, (1+\alpha) t^{-1}} $ & $\to$ & $\mathrm{A}_{08}^{\alpha} $ & $
E_{1}(t)= e_2$ & $
E_{2}(t)= t e_1 + t e_2$  \\
\hline

$\mathrm{A}_{13}^{t^{-1}} $ & $\to$ & $\mathrm{A}_{01} $ & $
E_{1}(t)= t e_1$ & $
E_{2}(t)= e_2$  \\
\hline

$\mathrm{A}_{13}^{t^{-1}} $ & $\to$ & $\mathrm{A}_{15}^{1} $ & $
E_{1}(t)= t e_1 + e_2$ & $
E_{2}(t)= -te_2$  \\
\hline

$\mathrm{A}_{14}^{t^{-1}} $ & $\to$ & $\mathrm{A}_{03} $ & $
E_{1}(t)= t e_1$ & $
E_{2}(t)= e_2$  \\
\hline

$\mathrm{A}_{15}^{2+t^{-1}} $ & $\to$ & $\mathrm{A}_{17} $ & $
E_{1}(t)= t e_1$ & $
E_{2}(t)= t^2 e_2$  \\
\hline

$\mathrm{A}_{15}^{\frac{1}{2}+t} $ & $\to$ & $\mathrm{A}_{09} $ & $
E_{1}(t)= t^{-1}e_2$ & $
E_{2}(t)= -2 e_1 - \frac{t^{-2}}{4}e_2$  \\
\hline

$\mathrm{A}_{15}^{2-\frac{3}{2(1 + t)}} $ & $\to$ & $\mathrm{A}_{16}^{\alpha} $ & $
E_{1}(t)= (1+t)e_1 + \frac{(1+t)^2t^{-1}\alpha}{3}e_2$ & $
E_{2}(t)= (1+t)^2e_2$  \\
\hline

$\mathrm{A}_{15}^{t^{-1}} $ & $\to$ & $\mathrm{A}_{08}^0{}$ & $
E_{1}(t)= e_2$ & $
E_{2}(t)= t e_1$  \\
\hline

$\mathrm{A}_{16}^{t^{-1}} $ & $\to$ & $\mathrm{A}_{09} $ & $
E_{1}(t)= 4t^{-1}e_2$ & $
E_{2}(t)= -2e_1$  \\
\hline

$\mathrm{A}_{18}^{t^{-2}} $ & $\to$ & $\mathrm{A}_{04} $ & $
E_{1}(t)= t e_1$ & $
E_{2}(t)= e_2$  \\
\hline



    \end{longtable}

    \begin{longtable}{|l|l|}
\hline
\multicolumn{1}{|c|}{\textrm{Non-degeneration}} & \multicolumn{1}{|c|}{\textrm{Arguments}}\\
\hline
\hline

$\begin{array}{cccc}
     \mathrm{A}_{18}^{*} &\not \to&  \mathrm{A}_{06}^{\alpha}, \mathrm{A}_{08}^{\alpha}  & 
\end{array}$ &
${\mathcal R}= \left\{ \begin{array}{l} c_{11}^{2}, c_{11}'^{2} \in \mathbb{C},  c_{ij}^{k}=c_{ij}'^{k}=0 \textrm{ otherwise} 
\end{array} \right\}$\\
\hline

$\begin{array}{cccc}
     \mathrm{A}_{06}^{*} &\not \to&  \mathrm{A}_{02}, \mathrm{A}_{08}^{\alpha\neq-1}
\end{array}$ &
${\mathcal R}= \left\{ \begin{array}{l} c_{11}'^{1}, c_{11}'^{2}, c_{12}'^{2}, c_{21}'^{1}, c_{21}'^{2}, c_{22}'^{2} \in \mathbb{C},
c_{12}'^{2} c_{21}'^{1}=c_{11}'^{1} c_{22}'^{2}, \\ c_{12}'^{2} c_{21}'^{2}=c_{11}'^{2} c_{22}'^{2},  c_{11}'^{2} c_{21}'^{1}=c_{11}'^{1} c_{21}'^{2}
\end{array} \right\}$\\
\hline

$\begin{array}{cccc}
     \mathrm{A}_{08}^{*} &\not \to&  \mathrm{A}_{03}, \mathrm{A}_{06}^{\alpha\neq-2}
\end{array}$ &
${\mathcal R}= \left\{ \begin{array}{l} c_{11}'^{1}, c_{11}'^{2}, c_{12}'^{1}, c_{12}'^{2}, c_{21}'^{1}, c_{21}'^{2}, c_{22}'^{2}  \in \mathbb{C},\\ 
c_{12}'^{1}+c_{22}'^{2}=2 c_{21}'^{1},  c_{12}'^{2}=2 c_{21}'^{2},\\  (c_{11}'^{1})^2-9 (c_{21}'^{2})^2+4 c_{11}'^{2} c_{22}'^{2}=0, \\ c_{11}'^{1} c_{12}'^{1}+6 c_{21}'^{1} c_{21}'^{2}-c_{11}'^{1} c_{22}'^{2}-2 c_{21}'^{2} c_{22}'^{2}=0
\end{array} \right\}$\\
\hline

$\begin{array}{cccc}
     \mathrm{A}_{14}^{*} &\not \to&  \mathrm{A}_{02}, \mathrm{A}_{06}^{\alpha}, 
     \mathrm{A}_{08}^{\alpha}, 
     \mathrm{A}_{12}^{\alpha,\beta} 
     & 
\end{array}$ &
${\mathcal R}= \left\{ \begin{array}{l} c_{11}^{1}, c_{11}^{2}, c_{11}'^{1}, c_{11}'^{2} \in \mathbb{C},
c_{11}^{2} c_{11}'^{1}=c_{11}^{1} c_{11}'^{2}
\end{array} \right\}$\\
\hline

$\begin{array}{cccc}
     \mathrm{A}_{16}^{*} &\not \to&  
     \begin{array}{l}\mathrm{A}_{02}, \mathrm{A}_{03}, \mathrm{A}_{06}^{\alpha},\\ 
     \mathrm{A}_{08}^{\alpha\neq-2}, \mathrm{A}_{12}^{\alpha,\beta}, \mathrm{A}_{14}^{\alpha}
     \end{array}& 
\end{array}$ &
${\mathcal R}= \left\{ \begin{array}{l} c_{11}^{2}, c_{11}'^{1}, c_{11}'^{2}, c_{12}'^{2}, c_{21}'^{2} \in \mathbb{C},
2 c_{11}'^{1}=3 c_{12}'^{2},  2 c_{21}'^{2}=c_{12}'^{2}
\end{array} \right\}$\\
\hline

$\begin{array}{cccc}
     \mathrm{A}_{12}^{*} &\not \to& \begin{array}{l} 
     \mathrm{A}_{03}, 
     \mathrm{A}_{06}^{\alpha\neq-2}, 
     \mathrm{A}_{11}^{\alpha,\beta},\\ 
     \mathrm{A}_{14}^{\alpha}, \mathrm{A}_{15}^{\alpha\neq\frac{1}{2}}, \mathrm{A}_{17}\end{array} & 
\end{array}$ &
${\mathcal R}= \left\{ \begin{array}{l} c_{11}^{1}, c_{11}^{2}, c_{12}^{2}, c_{21}^{2}, c_{11}'^{1}, c_{11}'^{2}, c_{12}'^{2}, c_{21}'^{2} \in \mathbb{C},\\ 
c_{12}^{2}=c_{11}^{1}, c_{21}^{2}=c_{11}^{1},    c_{11}^{2} (c_{12}'^{2}-2 c_{21}'^{2})=-c_{11}^{1} c_{11}'^{2},  \\  
2 c_{12}'^{2}-c_{21}'^{2}=c_{11}'^{1}
\end{array} \right\}$\\
\hline

$\begin{array}{cccc}
     \mathrm{A}_{13}^{*} &\not \to&  
     \mathrm{A}_{06}^{\alpha}, 
     \mathrm{A}_{08}^{\alpha}, 
     \mathrm{A}_{10}^{\alpha,\beta}, 
     \mathrm{A}_{12}^{\alpha,\beta} 
     & 
\end{array}$ &
${\mathcal R}= \left\{ \begin{array}{l} c_{11}^{1}, c_{11}^{2}, c_{11}'^{1}, c_{11}'^{2}, c_{12}'^{2}, c_{21}'^{2}, c_{22}'^{2} \in \mathbb{C},
c_{12}'^{2}=c_{21}'^{2}
\end{array} \right\}$\\
\hline

$\begin{array}{cccc}
     \mathrm{A}_{15}^{*} &\not \to&  
     \mathrm{A}_{03}, 
     \mathrm{A}_{06}^{\alpha\neq-2}, 
     \mathrm{A}_{12}^{\alpha,\beta}, 
     \mathrm{A}_{14}^{\alpha} & 
\end{array}$ &
${\mathcal R}= \left\{ \begin{array}{l} c_{11}^{2}, c_{11}'^{1}, c_{11}'^{2}, c_{12}'^{2}, c_{21}'^{2} \in \mathbb{C}, 
2 c_{12}'^{2}-c_{21}'^{2}=c_{11}'^{1}
\end{array} \right\}$\\
\hline

$\begin{array}{cccc}
     \mathrm{A}_{10}^{*} &\not \to&  
     \mathrm{A}_{06}^{\alpha}, 
     \mathrm{A}_{08}^{\alpha}, 
     \mathrm{A}_{12}^{\alpha, \beta\neq\alpha} 
     & 
\end{array}$ &
${\mathcal R}= \left\{ \begin{array}{l} c_{11}^{1}, c_{11}^{2}, c_{12}^{2}, c_{21}^{2}, c_{22}^{2}, c_{11}'^{1}, c_{11}'^{2}, c_{12}'^{2}, c_{21}'^{2}, c_{22}'^{2} \in \mathbb{C}, \\
c_{12}^{2}=c_{21}^{2},  c_{12}'^{2}=c_{21}'^{2},  c_{22}^{2} c_{12}'^{2}=c_{12}^{2} c_{22}'^{2}
\end{array} \right\}$\\
\hline

$\begin{array}{cccc}
     \mathrm{A}_{11}^{*} &\not \to&  
     \mathrm{A}_{03}, 
     \mathrm{A}_{06}^{\alpha\neq-2}, 
     \mathrm{A}_{14}^{\alpha} & 
\end{array}$ &
${\mathcal R}= \left\{ \begin{array}{l} c_{11}^{1}, c_{11}^{2}, c_{12}^{2}, c_{21}^{2}, c_{11}'^{1}, c_{11}'^{2}, c_{12}'^{2}, c_{21}'^{2} \in \mathbb{C},\\ 
c_{11}^{1}=c_{12}^{2},  c_{21}^{2}=c_{11}^{1},  2 c_{12}'^{2}-c_{21}'^{2}=c_{11}'^{1}
\end{array} \right\}$\\
\hline

    \end{longtable}

\end{proof}

By Lemma \ref{th:degapre} and Lemma \ref{th:degaprebf}, it follows the geometric classification.

\begin{theorem}
The variety of $2$-dimensional anti-pre-Lie Poisson algebras has three irreducible components, corresponding to the anti-pre-Lie Poisson algebra $\mathrm{A}_{05}$ and the families of anti-pre-Lie Poisson algebras $\mathrm{A}_{10}^{*}$ and $\mathrm{A}_{11}^{*}$.
\end{theorem}

\section{ Pre-Poisson  algebras}

Pre-Poisson algebras first appeared in \cite{A00}
and have many applications in other areas of mathematics (see, for example, the citations of \cite{A00}).
The operad of pre-Poisson algebras is isomorphic to the Manin black product of the Poisson operad with the preLie operad \cite{U}.

\subsection{The algebraic classification of $2$-dimensional Zinbiel  algebras}

\begin{definition}
An algebra  is called a (left) Zinbiel  algebra if it satisfies the identity 
\begin{center}
$x\cdot  \left(    y \cdot    z\right)= 
\left( y\cdot    x + x\cdot    y\right) \cdot    z.$
\end{center}%
 \end{definition}

The algebraic classification of  $2$-dimensional Zinbiel algebras can be obtained by direct verification from \cite{kv16}.

\begin{theorem}
\label{zinb2}
Let ${\rm A}$ be a nonzero $2$-dimensional Zinbiel algebra. 
Then ${\rm A}$  is
isomorphic to  the following algebra:

\begin{longtable}{lcllll}

${\rm Z}_{1}$&$:$&$ e_{1} \cdot   e_{1}  = e_{2}.$ 
 
\end{longtable}

\end{theorem}

\subsection{The algebraic classification of $2$-dimensional pre-Poisson algebras}

\begin{definition}
A pre-Poisson  algebra is a vector space  equipped with 
a (left) Zinbiel multiplication $\cdot$
and     
a (left) Pre-Lie multiplication $\circ  .$
These two operations are required to satisfy the following identities:
 
\begin{longtable}{rcl}

$(x\circ y - y \circ x ) \cdot z $&$=$&$ x \circ (y \cdot z) -y \cdot (x \circ z);$\\

$(x\cdot y + y \cdot x ) \circ z $&$=$&$ x \cdot (y \circ z) +y \cdot (x \circ z).$
\end{longtable}

\end{definition}

Pre-Poisson   algebras 
$({\rm P}, \cdot, \circ  )$ with zero Zinbiel multiplication  are 
isomorphic to 
pre-Lie algebras given in Theorem \ref{pre2}. 
Hence, we are studying  pre-Poisson  algebras defined on
the non-trivial Zinbiel algebra from Theorem \ref{zinb2}.

\begin{definition}
Let $\left( \rm{A},\cdot \right) $ be a Zinbiel 
algebra. 
Define ${\rm Z}_{\rm PP}^{2}\left( \rm{A},\rm{A}\right) $ to be
the set of all bilinear maps $\theta :\rm{A}\times \rm{A}%
\longrightarrow \rm{A}$ such that:%
\begin{longtable}{rcl}
$\theta \left( \theta \left( x,y\right) ,z\right) -\theta \left( x,\theta \left( y,z\right)
\right)$ & $=$&$\theta \left( \theta
\left( y,x\right) ,z\right) -\theta \left( y,\theta \left( x,z\right) \right),$ \\

$(\theta (x, y) - \theta (y, x) ) \cdot z $&$=$&$ \theta (x , y \cdot z) -y \cdot \theta (x , z),$\\

$\theta (x\cdot y + y \cdot x ,  z) $&$=$&$ x \cdot \theta (y , z) +y \cdot \theta (x , z).$
\end{longtable}

If $\theta \in {\rm Z}_{\rm PP}^{2}\left( \rm{A},%
\rm{A}\right) $, then $\left( \rm{A},\cdot ,\circ   \right) $ is a
 pre-Poisson  algebra where $x\circ   y=\theta \left( x,y\right) $ for
all $x,y\in \rm{A}$.
\end{definition}

\subsubsection{Pre-Poisson  algebras defined on  ${\rm Z}_{1}$}
 
From the computation of $\mathrm{Z}_{\mathrm{PP}}^{2}(\mathrm{Z}_{1},\mathrm{%
Z}_{1})$, the pre-Poisson structures defined on $\mathrm{Z}_{1}$ are of the
form:%
\[
\left\{ 
\begin{tabular}{lcllcllcl}
$e_{1}\cdot e_{1}$&$=$&$e_{2},$ &  &  \\ 
$e_{1}\circ e_{1}$&$=$&$\beta e_{1}+\alpha e_{2},$ & $e_{1}\circ e_{2}$&$=$&$\beta e_{2},
$ & $e_{2}\circ e_{1}$&$=$&$\beta e_{2}.$%
\end{tabular}%
\right. 
\]%
If $\beta =0$, we obtain the family:%
\[
\mathrm{P}_{09}^{\alpha }:=\left\{ 
\begin{tabular}{lcl}
$e_{1}\cdot e_{1}$&$=$&$e_{2},$ \\ 
$e_{1}\circ e_{1}$&$=$&$\alpha e_{2}.$%
\end{tabular}%
\right. 
\]%
If $%
\beta \neq 0$, we get the algebra:%
\[
\mathrm{P}_{10}:=\left\{ 
\begin{tabular}{lcllcllcl}
$e_{1}\cdot e_{1}$&$=$&$e_{2},$ &  &  \\ 
$e_{1}\circ e_{1}$&$=$&$e_{1},$ & $e_{1}\circ e_{2}$&$=$&$e_{2},$ & $e_{2}\circ
e_{1}$&$=$&$e_{2}.$%
\end{tabular}%
\right. 
\]%

Moreover, we have 
$\mathrm{P}_{09}^{\alpha }\cong \mathrm{P}_{09}^{\beta }$
if and only if $\alpha =\beta $.

\black

\begin{theorem}
\label{2dimPP} 
Let $\left( {\rm P},\cdot , \circ   \right) $ be a nonzero $2$-dimensional  pre-Poisson  algebra. Then ${\rm P}$ is isomorphic to one
pre-Lie algebra listed in Theorem \ref{pre2} or to one algebra listed below:     
\begin{longtable}{lcl}
$\mathrm{P}_{09}^{\alpha }$&$:$&$\left\{ 
\begin{tabular}{lcl}
$e_{1}\cdot e_{1}$  & $ =$  &$e_{2},$  \\ 
$e_{1}\circ e_{1}$ & $= $& $\alpha e_{2}.$ %
\end{tabular} \right.$ \\

$\rm{P}_{10}$& $:$ &$\left\{ 
\begin{tabular}{lcllcllcl}
$e_{1}\cdot e_{1}$&$=$&$e_{2},$ &  &  \\ 
$e_{1}\circ e_{1}$&$=$&$e_{1},$ & $e_{1}\circ e_{2} $&$=$&$e_{2},$ & $e_{2}\circ
e_{1}$&$=$&$e_{2}.$%
\end{tabular} \right.$

\end{longtable}
\end{theorem}

\subsection{Degeneration  of $2$-dimensional pre-Poisson algebras}

    \black

\begin{theorem}

\label{th:degpp}
The graph of primary degenerations and non-degenerations of the variety of $2$-dimensional  pre-Poisson algebras is given in Figure 3, where the numbers on the right side are the dimensions of the corresponding orbits.

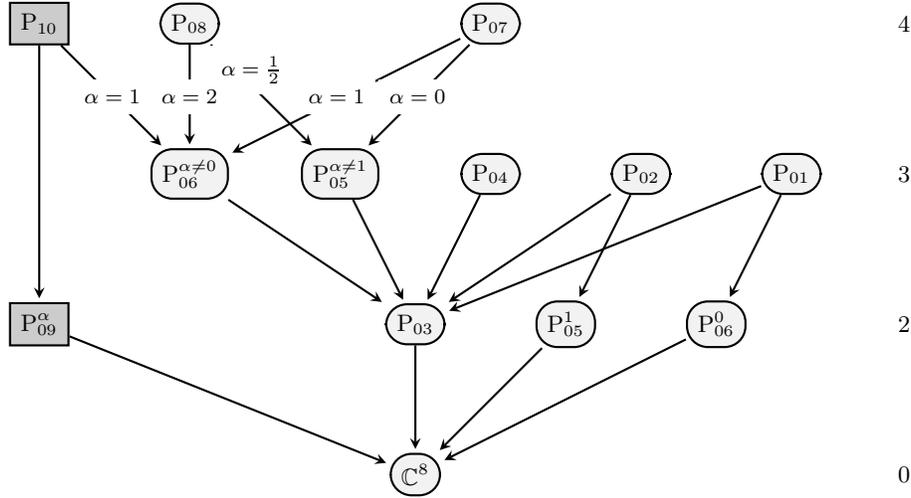
\begin{center}
	
	\begin{tikzpicture}[->,>=stealth,shorten >=0.05cm,auto,node distance=1.3cm,
	thick,
	main node/.style={rectangle,draw,fill=gray!10,rounded corners=1.5ex,font=\sffamily \scriptsize \bfseries },
	rigid node/.style={rectangle,draw,fill=black!20,rounded corners=0ex,font=\sffamily \scriptsize \bfseries }, 
	poisson node/.style={rectangle,draw,fill=black!20,rounded corners=0ex,font=\sffamily \scriptsize \bfseries },
	ac node/.style={rectangle,draw,fill=black!20,rounded corners=0ex,font=\sffamily \scriptsize \bfseries },
	lie node/.style={rectangle,draw,fill=black!20,rounded corners=0ex,font=\sffamily \scriptsize \bfseries },
	style={draw,font=\sffamily \scriptsize \bfseries }]

	\node (3) at (2.5,6) {$4$};
	\node (2) at (2.5,4) {$3$};
	\node (1) at (2.5,2) {$2$};
	\node (0)  at (2.5,0) {$0$};

    \node[main node] (c20) at (-4,0) {${\mathbb C^8}$};
    
    \node[main node] (c260) at (-0,2) {${\rm P}_{06}^{0}$};
    \node[main node] (c251) at (-2,2) {${\rm P}_{05}^{1}$};
    \node[main node] (c23) at (-4,2) {${\rm P}_{03} $};
    \node[ac node] (c2170) at (-9,2) {${\rm P}_{09}^{\alpha}$};
    
    \node[main node] (c21) at (1,4) {${\rm P}_{01}$};
    \node[main node] (c22) at (-1,4) {${\rm P}_{02}$};
    \node[main node] (c24) at (-3,4) {${\rm P}_{04}$};
    \node[main node] (c25a) at (-5,4) {${\rm P}_{05}^{\alpha\neq1}$};
    \node[main node] (c26a) at (-7,4) {${\rm P}_{06}^{\alpha\neq0}$};
    
    \node[main node] (c27) at (-3,6) {${\rm P}_{07}$};
    \node[main node] (c28) at (-7,6) {${\rm P}_{08}$};
    \node[ac node] (c209) at (-9,6) {${\rm P}_{10}^{}$};
    
	\path[every node/.style={font=\sffamily\small}]

     (c28) edge [bend left=0] node[above=10, right=-20, fill=white]{\tiny $\alpha=\frac{1}{2}$}  (c25a)   
    
    (c27) edge [bend left=0] node[above=0, right=-18, fill=white]{\tiny $\alpha=1$}  (c26a)

    (c27) edge [bend left=0] node[above=-0, right=-15, fill=white]{\tiny $\alpha=0$}  (c25a)

    (c28) edge [bend left=0] node[above=0, right=-15, fill=white]{\tiny $\alpha=2$}  (c26a)

    (c209) edge [bend left=0] node[above=0, right=-15, fill=white]{\tiny $\alpha=1$}  (c26a)

    (c22) edge [bend left=0] node{}  (c251)
    (c21) edge [bend left=0] node{}  (c260)
    (c26a) edge [bend left=0] node{}  (c23)
    (c25a) edge [bend left=0] node{}  (c23)
    (c24) edge [bend left=0] node{}  (c23)
    (c22) edge [bend left=0] node{}  (c23)
    (c21) edge [bend left=0] node{}  (c23)
    
    (c2170) edge [bend left=0] node{}  (c20)
    (c260) edge [bend left=0] node{}  (c20)
    (c251) edge [bend left=0] node{}  (c20)
    (c23) edge [bend left=0] node{}  (c20)

    (c209) edge [bend left=0] node{}  (c2170)

 
    ;

	\end{tikzpicture}

{\tiny 
\begin{itemize}
\noindent Legend:
\begin{itemize}
    \item[--] Round nodes: pre-Poisson algebras with trivial Zinbiel multiplication $\cdot$.
    \item[--] Squared nodes:  pre-Poisson algebras with non-trivial Zinbiel multiplication $\cdot$.

\end{itemize}
\end{itemize}}

{Figure 3.}  Graph of primary degenerations and non-degenerations.	
\end{center}
\end{theorem}

\begin{proof} 
The dimensions of the orbits are deduced by computing the algebra of derivations.
The primary degenerations are proven using the parametric bases included in the following Table:

    \begin{longtable}{|lcl|ll|}
\hline
\multicolumn{3}{|c|}{\textrm{Degeneration}}  & \multicolumn{2}{|c|}{\textrm{Parametrized basis}} \\
\hline
\hline

${\rm P}_{10} $ & $\to$ & ${\rm P}_{09}^{\alpha}  $ & $
E_{1}(t)= te_1+\alpha t e_2$ & $
E_{2}(t)= t^2 e_2$  \\
\hline

${\rm P}_{10} $ & $\to$ & ${\rm P}_{06}^1  $ & $
E_{1}(t)= e_1$ & $
E_{2}(t)= t^{-1} e_2$  \\
\hline

\end{longtable}

The primary non-degenerations are following from  Theorem   \ref{th:degcpl} and the following observations:
$({\rm P}_{10}, \circ)$ is commutative, 
but 
$
({\rm P}_{04}, \circ),
({\rm P}_{05}^{\alpha}, \circ)$ and 
$({\rm P}_{06}^{\alpha\neq1}, \circ)$
are non-commutative.

\end{proof}

 At this point, only the description of the closures of the orbits of the parametric families is missing. Although it is not necessary to study the closure of the orbits of each of the parametric families of
the variety of $2$-dimensional   pre-Poisson algebras in order to identify its irreducible components, we will study them to give a complete description of the variety.
The next corollary follows from 
Corollary \ref{th:degcplf} and Theorem \ref{th:degpp}.
 
\begin{corollary}\label{th:degppf}
The description of the closure of the orbit of the parametric families in the variety of $2$-dimensional   pre-Poisson algebras.

\begin{longtable}{lcl}
  $\overline{\{O({\rm P}_{05}^{*}) \}}$ & ${\supseteq}$ &
 $ \Big\{\overline{\{O({\rm P}_{04})\}}, \  
 \overline{\{O({\rm P}_{03})\}}, \  
 \overline{\{O({\rm P}_{02})\}}, \  
 \overline{\{O({\mathbb C^8})\}} \Big\}$\\
   
   $\overline{\{O({\rm P}_{06}^{*}) \}}$ & ${\supseteq}$ & 
   $ \Big\{\overline{\{O({\rm P}_{04})\}}, \  
   \overline{\{O({\rm P}_{03})\}}, \ 
   \overline{\{O({\rm P}_{01})\}}, \ 
   \overline{\{O({\mathbb C^8})\}} \Big\}$\\
  
$\overline{\{O({\rm P}_{09}^{*}) \}}$ & ${\supseteq}$ &
$ \Big\{ 
{\overline{\{O({\rm P}_{03})\}}, \{O({\mathbb C^8})\}}  \Big\}$\\

\end{longtable}

\end{corollary}

The geometric classification of $2$-dimensional   pre-Poisson algebras follows by 
Theorem  \ref{th:degpp} and Corollary \ref{th:degppf}.

\begin{theorem}
The variety of $2$-dimensional   pre-Poisson algebras has five irreducible components, corresponding to the rigid algebras 
${\rm P}_{07}$, ${\rm P}_{08}$ and  ${\rm P}_{10}$,   
and the families of  algebras ${\rm P}_{05}^{*}$ and ${\rm P}_{06}^{*}$.
\end{theorem}

\end{document}